\documentclass[11pt,a4paper,oneside]{amsart}
\usepackage{graphicx,amscd,amssymb}
\usepackage{amsmath}
\usepackage{amssymb}

\usepackage{amsthm}

\theoremstyle{plain}
\newtheorem{theorem}{Theorem}\numberwithin{theorem}{section}
\newtheorem{main}{Main~Theorem}{}
\newtheorem{lemma}{Lemma}\numberwithin{lemma}{section}
\newtheorem{proposition}{Proposition}\numberwithin{proposition}{section}
\numberwithin{corollary}{section}
\theoremstyle{definition}
\newtheorem{definition}{Definition}\numberwithin{definition}{section}
\theoremstyle{remark} \theoremstyle{exam} \theoremstyle{ob}
\newtheorem{remark}{Remark}\numberwithin{remark}{section}
\numberwithin{notation}{section}
\numberwithin{example}{section}
\numberwithin{ob}{section}
\numberwithin{nt}{section}
\numberwithin{equation}{section}

\begin{document}
%\begin{frontmatter}
\title[Maximal Averages on  Variable Hyperplanes]
    {Annulus Maximal Averages   
   on Variable Hyperplanes}
    \author{ Joonil Kim }

    \address{Department of Mathematics\\
            Yonsei University  \\
            Seoul 121, Korea}
   \email{jikim7030@yonsei.ac.kr}

   \keywords{
      2000 Mathematics Subject Classification : 42B15, 42B30.}

 %  \subjclass[2000]{42B15, 42B30}
         %This study was supported (in part ) by the research
         %funds from Chosun University (2004). }
   %\subjclass{Primary: 42B15, 42B30}
   %\date{December 20, 200

\begin{abstract}
By giving   a thin width   $  \delta\ll 1$   to  a unit circle $S^1$, we set an annulus $S^1_{\delta}$  on the Euclidean plane $\mathbb{R}^2$. Consider the    maximal means $M^{\delta}_{S^1}$ over dilations of the annulus $S^1_{\delta}$. It is known that the operator norm of $M^\delta_{S^1}$   on  $L^2(\mathbb{R}^2)$ is $O(|\log 1/\delta|^{1/2})$. In this paper, we study the maximal operator $\mathcal{M}^\delta_{S^1(A)}$   over those annuli    now imbedded on the variable hyperplanes $(x,x_3)+\left\{\left(y, \langle A(x), y\rangle\right): y\in\mathbb{R}^2\right\}\subset \mathbb{R}^{2+1}$  where $A$ is a $2\times 2$ real matrix.
 The model hyperplane is the horizontal plane of the Heisenberg group $\mathbb{H}^1$  when  $A$ is given by the skew--symmetric matrix  $J$.    It turns out that a rank of matrix $JA+(JA)^T$     determines   $\|\mathcal{M}^\delta_{S^1  (A)} \|_{L^2(\mathbb{R}^{2+1})\rightarrow L^2(\mathbb{R}^{2+1})} $.   In the higher dimension $\mathbb{R}^{d+1}$,   the corresponding spherical maximal operator $\mathcal{M}_{S^{d-1}(A)}$ is bounded in $L^p(\mathbb{R}^{d+1})$ for $p>d/(d-1)$ if all eigenvalues of $A$ have nonzero imaginary parts.\end{abstract}
    \maketitle

\setcounter{tocdepth}{1}

\footnote{This research was supported by Basic Science Research Program through the National Research Foundation of Korea(NRF) funded by the Ministry of Education(2019R1H1A2039703)}

\tableofcontents

\section{Introduction} 
With a thickness $0<\delta<1$, we set an annulus $S^1_{\delta}=\{y\in \mathbb{R}^2:1-\delta/2\le |y|\le 1+\delta/2\}$, and  a tube $T_{\delta}=\{y\subset \mathbb{R}^2:|y_1|<1/2, |y_2|<\delta/2\}$. Let $R_\theta$ be a $2\times 2$ matrix of  rotation by an angle $\theta\in [0,2\pi]$. For a locally  integrable function $f\in L^1_{loc}(\mathbb{R}^2)$, we  consider  the   annulus maximal average  of $f$ over the dilations $tS^1_\delta$ of   the annulus $S^1_\delta$ given by
\begin{align}\label{95n}
M^\delta_{S^1} f(x)=\sup_{t>0} \frac{1}{|S^1_{\delta}|}  \int_{y\in S^1_{\delta}}|f(x- ty)|dy 
\end{align}
 and the tube  (Nikodym)  maximal average  of $f$ over the rotations $R_\theta  T_\delta$ of  the tube $T_\delta$,
 $$N^\delta_T f(x)=\sup_{\theta\in [0,2\pi]} \frac{1}{|T_{\delta}|}  \int_{y\in T_{\delta}}|f(x- R_\theta y)|dy.   $$
 The operator norms  of  the two maximal averages     have the same growth rate of $1/\delta$ as
\begin{align*}
\| N^\delta_T \|_{L^2(\mathbb{R}^2)\rightarrow L^2(\mathbb{R}^2) }= O(|\log (1/\delta)|^{1/2}) \ \text{and}\  \| M^\delta_{S^1} \|_{L^2(\mathbb{R}^2)\rightarrow L^2(\mathbb{R}^2) }= O(|\log (1/\delta)|^{1/2}).  
\end{align*}
 The former was obtained by C\'{o}rdoba \cite{Cordoba}  and the latter by Bourgain \cite{B1} and Schlag \cite{Shlag}.
\subsection{Maximal  Average along Variable Hyperplanes in $\mathbb{R}^{3}$} 
To each $(x,x_3)\in\mathbb{R}^{2+1}$ and a $2\times 2$ matrix $A$, we assign the  hyperplanes $ \pi_A(x,x_3)$ in $\mathbb{R}^{2+1}$ given by \begin{align}\label{400m}
 \pi_A(x,x_3):=(x,x_3)-\left\{\left(y,\langle A(x),y\rangle\right): y\in \mathbb{R}^2\right\}.  
\end{align}
We lift the above annuli and   tubes on the plane $\mathbb{R}^2\times\{0\}$ to   the hyperplane $ \pi_A(x,x_3)$.
Associated with each  region,    define the annulus maximal  average of $f\in L^1_{loc}(\mathbb{R}^{2+1})$ as
\begin{align}\label{0g}
\mathcal{M}^{\delta}_{S^1(A)} f(x,x_3)=\sup_{t>0} \frac{1}{|S^1_{\delta}|}  \int_{y\in S^1_{\delta}}|f(x- ty,x_3-\langle A(x), ty\rangle)|dy
\end{align}
and the tube (Nikodym) maximal average of $f\in L^1_{loc}(\mathbb{R}^{2+1})$ as
\begin{align}\label{08g}
\mathcal{N}^{\delta}_{T(A)}  f(x,x_3)=\sup_{\theta\in [0,2\pi]} \frac{1}{|T_{\delta}|}  \int_{y\in T_{\delta}}|f(x- R_\theta y,x_3-\langle A(x), R_\theta y\rangle)|dy.
\end{align}
The main purpose of this paper is to investigate  the operator norms of   $\mathcal{M}^{\delta}_{S^1(A)} $  on $L^2(\mathbb{R}^{2+1})$ according to $2\times 2$ matrices $A$, and revisit the corresponding known classification of  $ \mathcal{N}^{\delta}_{T(A)}  $ for a comparison with  $\mathcal{M}^{\delta}_{S^1(A)} $.  For this purpose, we define a notion of symmetric and skew-symmetric ranks of $2\times 2$ matrices. \begin{definition}\label{df11}
Let $M_{d\times d}(\mathbb{R})$ be the set of $d\times d$ real matrices. In $M_{2\times 2}(\mathbb{R})$,  we select the three matrices as
$$J=\left(\begin{matrix}0&1\\ -1&0\end{matrix}\right),\  I=\left(\begin{matrix}1&0\\0&1\end{matrix}\right) \ \text{and}\    I_c=  \left(\begin{matrix}1&c\\0&1\end{matrix}\right)  \ \text{ for $c\ne 0$ }.$$  Notice that  $J$, called the skew symmetric matrix,  rotates a vector  by the angle $\pi/2$  clockwise. \end{definition}

\begin{definition}[Symmetric  and Skew--Symmetric Ranks]\label{def12}
Let $A\in M_{2\times2}(\mathbb{R})$.   Then, we
 define the symmetric rank and the skew--symmetric rank of $A$ by $$\text{rank}\left(A+A^T\right)\ \ \text{and}\ \    \text{rank}\left(JA+(JA)^T\right).$$
Roughly speaking,   $ \text{rank}\left(JA+(JA)^T\right)$ measures  the extent to which $A$ is close to the skew--symmetric matrix $J$. For example,  $\text{rank}\left(JA+(JA)^T\right)=2$ for $A=J$, while 
$$\text{  $\text{rank}\left(JA+(JA)^T\right)=1$ for $A=I_c$  and  $\text{rank}\left(JA+(JA)^T\right)=0$ for $A=I$.}    $$    \end{definition}
 
Before stating the main result, we introduce a few  notations.
 Given  two scalars $F,G\ge 0$,  we write $F \lesssim G $   if $F \le CG $ for a   constant $C>0$ depending only on  a given matrix $A$ and dimension  $d$. The notation $F \approx G $ indicates $F \lesssim G $ and $G \lesssim F $. 
In particular, given $\delta>0$, we write $\|\cdot\| \lesssim_{\epsilon} \left(\frac{1}{\delta}\right)^{s}$ if  for an arbitrary   small $\epsilon>0$, there exists $C_\epsilon>0$ such that $\|\cdot\|\le C_\epsilon \delta^{-\epsilon}\delta^{-s}$. Moreover,  we write  $\| \cdot\|\approx_{\epsilon}  \left(\frac{1}{\delta}\right)^{s}$ if  for an arbitrary   small $\epsilon>0$, there exists $C_\epsilon>0$ and $C>0$ such that  $ \delta^{-s}/C\le  \|\cdot\| \le C_\epsilon\delta^{-\epsilon}\delta^{-s}  $. 
  
\begin{main}[Annulus  Maximal Function] \label{main1}
Suppose that $A\in M_{2\times 2}(\mathbb{R})$.
\begin{itemize}
\item Let  $\text{rank}(A)=2$. Then it holds that
\begin{itemize}
\item[(1)]   if $\text{rank} (JA+(JA)^T )=2$, then $ \|\mathcal{M}^{\delta}_{S^1(A)}  \|_{L^2(\mathbb{R}^3)\rightarrow L^2(\mathbb{R}^3) }\approx_{\epsilon}   (\frac{1}{\delta} )^0$,
\item[(2)]  if $\text{rank} (JA+(JA)^T )=1$, then  $ \|\mathcal{M}^{\delta}_{S^1(A)} \|_{L^2(\mathbb{R}^3)\rightarrow L^2(\mathbb{R}^3) }\approx_{\epsilon}    (\frac{1}{\delta} )^{1/6},$
\item[(3)] if $\text{rank}  (JA+(JA)^T )=0$, then $  \|\mathcal{M}^{\delta}_{S^1(A)}  \|_{L^2(\mathbb{R}^3)\rightarrow L^2(\mathbb{R}^3) }\approx_{\epsilon}  (\frac{1}{\delta} )^{1/2}$.
\end{itemize}
\item Let  $\text{rank}(A)=1$, then  $ \|\mathcal{M}^{\delta}_{S^1(A)} \|_{L^2(\mathbb{R}^3)\rightarrow L^2(\mathbb{R}^3) }\approx_{\epsilon}  (\frac{1}{\delta} )^0$. 
\item
Let $\text{rank}(A)=0$, then   $ \| \mathcal{M}^{\delta}_{S^1(A)}  \|_{L^2(\mathbb{R}^3)\rightarrow L^2(\mathbb{R}^3) }\approx_{\epsilon}  (\frac{1}{\delta} )^0$ (Euclidean case).
\end{itemize}
\end{main}
  Let $A$ be invertible. Then   Main Theorem \ref{main1} states that    $\mathcal{M}^\delta_{S^1(A)}$ has the best bound   $\approx_{\epsilon}\delta^{-0}$  when $\text{rank}\left(JA+(JA)^T\right)=2$ whereas  $\mathcal{M}^\delta_{S^1(A)}$  has the worst bound $\approx_\epsilon\delta^{-1/2}$ when  $\text{rank}\left(JA+(JA)^T\right)=0$ (occurs exactly when $A=cI$). 
  \subsection{Comparison with Nikodym Maximal functions}
We compare the annulus maximal operator $\mathcal{M}^\delta_{S^1(A)}$ with  the tube maximal operator $\mathcal{N}^\delta_{T(A)}$ of (\ref{08g}) from the author's previous results  of \cite{K1,K2}. 
 \begin{theorem}[Nikodym Maximal Functions in  \cite{K2}]\label{1tt}
Suppose that $A\in M_{2\times 2}(\mathbb{R})$.
\begin{itemize}
\item Let  $\text{rank}(A)=2$. Then it holds that
\begin{itemize}
\item[(1)] if $\text{rank}\left(A+A^T\right)=2$, then $ \| \mathcal{N}^{\delta}_{T(A)}   \|_{L^2(\mathbb{R}^3)\rightarrow L^2(\mathbb{R}^3) }\approx_{\epsilon}  (\frac{1}{\delta} )^0$,
\item[(2)]  if $\text{rank}\left(A+A^T\right)=1$, then  $ \| \mathcal{N}^{\delta}_{T(A)}   \|_{L^2(\mathbb{R}^3)\rightarrow L^2(\mathbb{R}^3) }\approx_{\epsilon}   \ (\frac{1}{\delta}  )^{1/6},$
\item[(3)]  if $\text{rank}\left(A+A^T\right)=0$, then $   \| \mathcal{N}^{\delta}_{T(A)}  \|_{L^2(\mathbb{R}^3)\rightarrow L^2(\mathbb{R}^3) }\approx_{\epsilon}  (\frac{1}{\delta} )^{1/4}$.
\end{itemize}
\item  Let  $\text{rank}(A)=1$. Then  $ \| \mathcal{N}^{\delta}_{T(A)}   \|_{L^2(\mathbb{R}^3)\rightarrow L^2(\mathbb{R}^3) }\approx_{\epsilon}   (\frac{1}{\delta} )^0$.
\item
Let $\text{rank}(A)=0$.  Then   $  \| \mathcal{N}^{\delta}_{T(A)}   \|_{L^2(\mathbb{R}^3)\rightarrow L^2(\mathbb{R}^3) }\approx_{\epsilon}  (\frac{1}{\delta} )^0$ (Euclidean case).
\end{itemize}

\end{theorem}
\begin{remark}
   In \cite{K2},  the above theorem is stated in terms of $D:=(a_{12}+a_{21})^2-4a_{11}a_{22}=\det\left(A+A^T\right) $ and $(a_{11},a_{22})$ if $A=\left(\begin{matrix}a_{11}&a_{12}\\ a_{21}&a_{22}\end{matrix}\right)$ is not symmetric. The symmetric case $a_{12}=a_{21}$ is reduced to the Euclidean case  
 $ \| \mathcal{N}^{\delta}_{T(A)}   \|_{L^2(\mathbb{R}^3)\rightarrow L^2(\mathbb{R}^3) }\approx_\epsilon \delta^{-0} $, which corresponds to    $\rm{rank}(A)\le 1$ in the  theorem \ref{1tt}. In \cite{K2},   the line segments $\{(t,k\delta t):t\in [-1/2,1/2]\}$ for $k=1,\cdots,[1/\delta]$ are treated   rather than the  rectangles $\{R_\theta y:y\in  T_{\delta}\}$ in (\ref{08g}). \end{remark}
Theorem \ref{1tt} states that
$\mathcal{N}^{\delta}_{T(A)} $ has the best bound $\approx_\epsilon \delta^{-0}$ if   $\rm{rank}(A+A^T)=2$, however $\mathcal{N}^{\delta}_{T(A)} $ has the worst bound $\approx_\epsilon \delta^{-1/4}$ when $\rm{rank}(A+A^T)=0$ (occurs exactly when $A=cJ$). 
The novelty of this paper is to demonstrate  that
   the skew--symmetricity measured from $\rm{rank}(JA+(JA)^T)$  or  symmetricity from $\rm{rank}(A+A^T)$   plays a strikingly opposite role  between  annuli and tubes on the variable planes $\pi_{A}(x,x_3)$   in order to classify the operator norms of $\mathcal{M}^\delta_{S^1(A)}$ and $\mathcal{N}^{\delta}_{S^1(A)}$. 
  \\
  \\
\noindent
{\bf Notations}.   Given two vectors $u$ and $v$ in $\mathbb{R}^d$, we write 
$v=u+O(\rho)$ if there exists $C>0$ independent of $u$, $v$, and $\rho$ such that
$
|v-u|\le C\rho.
$
For every  $m\in\mathbb{Z}_+$, we frequently use the  smooth  cutoff functions:
 \begin{itemize}
 \item[(1)] $\psi$ supported in $\{u:|u|\le 1\}\subset\mathbb{R}^m$ with $\psi(u)\equiv 1$ in $|u|<1/2$,
 \item[(2)]  $\chi$  supported in $\{ u:1/2\le |u|\le 2\}\subset \mathbb{R}^m$ 
  \end{itemize}
where we  allow   slight  changes of $\chi$ and $\psi$  line by line.  We denote the phase functions  by $\Phi,\phi$, and  the integral kernels by $K,L,\Psi$, which can be different in cases.   Finally, our positive constants $c$ and $C$ can be also different   line by line.

\subsection{The Rotational Curvature and the Heisenberg group} 
Given $A\in M_{d\times d}(\mathbb{R})$ and $(x,x_{d+1})\in\mathbb{R}^{d+1}$,  set  the   hyperplanes whose normal vector depending on $x$ as
\begin{align}\label{u88}
\pi_A(x,x_{d+1}) =(x,x_{d+1})- \{\left(y,\langle A(x), y\rangle\right):y\in\mathbb{R}^d \} \subset \mathbb{R}^{d+1}
\end{align}
  as in (\ref{400m}).  
Consider the average of $f$ over a ball embedded in the plan $\pi_A(x,x_{d+1}) $  given by
$$\mathcal{A}_{\pi_A}(f)(x,x_{d+1})=\int_{ \mathbb{R}^{d}} f(x-y,x_{d+1}-\langle A(x),y\rangle) \psi(y)dy.$$
If $A$ is invertible, the smoothing effect of the average   from the variable planes  $\pi_A $,   measured by
$$\|\mathcal{A}_{\pi_A}\|_{L^2_{\alpha}(\mathbb{R}^{d+1})\rightarrow L^2(\mathbb{R}^{d+1})} \lesssim 1\ \text{for $\alpha\ge -\frac{d}{2}$, } $$
is due to
 $\det(A)\ne 0$ which is the  rotational curvature developed by Phong and  Stein 
in 1980s.   They  used   the concept of the rotational curvature for establishing the $L^p$ theory of the singular Radon transforms and generalized Radon transforms \cite{PS}.  This was  preceded by the model case study  of the horizontal plane (\ref{u88}) with $A=J$ of the Heisenberg group $\mathbb{H}^n$  by  Geller and  Stein \cite{GeSt}.  These   effects of the curvature arising from the $x$-side were culminated in the study of  the maximal average on the variable hyper-surfaces  constructed by   Sogge  and  Stein \cite{Sogge,SS1}.  Their theory covers the maximal averages associated with the variable surfaces of co-dimension one.       
To study the maximal average along the surfaces of co-dimension two,   we
 consider the one-parameter family of the $(d-1)$-dimensional surfaces $(x,x_{d+1})- t\{ \left(y,\langle A(x), y\rangle\right): y\in  S^{d-1}\}$  in $ \mathbb{R}^{d+1}$ with a parameter $t\in\mathbb{R}_+$.
Over these surfaces, we  set the average of $f\in L^1_{loc}(\mathbb{R}^{d+1})$   as
 \begin{align}\label{47kg}
 \mathcal{A}_{S^{d-1}(A)}(f) (x,x_{d+1},t):=\int_{y\in S^{d-1}}  f\left(x-ty, x_{d+1}- \langle A(x),ty\rangle \right) d\sigma(y) 
 \end{align}
where $d\sigma $  is the measure  on  the unit sphere $S^{d-1}$, and define the maximal average over all $t>0$ as
 \begin{align}\label{nsc12}
  \mathcal{M}_{S^{d-1}(A)} f(x,x_{d+1}):=\sup_{t>0} \mathcal{A}_{S^{d-1}(A)}(f)(x,x_{d+1},t). 
  \end{align}
Let $d=2n$ and $A=J:=\left(\begin{matrix}0& I\\-I&0\end{matrix}\right)\in M_{2n\times 2n}(\mathbb{R})$. Take  a dilated measure
  $d\sigma_t=  d   \sigma(\cdot /t)/t^{2n}$     and   a Dirac mass  $\delta_{2n+1}$  at $0\in\mathbb{R}$ (in the last coordinate).  Then  the above average $ \mathcal{A}_{S^{2n-1}(J)}(f) (x,x_{2n+1},t)$ for each fixed $t>0$ is the  group convolution of a function $f$ and the surface carried measure  $d\sigma_t\otimes \delta_{2n+1}$ supported on the horizontal plane  $ \mathbb{R}^{2n}\times\{0\}$ of the Heisenberg group $\mathbb{H}^n$:
\begin{align}\label{69kg}
\mathcal{A}_{S^{2n-1}(J)}(f) (x,x_{2n+1},t)=f*_{J}\left(d\sigma_t\otimes \delta_{2n+1}\right)(x,x_{2n+1}). 
\end{align}
    Here $\mathbb{H}^n\cong \mathbb{R}^{2n+1}$ is the $2n+1$ dimensional Heisenberg group endowed with the following group law:
$$(x,x_{2n+1})\cdot (y,y_{2n+1})=\left(x+y, x_{2n+1}+y_{2n+1}+ \langle J(x), y\rangle\right).$$
In 1997,  Nevo   and  Thangavelu   \cite{NT}   initiated a study  on  the maximal average of (\ref{69kg}) for $A=J$ in (\ref{nsc12}) and obtained the maximal and pointwise ergodic theorems for the radial average  on the Heisenberg group $\mathbb{H}^n$ with $n\ge 2$. 
  In 2004,  M\"{u}ller and  Seeger   \cite{MS} proved that, for $n\ge 2$,
\begin{align}\label{nsc112}
\|  \mathcal{M}_{S^{2n-1}(J)}& f\|_{L^p(\mathbb{H}^n)} \le C \|f\|_{L^p(\mathbb{H}^n)}\ \text{for all $f\in L^p(\mathbb{H}^n)$}\nonumber\\
& \text{ if and only if}\   \frac{2n}{2n-1} <p\le \infty 
\end{align}
 by observing that the phase function of the corresponding Fourier integral operators satisfies the two sided fold singularities of  \cite{GS2}.  Indeed, Muller and Seeger in \cite{MS} obtained the result for a   class of operators defined on
M\'etivier groups,  whose last components (multi-dimension) contain the bilinear forms induced from skew symmetric matrices.  
   In the same year,  Narayanan and  Thangavelu  \cite{NT2}  proved  (\ref{nsc112}) by using the spectral theory of the Heisenberg group for $n\ge 2$.   Recently,
 Anderson,  Cladek,  Pramanik and   Seeger   \cite{ACPS}  considered  the  maximal average 
 \begin{align*} 
\sup_{t>0} \int_{y\in S^{2n-1}}  f\left(x-ty, x_{2n+1}-t^2\Lambda(y)- \langle J(x),ty\rangle \right) d\sigma(y) 
 \end{align*}
where  $\Lambda$ is a linear functional
 $\Lambda:\mathbb{R}^{2n}\rightarrow \mathbb{R} $ and  obtained its   $L^p(\mathbb{H}^n)$ boundedness for the same range of $p$  as in (\ref{nsc112}), even though  the dilates  of the surface measure is no longer supported in a fixed hyperplane.   For the case $n=1$, the  boundedness of
 $\mathcal{M}_{S^1(J)}$ on $L^p(\mathbb{H}^1)$  has been  an  open problem   since it was conjectured for the range $p>2 $  in  \cite{NT}.  However, more recently, Beltran, Guo, Hickman and Seeger \cite{BGHS} showed that  $\| \mathcal{M}_{S^1(J)}f\|_{L^P(\mathbb{H}^1)}\le C\|f\|_{L^p(\mathbb{H}^1)}$ with $p>2$ for all  functions $f\in L^p(\mathbb{H}^1)$  satisfying    $f(\cdot,x_3)$ is radial in $\mathbb{R}^2$ for each $x_3\in\mathbb{R}$.
 
 In this paper,  we investigate   the general matrices $A\in M_{d\times d}(\mathbb{R})$  rather than the skew symmetric matrix $J\in M_{2n\times 2n}(\mathbb{R})$ of the Heisenberg group $\mathbb{H}^n$ for  the  $L^p(\mathbb{R}^{d+1})$ boundedness of $\mathcal{M}_{S^{d-1}(A)}$ of  (\ref{nsc12}).  \begin{main}\label{main2}
  Let $d\ge 3$. Suppose that all   eigenvalues of $A\in M_{d\times d}(\mathbb{R})$ have nonzero imaginary parts. Then there exists $C>0$ such that
  \begin{align*} 
&\|  \mathcal{M}_{S^{d-1}(A)} f\|_{L^p(\mathbb{R}^{d+1})}\le C \|f\|_{L^p(\mathbb{R}^{d+1})}\ \text{for all $f\in L^p(\mathbb{R}^{d+1})$}\\
&\qquad\qquad\qquad\qquad \text{if and only if}\   \frac{d}{d-1}<p\le \infty. 
\end{align*}  
Notice that any non-degenerate skew symmetric $d\times d$ matrix satisfies the above hypothesis.   \end{main}
 
\subsection{Homogeneous Group}
To each $A=(a_{ij})\in M_{d\times d}(\mathbb{R}) $, we   assign a  Lie group $\mathbb{G}=\mathbb{G}^{d+1}(A)$  identified with $\mathbb{R}^{d+1}$ endowed with the group multiplication
$$(x,x_{d+1})\cdot (y,y_{d+1})=\left(x+y,x_{d+1}+y_{d+1}+\langle A(x),y\rangle\right)$$
with its Lie algebra $\mathcal{G}=\mathcal{G}^{d+1}(A)$ generated by the basis $\{X_1,\cdots,X_d,X_{d+1}\}$ with 
$X_i=\partial/\partial x_i+(a_{i1}x_1+\cdots +a_{id}x_d)\partial/\partial_{x_{d+1}}$ for $i=1,\cdots,d$ and $X_{d+1}=\partial/\partial x_{d+1}$. They satisfy the commutator relation $[X_i,X_j]=(a_{ji}-a_{ij})X_{d+1}$ with the other commutators vanished. So, for any matrix $A$, it holds that  $[[ \mathcal{G},\mathcal{G}], \mathcal{G}]=0$, i.e.,  the step of   $\mathcal{G}$ is at most 2.  Thus $\mathbb{G}$ is a two step  Nilpotent Lie group if $A\ne A^T$. Indeed, $\mathbb{G}$ is abelian if and only if $A^{T}=A$. For example, the Heisenberg group  $\mathbb{H}^n$ is  the non-abelian group $ \mathbb{G}^{2n+1}(J)$ as $J^T=-J$ while $\mathbb{G}^{2n+1}(I)$ is the abelian group.
On the other hand, the group $\mathbb{G}$ has the inverse element of $(y,y_{d+1})$ given by
 $$(y,y_{d+1})_A^{-1}=\left(-y,-y_{d+1}+\langle A(y),y\rangle\right).$$ 
So, we define the group convolution of $f*_{A}g$ for two   functions $f$ and $g$ in $L^1(\mathbb{G}^{d+1}(A))$  
as
\begin{align*}
f*_{A}g(x,x_{d+1})&=\int f\left(x,x_{d+1})\cdot (y,y_{d+1})_A^{-1}\right) g(y,y_{d+1})dydy_{d+1}\\
&=\int f(x-y,x_{d+1}-y_{d+1}-\langle A(x-y),y\rangle)g(y,y_{d+1})dydy_{d+1}.
\end{align*}
Here $dydy_{d+1}$ is the Euclidean measure, which can be regarded as the Haar measure (both left and right invariant with respect to the group multiplication).
We set a simpler bilinear operation  of two functions $f$ and $g$  rather than $f*_A g$ as
$$f \cdot_A g(x,x_{d+1}):=\int f(x-y,x_{d+1}-y_{d+1}-\langle A(x),y\rangle)g(y,y_{d+1})dydy_{d+1}.$$ 
Consider $ [d\sigma_t \otimes \delta_{d+1} ]$ where $d\sigma_t$ is the  measure along the sphere of radius $t$ in $\mathbb{R}^{d}$ and $\delta_{d+1}$ is a Dirac mass along the last coordinate. Then   $\mathcal{A}_{S^{d-1}(A)}f(\cdot,t)=f\cdot_A[d\sigma_t \otimes \delta_{d+1}]$ in (\ref{47kg}).
 By defining the coordinate change  $\mathcal{U}_Af(x,x_{d+1})=f(x,x_{d+1}+\langle A(x),x\rangle)$,   we check that
\begin{align}\label{0ch}
f*_{A}g=\mathcal{U}_{-A}(\mathcal{U}_Af \cdot_{-A^T} g)\ \text{equivalently}, f\cdot_A g=\mathcal{U}_{-A}(\mathcal{U}_{A}f*_{-A^T}g).
\end{align}   
Thus, the second  identity of  (\ref{0ch})  for $g= [d\sigma_t \otimes \delta_{d+1} ]$ implies 
$$\mathcal{A}_{S^{d-1}(A)}f(\cdot,t)=\mathcal{U}_{-A}(\mathcal{U}_Af*_{-A^T}  [d\sigma_t\otimes \delta_{d+1}]  ).   $$ 
Therefore, we can  regard the average operator $\mathcal{A}_{S^{d-1}(A)}f(\cdot,t)$ in (\ref{47kg}) as the group convolution of a function $f$ and a measure $[d\sigma_t(\cdot )\otimes \delta_{d+1}(\cdot)]$ in the group $\mathbb{G}^{d+1}(-A^T)$.   \subsection{Oscillatory Integral Operators}
Let $A\in M_{d\times d}(\mathbb{R})$ and $j\in \mathbb{Z}_+$. Given  $\lambda\in \mathbb{R} $, we consider an oscillatory integral operator $\mathcal{T}^\lambda_j$ mapping $g\in L^2(\mathbb{R}^d)$ to $ \mathcal{T}^{\lambda}_j g\in L^2(\mathbb{R}^d\times \mathbb{R})$  defined by
 \begin{align}\label{5rpp}
\mathcal{T}_{j}^{\lambda}g (x,t )&= \lambda^{d/2}  \chi\left(t\right)  \int_{\mathbb{R}^d}
e^{2\pi i\lambda[\langle x,\xi\rangle+t| \xi+A(x)|]}
\chi\left(\frac{\lambda t|\xi+A(x)|}{2^j}\right) \widehat{g}(  \xi)d\xi. \end{align}
If $A$ is  invertible and $2^j=\lambda$,  we use a change of variable $x\rightarrow A^{-1}(x)$.  Next we replace (rename) $A^{-1}$  in the phase function by $A$ so that  the phase function   becomes $$\phi(x,t,\xi)=\langle A(x), \xi\rangle+t|\xi+x|.$$ Next, we also  localize $x$ with $\psi(x)$ and   set
\begin{align}\label{apl}
 \mathcal{T}_{\rm{annulus}}^\lambda g(x,t)= \lambda^{d/2}  \chi\left(t\right)\psi\left(x\right) \int_{\mathbb{R}^d}
e^{2\pi i\lambda\phi(x,t,\xi)}
\chi\left(  t|\xi+x| \right) \widehat{g}(  \xi)d\xi.
\end{align}
Then we   reduce the estimate of  the maximal average to  that of the oscillatory integral operators as it states in Main Theorem \ref{main3}. 
\begin{main}\label{main3}
Let $A\in M_{d\times d}(\mathbb{R})$ and $j\in\mathbb{Z}_+$. Suppose    that  
\begin{align}\label{s100}
 \|\mathcal{T}_{j}^{\lambda}\|_{L^2(\mathbb{R}^d)\rightarrow L^2(\mathbb{R}^d\times \mathbb{R})}\lesssim_{\epsilon}   2^{ c(A) j/2} \ \text{uniformly in $\lambda$}.
\end{align}
If $d\ge 3$ in (\ref{s100}) with $d-c(A)>2$, then there exists $C>0$ such that
\begin{align}\label{s81}
\left\| \mathcal{M}_{S^{d-1}(A)}\right\|_{L^p(\mathbb{R}^{d+1})\rightarrow L^p( \mathbb{R}^{d+1})}\le C\ \text{for $p>\frac{(d- c(A))}{(d- c(A))-1}$.} 
\end{align}
If $d=2$ in (\ref{s100}), then it holds that \begin{align}\label{s82}
\left\| \mathcal{M}^{\delta}_{S^1(A)}  \right\|_{L^2(\mathbb{R}^3)\rightarrow L^2(\mathbb{R}^3) }\lesssim_{\epsilon}\left(\frac{1}{\delta}\right)^{c(A)/2}.
\end{align}
   If $A$ is invertible, the hypothesis    (\ref{s100}) can be reduced to the case   $2^j=\lambda$, i.e., 
\begin{align}\label{s83}
  \|\mathcal{T}_{\rm{annulus}}^{\lambda}\|_{L^2(\mathbb{R}^d)\rightarrow L^2(\mathbb{R}^d\times \mathbb{R})}\lesssim_{\epsilon} \lambda^{ c(A)/2} \ \text{ for    $\lambda\ge 1$. }
  \end{align}
   
 \end{main}

Main Theorem \ref{main3} tells that  (\ref{s100}) (or (\ref{s83}) when $A$ invertible) gives
\begin{itemize}
\item The  upper bounds of the main theorem 1  with  the exponents
\begin{align}\label{314a}
c(A)/2=\begin{cases} 0\ \text{if $\text{rank}\left(JA+(JA)^T\right)=2$ for $\text{rank}(A)=2$}\\
1/6\ \text{if $\text{rank}\left(JA+(JA)^T\right)=1$ and $\text{rank}(A)=2$}\\
1/2\ \text{if $\text{rank}\left(JA+(JA)^T\right)=0$ and $\text{rank}(A)=2$}\\
0\ \text{if $\text{rank}(A)=1$.}
\end{cases}
\end{align}
\item  The upper bound of the main theorem 2  with the exponent $c(A)/2=0$.
\end{itemize}
\begin{main}\label{main4}[Regularity and Local smoothing in $L^2$]\label{main4}
Let $A$ be $d\times d$ invertible matrix. Suppose that   
\begin{align}\label{4899}
\|\mathcal{T}_{\rm{annulus}}^\lambda\|_{L^2(\mathbb{R}^{d+1})\rightarrow L^2(\mathbb{R}^{d+1}\times \mathbb{R})}\lesssim_{\epsilon}  \lambda^{c(A)/2}\ \text{for $c(A)/2\ge 0$}.
\end{align}
Then  it holds that
\begin{align}\label{4880}
\left\| \mathcal{A}_{S^{d-1}(A)} \right\|_{L^2_{ \alpha}(\mathbb{R}^{d+1})   \rightarrow L^2(\mathbb{R}^{d+1} \times [1,2])  }  \le  C
\text{ for   $ -(d-1)/2+c(A)/2<\alpha $}.
\end{align}
This implies that (\ref{4880}) holds true for each $c(A)/2$ in (\ref{314a}). In general (\ref{4880}), without   time integral over $[1,2]$,  fails to hold, namely, there exists $A\in M_{d\times d}(\mathbb{R})$ and the range $   -(d-1)/2+c(A)/2< \alpha<\alpha_0$ so that  \begin{align}\label{448}
f\rightarrow   \mathcal{A}_{S^{d-1}(A)}f(\cdot,1)\ \text{is unbounded from $L^2_{ \alpha}(\mathbb{R}^{d+1})$ to $L^2 (\mathbb{R}^{d+1})$}.\end{align}
This means that the integral with respect to $t\in [1,2]$ induces  a nontrivial  $L^2$-regularity   (\ref{4880}) for all  $ -(d-1)/2+c(A)/2< \alpha<\alpha_0$. 
  We shall manifest this observation for the case $d=2$ and $A=\left(\begin{matrix}0&1\\ 1&0\end{matrix}\right)$ in Section \ref{Sec91}.
 \end{main}
Our paper is orgainized as it follows. 
\\
{\bf Organization}.
In Section 2, we  shall classify the matrices   of $M_{2\times 2}$ according to the skew symmetric rank given by $\text{rank}\left(JA+(JA)^T\right)$.  In Section 3, we    discuss  why $JA+(JA)^T$ appears in the proof of  Main Theorem \ref{main1}, and   compare this with $A+A^T$ arising from the Nikodym maximal case.  In Sections \ref{sec3} and \ref{sec5}, we   prove Main Theorem \ref{main3}.
In Sections \ref{secc6} through \ref{sec90}, we  establish
   the   estimates of  (\ref{s100}) and (\ref{s83})  to show the sufficient parts of Main Theorems 1 and 2. In Section \ref{Sec91}, we prove Main Theorem \ref{main4}.  In Section \ref{sec9}, we obtain the lower bounds of Main Theorem \ref{main1}. 
  In the appendix, we revisit some basic estimates (maximal theorem, dilation invariance, orthogonality from cancellation) for the average operators associated with the bilinear forms $B(x,y)=\langle A(x),y\rangle$.
\section{Skew--Symmetric Ranks }\label{sec2}
\subsection{Properties of Skew--Symmetric  Ranks}
\begin{lemma}\label{lem21}
Let $A\in M_{d\times d}(\mathbb{R})$. The rank of the matrix $ JA+(JA)^T $ is invariant under a switch  between (i)  $JA$ and $AJ$, (ii) $A$ and $A^T$,  (iii) $A$ and $A^{-1}$ (when $A$ is invertible),  (iv) $A$ and $QAQ^T$ (for an orthogonal matrix $Q$), and (v) $A$ and $cA$ with $c\ne 0$.   
 \end{lemma}

\begin{proof}[Proof of Lemma \ref{lem21}]
  By $JJ^T=J^TJ=I$ and $J^T=-J$,  
we obtain (i) and (ii) from $$ JA+(JA)^T=J (AJ+(AJ)^T)J^T\ \text{and}\ JA+(JA)^T=-(A^TJ+(A^TJ)^T).$$
 We next obtain (iii)  by inserting $AJ$ into $A$ below
  \begin{align}\label{ss44}
(A+A^T)(A^{-1})= (A^T)( (A^{-1})^{T}+A^{-1})  
\end{align}
 and obtain (iv)   from $  Q[JA+(JA)^T] Q^T =\pm (J(QAQ^T)+[J(QAQ^T)]^T) $ due to     $QJ=\pm JQ$.  Finally (v) follows from $(cA)^T=cA^T$.
 \end{proof}

\begin{proposition}\label{lem10044}
Let $A\in M_{2\times 2}(\mathbb{R})$ be an invertible matrix. Then it holds that \begin{itemize}
\item[(1-1)] $\text{rank}\left(JA+(JA)^T\right)=2$  if and only if   $A$ has two different eigenvalues.
\item[(1-2)] $\text{rank}\left(JA+(JA)^T\right)=1$   if and only if     there is   $r\ne 0$ and an orthogonal matrix $Q$ such that 
 $\text{$rA=Q^T \left(\begin{matrix}r\lambda&1\\ 0&r\lambda\end{matrix}\right)    Q$   for a nozero eigenvalue $\lambda $ of $A$.}$    Here 1 can be located at the opposite side. 
 \item[(1-3)] $\text{rank}\left(JA+(JA)^T\right)=0$   if and only if   $A =cI$ for an identity matrix $I$ and $c\ne 0$.
\end{itemize}
 \end{proposition} 
 
To prove Proposition \ref{lem10044} for $A=\left(\begin{matrix}a_{11}&a_{12}\\
a_{21}&a_{22}
\end{matrix}\right)$, we observe that 
\begin{align}\label{78kg}
\det\left(JA+(JA)^T\right)&=\det\left(\begin{matrix}-2a_{21}& a_{11}-a_{22}\\
 a_{11}-a_{22}& 2a_{12}
\end{matrix}\right) \\
&=4(a_{11}a_{22}-a_{12}a_{21})-(a_{22}+a_{11})^2 \nonumber\end{align}
which appears   inside of the square root in the following formula of   an eigenvalue  $\lambda$  of $A$,
\begin{align}\label{1001kg}
 \lambda= \frac{(a_{11}+a_{22})\pm \sqrt{-[4(a_{11}a_{22}-a_{12}a_{21}) -(a_{22}+a_{11})^2]   }  }{2}. \end{align} 
 \begin{proof}[Proof of (1-1)]
By  (\ref{78kg}) and (\ref{1001kg}),  it holds that $\det(JA+(JA)^T)=0$ if and only if $A$ has  an eigenvalue $\lambda$ of the algebraic multiplicity 2. This shows (1-1) of Proposition \ref{lem10044}.
 \end{proof}
 \begin{proof}[Proof of (1-3)]
Let $A$ be a nonzero matrix.  Then it  holds that   $JA+(JA)^T={\bf 0}$ in (\ref{78kg}) if and only if $a_{12}=a_{21}= 0$ and $a_{11}=a_{22}\ne 0$,  namely $A=cI$ with $c\ne 0$. This shows (1-3) of Proposition \ref{lem10044}
\end{proof}
\begin{proof}[Proof of (1-2)]
Let $\mathcal{E}(\lambda)$ be the eigenspace associated with an eigenvalue $\lambda$. Note  that from   (1-1) and (1-3), 
 the condition $\text{rank}\left(JA+(JA)^T\right)=1$ and $\rm{rank}(A)=2$ implies that
 $A$ has an eigenvalue $\lambda\ne 0$ of algebraic multiplicity 2 with    
  $\text{dim}(\mathcal{E}(\lambda))=1$ because an algebraic multiplicity 2 with $ \text{dim}(\mathcal{E}(\lambda))=2$ happens only if $A=cI$ at (1-3).     Let $\lambda$ be such an eigenvalue. Thus 
$\left(\begin{matrix}a_{11}-\lambda &a_{12}\\
a_{21}&a_{22}-\lambda
\end{matrix}\right)= A-\lambda I =\left(\begin{matrix}\alpha_1&\alpha_2\\
k\alpha_1&k\alpha_2
\end{matrix}\right)$ where $k=-\alpha_1/\alpha_2$ because 
\begin{align}\label{2323}
\alpha_1+k\alpha_2=a_{11}+a_{22}-2\lambda=0 
\end{align}
 from 
$ 4(a_{11}a_{22}-a_{12}a_{21})-(a_{22}+a_{11})^2=\det(JA+(JA)^T)=0$ in (\ref{1001kg}).  At first, take  $v = (v_1,kv_1) \in \mathbb{R}^2$ with $v_1=\frac{1}{1+k^2}$. Then as the rows of  $A-\lambda I$  are parallel to $(\alpha_1,\alpha_2)$, \begin{align}\label{ph00}
|v|=1\ \text{and}\  r(A-\lambda I )v=0\ \text{ for any $r\ne0$}  
\end{align}
For this fixed vector $v= (v_1,kv_1)$ chosen in (\ref{ph00}),  let us next find a real number $r\ne 0$ and $w\in \mathbb{R}^2$ such that 
\begin{align}\label{ph01}
\text{$r(A-\lambda I)w=v$, $w\cdot v=0$, and $|w|=1$}.
\end{align}
\begin{proof}[Proof of (\ref{ph01})]
Let $ v= (v_1,kv_1)$ satisfying (\ref{ph00}).  Obviously, we can first find  $w=(w_1,w_2)$ satisfying  $w\perp v=0$  with $|w|=1$. Since $(\alpha_1,\alpha_2)\perp v$ (due to $\alpha_1+k\alpha_2=0$ in the above), it holds that $w$ and $(\alpha_1,\alpha_2)$ are parallel. Thus   $w\cdot (\alpha_1,\alpha_2)\ne 0$.
Next, given such $w$, we can take $r\ne 0$ satisfying
 $r(w_1 \alpha_1+w_2\alpha_2)=v_1$, namely, $r(A-\lambda I)w=r\left(\begin{matrix}\alpha_1&\alpha_2\\
k\alpha_1&k\alpha_2
\end{matrix}\right) \left(\begin{matrix}w_1\\
w_2
\end{matrix}\right)= v$.   This implies (\ref{ph01}).
 \end{proof}
From (\ref{ph00}) and (\ref{ph01}),  set the orthogonal matrix $Q=[v,w]$ satisfying the Jordan form $ [r(A-\lambda I)]Q=  [0,v]=Q\left(\begin{matrix}0 &1\\ 0&0\end{matrix}\right)   $ where $Q^{-1}=Q^T$ due to   $w\cdot v=0$ and $|v|=|w|=1$.  Therefore, we proved $\Rightarrow$ of (1-2). Finally, check $JA+(JA)^T=\left(\begin{matrix}0 &0\\ 0&-c\end{matrix}\right) $ for $A=\left(\begin{matrix}1 &c\\ 0&1\end{matrix}\right) $. This with (iv) and (v) of Lemma \ref{lem21}    leads the other direction $\Leftarrow$ of (1-2).  Therefore we have finished the proof  for (1-2) of  Proposition \ref{lem10044}.
\end{proof} 
\begin{lemma}\label{lem2.2}
Given a surface  measure $ \sigma$ over a sphere $S^{d-1}$  and a matrix $A\in M_{d\times d}(\mathbb{R})$, we recall
$$\mathcal{M}_{S^{d-1}(A)}(f)(x,x_{d+1})=\sup_t\int_{S^{d-1}} f(x-ty,x_{d+1}-\langle A(x),ty\rangle)d\sigma(y).$$
Then the $L^p(\mathbb{R}^{d+1})\rightarrow L^p(\mathbb{R}^{d+1})$ norm of the maximal operator $\mathcal{M}_{S^{d-1}(A)}$ is invariant when $A(x)$ in the above is replaced by  
 (i) $rA(x)$ for $r\ne 0$ and  (ii) $Q^TAQ(x)$ for any orthogonal matrix $Q$. This   yields the invariance  of the operator norms of the corresponding maximal operators  $\mathcal{M}_{S^{d-1}(A)}^{\delta} $  in
Main Theorems  \ref{main1}   under a switch of $A$ with (i) or  (ii).  
\end{lemma}
\begin{proof} 
  Set $[f]^{1/r}(x,x_{d+1})= f(x,x_{d+1}/r)$ and $[f]_Q(x,x_{d+1})=f(Q(x),x_{d+1})$. Then,
\begin{align*}
[\mathcal{M}_{S^{d-1}(A)} f]^{1/r}(x,x_{d+1})&=\sup_{t>0}\int
[f]^{1/r}\left(  x-  t  y,  x_{d+1} -  \langle rA  (x),   ty \rangle  \right)d\sigma(y),\\
[\mathcal{M}_{S^{d-1}(A)} f]_{Q}(x,x_{d+1})&=\sup_{t>0}  \int_{S^{d-1}}
f\left( Q( x)-  Q (ty),x_{d+1} - \langle A Q(x),   tQ(y) \rangle \right)d\sigma(y) \\
&=\sup_{t>0} \int_{S^{d-1}}
[f]_Q\left(x- ty,x_{d+1} - \langle Q^TAQ (x),  ty \rangle \right)d\sigma(y) .
\end{align*} 
The above two identities implies the desired $L^p$ norm invariance.
\end{proof}
\begin{remark}
Let $\text{rank}\left(JA+(JA)^T\right)=1$ and $\text{rank}(A)=2$ in Main Theorem \ref{main1}. Then from  (1-2) in Proposition \ref{lem10044}   with  Lemmas \ref{lem21} and   \ref{lem2.2}, it suffices to treat  $A=  \left(\begin{matrix}1&c\\ 0&1\end{matrix}\right).$ 
 \end{remark}
\section{Idea of Proof}    
 In this section, we shall see why $JA+(JA)^T$ is involved in our estimate  for the oscillatory integral operators  $ \mathcal{T}_{\rm{annulus}}^{\lambda}$ in (\ref{s83}). 
 \subsection{Why $JA+(JA)^T $     arises} To obtain (\ref{s83})     in Sections \ref{secc6} through \ref{sec90}, we shall  estimate the integral $\int |K(\xi,\eta)|d\xi$ for the integral kernel $ K(\xi,\eta)$ of $ [\mathcal{T}_{\rm{annulus}}^{\lambda}]^*\mathcal{T}_{\rm{annulus}}^{\lambda}$,
\begin{align}
K(\xi,\eta)&=\lambda^{d }\int e^{2\pi i  \lambda \Phi(x,t,\xi,\eta)}\psi(x)^2\chi(t)^2\chi\left( 
t  |\xi+ x|\right)\chi\left(  t |\eta+
x| \right) dxdt.\label{28n}\end{align}
From (\ref{apl}), the phase function is 
\begin{align*} 
\Phi(x,t,\xi,\eta)&=\phi(x,t,\xi)-\phi(x,t,\eta)= \langle A^T(\xi-\eta),x\rangle +t\left( |\xi+x  |-|\eta+x|\right).
\end{align*}
For fixed $\eta$,  the support of $(x,\xi,t)$ in (\ref{28n}) is contained in the set $$D(\eta)=\{(x,t,\xi):|x|\le 1\ \text{and}\ |t|\approx |\xi+x|\approx|\eta+x|\approx 1\}.$$
On this region, if $|\partial_{t}\Phi|$ or $|\nabla_x\Phi|$    is away from  $ \lambda^{-1+\epsilon}$ or $c|\xi-\eta|$ respectively, then the integral $\int |K(\xi,\eta)|d\xi=O(1)$.  Thus, it suffices to consider the region $D(\eta)$ satisfying
\begin{align}
\partial_t\Phi(x,t,\xi,\eta)|&= |\xi+x  |-|\eta+x|=O( \lambda^{-1+\epsilon})\label{33}\\
\nabla_x\Phi(x,t,\xi,\eta)&= A^T(\xi-\eta) + t \left( \frac{\xi+x}{|\xi+x|}-  \frac{\eta+x}{|\eta+x|} \right) =o(|\xi-\eta|).\label{34}
\end{align}
Under the condition of (\ref{33}), we shall verify   in (\ref{34}) that
\begin{align}\label{0023g}
\nabla_x\Phi(x,t,\xi,\eta)&=\left(A^T+\frac{t}{|\xi+x|}I\right)(\xi-\eta)+O(\text{smaller term}).
 \end{align}
 which is very adaptive to the application of the integration by parts.
From  the invertibility of $\left(A^T+\frac{t}{|\xi+x|}I\right)  $ due to our hypothesis of $A $ having no  purely real eigenvalues,  we shall obtain   the following core estimate 
\begin{align}\label{pc001}
 \|  \mathcal{T}_{\rm{annulus}}^\lambda\|_{L^2(\mathbb{R}^d)\rightarrow L^2(\mathbb{R}^d\times \mathbb{R})}=O(\lambda^{d/2} \lambda^{-d/2}\lambda^{\epsilon})
\end{align} 
for the proof of 
Main Theorem 2 when $d\ge 3$.  Here $\lambda^{-d/2}$ is the decay part, which is usually obtained from the full rank of the mixed hessian matrix.  Before estimating $\|\mathcal{T}_{\rm{annulus}}^\lambda\|_{op}$ of  (\ref{pc001}) for $d=2$ of Main Theorem 1 with $\text{rank}(A)=2$, we need to write out the mixed hessian matrix of  the phase function  $\phi(x,t,\xi)=\langle A(x), \xi\rangle+t|x+\xi|$ for $x,\xi\in \mathbb{R}^2$ and  $t\in\mathbb{R}$ in (\ref{apl}) as
\begin{align}
\phi_{(x_1x_2t)(\xi_1\xi_2)}''=\left(\begin{matrix} \phi_{x_1\xi_1}&\phi_{x_1\xi_2} \\
\phi_{x_2\xi_1}&\phi_{x_2\xi_2}\\
\phi_{t\xi_1}&\phi_{t\xi_2}
 \end{matrix}  \right)=   \left(\begin{matrix}a_{11}+t\frac{u_2^2}{|u|^3}&u_{12}- t\frac{u_1u_2}{|u|^3} \\ a_{21}-
t\frac{u_1u_2}{|u|^3} &a_{22}+t\frac{u_1^2}{|u|^3} \\ \frac{u_1}{|u|} &\frac{u_2}{|u|} \end{matrix}  \right)\label{a71}
\end{align}
where we denote $u=(u_1,u_2):=(x_1+\xi_1,x_2+\xi_2)$.
 A further calculation gives 
\begin{align}
\det\left(\phi_{(x_1x_2)(\xi_1\xi_2)}''\right)&= \det(A)+\frac{t}{|u|^3}\left\langle \frac{(A+A^T)}{2}u,u\right\rangle,\nonumber\\
\det\left(\phi_{(x_1t)(\xi_1\xi_2)}''\right)& =\frac{1}{|u|} \left(\det\left(   \begin{matrix}a_{11} &a_{12} \\ u_1&u_2\end{matrix}  \right) +\frac{tu_2}{|u|} \right),\label{a72} \\
\det\left(\phi_{(x_2t)(\xi_1\xi_2)}''\right)& =\frac{1}{|u|} \left(\det\left(   \begin{matrix}a_{21} &a_{22} \\ u_1&u_2\end{matrix}  \right) -\frac{tu_1}{|u|} \right).\nonumber
\end{align}
The case $A=J$ (corresponding to the Heisenberg group  $A+A^T={\bf 0}$)  is the best case satisfying that   $\text{rank}(JA+(JA)^T)=2$. For this case, the determinant of the mixed hessian in (\ref{a72}) is  
\begin{align*}
 \det\left( \phi_{(x_1x_2)(\xi_1\xi_2)}''\right)=\det(J)+\frac{t}{|u|^3}\left\langle \frac{(J+J^T)}{2}u,u\right\rangle=\det(J)= 1.
\end{align*}
This  enables us to  apply  the H\"{o}rmander  theorem for the estimate of the above $\mathcal{T}_{\rm{annulus}}^\lambda f(\cdot,t)$ with  a fixed $t$, to obtain a good operator norm $O(\lambda^{d/2}\lambda^{-d/2})$ with $d=2$ in (\ref{pc001}). For this case $A=J$ (Heisenberg group), the time variable $t$ does not play a role for earning any regularity. In general case, we need to utilize the time variable for the $L^2$ estimate in this paper. However, the difficulty occurs in the sense that    
\begin{itemize}
\item we are  not able   to apply the H\"{o}rmander  theorem for the $2\times 2$ sub-matrix  as  $$ \rm{rank}(\phi_{(x_1x_2t)(\xi_1\xi_2)}'')=2  \ \text{ in (\ref{a71}) can fail,} $$ even if the phase has the good  rank condition, namely, $\text{rank}( JA+(JA)^T)=2$. 
\item any direct application of  fold singularity condition does not seem to give a full additional regularity $+1/2$ as it appears in the following example.
\end{itemize}
A model example related with these obstacles  is
\begin{align*}
A=\left(\begin{matrix}0&c\\c&0\end{matrix}  \right)  \ \text{where}\ \text{$a_{11}=a_{22}=0$ and $a_{12}=a_{21}=c  $ with $c\ne 0$.}
\end{align*}
For this case, the  three   mixed hessians  of (\ref{a72}) are simultaneously singular since $$
 \det( \phi_{(x_1x_2)(\xi_1\xi_2)}'') =\det(\phi_{(x_1t)(\xi_1\xi_2)}'')=\det(\phi_{(x_2t)(\xi_1\xi_2)}'')=0,  $$
   at $(x_1,x_2,t,\xi_1,\xi_2)$      satisfying for  $(u_1,u_2)=(x_1+\xi_1,x_2+\xi_2)$
$$  |u_1|=|u_2|=1   \ \text{and}\ t=\pm \sqrt{2} c\ \text{where $\pm$ is the sign of  $\frac{u_2}{u_1}$}.$$
 From the rank condition  $ \rm{rank}(\phi_{(x_1x_2t)(\xi_1\xi_2)}'')=1$ for this $A$,  we get only $O(\lambda^{2/2} \lambda^{-1/2} )$ while we need to obtain $ \|  \mathcal{T}_{\rm{annulus}}^\lambda\|_{L^2(\mathbb{R}^2)\rightarrow L^2(\mathbb{R}^2\times \mathbb{R})}=O(\lambda^{2/2} \lambda^{-2/2}\lambda^{\epsilon})$. 
In order to overcome this lack of nondegeneracy and obtain an additional factor $\lambda^{-1/2}$,   we    estimate the intersection of the sublevel sets satisfying (\ref{33}) and (\ref{34}):
\begin{itemize}
\item $   |\{(x,t,\xi)\in D(\eta): |\partial_t\Phi(x,t,\xi,\eta) | <\lambda^\epsilon/\lambda\}|$ of (\ref{33}) and 
\item   $|\{(x,t,\xi)\in D(\eta):| \langle v,\nabla_x\Phi(x,t,\xi,\eta)\rangle| <\lambda^\epsilon/\lambda\}|$ derived from (\ref{34}).  
 \end{itemize}
Here $v= J\frac{\left( \xi-\eta\right)}{|\xi-\eta|} $ perpendicular to  $\nabla_x\partial_t\Phi(x,t,\xi,\eta) $.  From  (\ref{0023g}), it follows    that
\begin{align*}
  \left\langle v,\nabla_x\Phi  (x,t,\xi,\eta)\right\rangle&= \left\langle \frac{1}{2}(JA+(JA)^T)(\xi-\eta),\frac{\left( \xi-\eta\right)}{|\xi-\eta|} \right\rangle+O(small).
\end{align*}
So, we shall focus on  the estimation of the measure of the   region $\mathcal{R}^{A}_{\rm{annulus}}(\eta)$ given by
\begin{align*} 
 \left\{(x,t,\xi)\in D(\eta):\big| |\xi+x  |-|\eta+x|\big|+\left|\left\langle (JA+(JA)^T)(\xi-\eta),  \frac{\left( \xi-\eta\right)}{|\xi-\eta|}\right\rangle \right|\lesssim\frac{\lambda^{\epsilon}}{\lambda}  \right\} \end{align*}
  where the oscillation  in  (\ref{28n})  ceases.
On the region $ \mathcal{R}^{A}_{\rm{annulus}}(\eta)$, notice that
\begin{itemize}
\item the first term  $\big| |\xi+x  |-|\eta+x|\big|$ arises from the circles $S^{1}$,
\item the second one  $\left\langle (JA+(JA)^T)(\xi-\eta),    \frac{\left( \xi-\eta\right)}{|\xi-\eta|} \right\rangle  $     from the planes $\pi_A(x,x_{3}) $ in (\ref{u88}).
\end{itemize}
   From those two terms, we   derive the  desired bound of $ \|  \mathcal{T}_{\rm{annulus}}^\lambda\|_{op}$ in Main Theorem \ref{main1}: 
\begin{itemize} 
\item[(1)] If  $\text{rank}(JA+(JA)^T)=2$,  then     $
   \left| \mathcal{R}^{A}_{\rm{annulus}}(\eta) \right|\lesssim \frac{\lambda^{2\epsilon}}{\lambda^2}\ $ and   $\|  \mathcal{T}_{\rm{annulus}}^\lambda\|_{op}=O(\lambda^{ \epsilon})$.
\item[(2)] If $\text{rank}(JA+(JA)^T)=0$,    then $   \left| \mathcal{R}^{A}_{\rm{annulus}}(\eta)\right|\lesssim \frac{\lambda^{2\epsilon}}{\lambda} $ and    $\|  \mathcal{T}_{\rm{annulus}}^\lambda\|_{op}=O(\lambda^{ 1/2+ \epsilon})$.
\item[(3)]
 If  $\text{rank}\left(JA+(JA)^T\right)=1$ and $\text{rank}(A)=2$, we shall apply the well known estimates of the oscillatory integral operators having the phase of the  two sided fold singularities of \cite{GS2} in Section \ref{sec8.2}.  This leads
 $ \|  \mathcal{T}_{\rm{annulus}}^\lambda\|_{op}=O(\lambda^{ 1/6+\epsilon})$.
 \end{itemize}
Thus,  this with the reduction of Main Theorem \ref{main3}  yields  $\|\mathcal{M}^{\delta}_{S^1(A)} \|_{L^2\rightarrow L^2}\lesssim_{\epsilon} \delta^{-c(A)/2} $ of   Main Theorem 1 for $\det(A)\ne 0$. Finally, we remark that
 (1) cannot be obtained by any direct application of fold singularity conditions associated with our phase function $\phi$, whereas (3) can be obtained from the two sided fold singularity associated with $\phi$.
  \subsection{Where  $ A+A^T $     arises from} 
For the Nikodym maximal function $\mathcal{N}^{\delta}_{T(A)}$  in (\ref{08g}), we can obtain an analoguous  result of Main Theorem 3.  Let  $e(\theta)=(\cos \theta,\sin \theta)$ with $\theta\in [0,2\pi] $ and $x\in\mathbb{R}^2$.  Then the   oscillatory integral operator  $\mathcal{T}_{\rm{tube}}^\lambda$ corresponding to $ \mathcal{T}_{\rm{annulus}}^\lambda$ is defined as
   $$\mathcal{T}_{\rm{tube}}^\lambda  f(x,\theta)=\lambda^{1/2}  \psi(x)\int  e^{2\pi i\lambda \langle A(x), \xi\rangle }\chi\left( \langle \xi+x,
e^{\perp}(\theta)\rangle  \right) \psi\left(  \lambda \langle \xi+x,
e(\theta)\rangle \right)  \widehat{f}(\xi)d\xi.$$
Then the singular set $\mathcal{R}_{\rm{tube}}^A(\eta)$ of $\mathcal{T}_{\rm{tube}}^\lambda$  corresponding to     $\mathcal{R}_{\rm{annulus}}^A(\eta)$ of   $\mathcal{T}^\lambda_{\rm{annulus}}$, is
$$\mathcal{R}_{\rm{tube}}^A(\eta)= \left\{(x,\theta,\xi):  \left|\left\langle (A+A^T)(\xi-\eta), \frac{\left( \xi-\eta\right)}{|\xi-\eta|} \right\rangle \right|\lesssim \frac{\lambda^{\epsilon }}{\lambda} \ \text{and}\ -x,\xi \in  T_{1/\lambda}(e^{\perp}(\theta),\eta)\right\}.$$
Here $ T_{1/\lambda}(e^{\perp}(\theta),\eta)$ is a tube with  dimensions $1\times 1/\lambda$ along the direction $e^{\perp}(\theta)$ centered at $\eta$, as
\begin{align*}
T_{1/\lambda}(e^{\perp}(\theta),\eta) :=\left\{\xi:|\langle \xi-\eta,  e^{\perp}(\theta)\rangle| \lesssim 1\ \text{and}\ |\langle \xi-\eta,  e(\theta)\rangle|  \lesssim 1/\lambda\right\}.
\end{align*}
Indeed,   we do not treat   any detail   regarding the estimate of the Nikodym maximal function in this paper. However, it is remarkable that
these two singular region $\mathcal{R}_{\rm{annulus}}^A(\eta)$ and $\mathcal{R}_{\rm{tube}}^A(\eta)$  are where the corresponding phase functions of $\mathcal{T}_{\rm{annulus}}^\lambda $ and $ \mathcal{T}_{\rm{tube}}^\lambda $ stop their oscillations respectively. These two sets give rise to  the symmetric and skew--symmetric rank conditions of $A$ in Definition \ref{def12} which determine the operator norms of   $\mathcal{N}^\delta_{T(A)}$ and $\mathcal{M}^\delta_{S^1(A)}$.

\section{Proof of  of Main Theorem \ref{main3}; Reduction to  $\mathcal{T}_j^\lambda$ }\label{sec3}
The first part of  Main Theorem \ref{main3} is to change  the maximal function estimate to that of the  oscillatory integral operators $\mathcal{T}_j^\lambda$. Its proof  is based on  the Plancherel theorem with respect to the last variable $x_{d+1}$,  combined with   (i) the asymptotic expansion of $\widehat{d\sigma}(\xi)$ and (ii) the majorization of $L^\infty(dt)$ by $L^2_{1/2}(dt)$. 
\subsection{Frequency decomposition}
Let  $\widehat{d\sigma}$ be the Fourier transform of the
 measure $d\sigma$ on the unit sphere $S^{d-1}$. Decompose
 $  \widehat{d\sigma}(\xi)=\sum_{j=0}^\infty\widehat{d\sigma_j}(\xi)$ 
 with  $d\sigma_j$ having the frequency support $|\xi|\approx 2^j$:
\begin{align}\label{kc1}
\widehat{d\sigma_0}(\xi)=\widehat{d\sigma}(\xi)\psi(\xi)
  \ \text{and}\ \widehat{d\sigma_j}(\xi)=\widehat{d\sigma}(\xi)\chi\left(\frac{\xi}{2^j}\right)\ \text{for $j\ge 1$.}
  \end{align}
Let $A\in M_{d\times d}(\mathbb{R})$ and set the average  associated with  each piece of measure  $d\sigma_j$   as
  \begin{align}\label{kcc1}
\mathcal{A}_{j} f(x,x_{d+1},t)&: 
 = \int_{\mathbb{R}^d}  f\left(x-ty,x_{d+1}-\langle A(x), ty\rangle\right)d\sigma_j(y) 
  \end{align}
where $(x,x_{d+1},t)\in\mathbb{R}^{d+1}\times\mathbb{R}$. To each   integer $j \ge 0$, we assign the corresponding maximal operator $\mathcal{M}_j$ defined by
\begin{align}\label{kcc6}
\mathcal{M}_jf(x,x_{d+1})= \sup_{t>0}\mathcal{A}_{j} f(x,x_{d+1},t).
\end{align}
For the case $j=0$ in (\ref{kcc6}),  a rapidly decreasing function  $d\sigma_0$ due to  $\widehat{d\sigma_0}\in C_0^\infty$ in (\ref{kc1}) gives a good $L^p$ result for $\mathcal{M}_0$. In fact, in the appendix (Section   \ref{sec10}),  we show  that for any $A\in M_{d\times d}(\mathbb{R}),$ there is a constant $C_p>0$ such that
\begin{align}\label{kcc}
\|\mathcal{M}_0f\|_{L^p(\mathbb{R}^{d+1})}\le C_p \| f\|_{L^p(\mathbb{R}^{d+1})} \ \text{  for all  $f\in L^p(\mathbb{R}^{d+1})$  where $1<p\le \infty$.}
\end{align}
 To deal with $j\ge 1$ in (\ref{kcc6}), we use the Fourier inversion for the second term of (\ref{kc1}) and the rapid decay   $\chi^{\vee}(x-y)=O(1+|x|)^{-N})$ with $|x|\ge 2$ and $y\in S^{d-1}$, we  verify that
\begin{align}\label{kc2}
d\sigma_j(x)=\int_{S^{d-1}} 2^{dj} \chi^{\vee}\left(2^j(x-y)\right)d\sigma(y)=O\left(\frac{2^j}{(1+|x|)^{N}}\right)
\end{align}  
where 
$2^j$ on the RHS of (\ref{kc2}) follows from the measure estimate of   the support  for $y\in S^{d-1}$.  
Hence  we apply (\ref{kcc}) and  (\ref{kc2}) to obtain  an upper bound for the $L^p$ norm of the maximal operator $\mathcal{M}_j$ in (\ref{kcc6}) with $j\ge 0$   for any $1<p\le \infty$ 
 \begin{align}\label{kcc2}
\|\mathcal{M}_j\|_{L^p(\mathbb{R}^{d+1})\rightarrow L^p(\mathbb{R}^{d+1})}\lesssim 2^j \ \text{ for all  matrices $A\in M_{d\times d}(\mathbb{R})$ in (\ref{kcc1}).}
\end{align}
To prove Main Theorem \ref{main3}, we first reduce the estimate of $\mathcal{M}_{S^{d-1}(A)}$ to that of $\mathcal{M}_j$.
\begin{proposition}\label{prop44}
Let  $d-2>c(A)$. Suppose that    $\mathcal{M}_j$ for every $j\ge 0$ is bounded as,
\begin{align}\label{kc4}
  \|\mathcal{M}_jf\|_{L^2(\mathbb{R}^{d+1})}\lesssim_{\epsilon}  2^{c(A)j/2}2^{-(d-2)j/2}\| f\|_{L^2(\mathbb{R}^{d+1})}\ \text{for all $f\in L^2(\mathbb{R}^{d+1})$}.
\end{align}
Then   (\ref{s81}) and (\ref{s82}) hold, which lead  the desired bounds of $\mathcal{M}_{S^{d-1}(A)}$ and $\mathcal{M}_{S^1(A)}^{\delta}$.
\end{proposition}
\begin{proof}[Proof of  Proposition \ref{prop44}]
From the condition $d-2>c(A)$, we can sum (\ref{kc4}) over $j$ to obtain the $L^2(\mathbb{R}^{d+1})$ boundedness of $\mathcal{M}_{S^{d-1}(A)}$. 
The interpolation of (\ref{kc4}) for $p=2$ and (\ref{kcc2}) for $p$ near $1$, combined with (\ref{kcc}), yields that  for $1<p<2$ 
\begin{align*}
  \|\mathcal{M}_{S^{d-1}(A)}\|_{L^p\rightarrow L^p}&\le \sum_{j=0}^\infty \|\mathcal{M}_{j}\|_{L^p\rightarrow L^p}  \\
  & \le\sum_{j\ge 0} 2^{O(\epsilon j)}2^{j(1-\theta)}2^{j(c(A)/2-(d-2)/2)\theta} 
 \le \sum_{j\ge 0} 2^{O(\epsilon j)} 2^{j[1+(1-1/p)(-d+c(A))]}
\end{align*}
for $1/p=\frac{(1-\theta)}{1+\epsilon}+\frac{\theta}{2}$, namely, $\theta=2(1-1/p)+O(\epsilon)$. Thus (\ref{kc4}) implies the $L^p$ boundedness of the spherical maximal operator $\mathcal{M}_{S^{d-1}(A)}$ for $p>\frac{(d- c(A))}{(d- c(A))-1}$ with $d\ge 3$ in  (\ref{s81}).
 Next, we claim that  (\ref{kc4}) for $d=2$ would lead the desired bound  for  the annulus maximal operator $\mathcal{M}^\delta_{S^1(A)}$ in (\ref{s82}). Write the  measure  supported on the annulus $S^1_\delta$ in (\ref{0g}) with  
 $\delta=2^{-j}$ as
 \begin{align}\label{ike1}
 \left(d\sigma*2^{dj}\psi (2^j\cdot )\right)\left(y\right)\ \text{for $y\in \mathbb{R}^2$  where $*$ is the Euclidean convolution in $\mathbb{R}^2$.}
\end{align}
Take the Fourier transform of (\ref{ike1}) in the Euclidean space $\mathbb{R}^2$ and make a dyadic decomposition,
$$ \widehat{d\sigma}(\xi) \widehat{\psi}\left(  \frac{\xi}{2^j}\right)=\widehat{d\sigma}(\xi) \widehat{\psi}\left(  \frac{\xi}{2^j}\right)\left( \psi(\xi)+  \sum_{\ell=1}^\infty\chi\left(  \frac{\xi}{2^\ell}\right)  \right). $$
Fix  $j\in\mathbb{Z}_+$ for $2^{-j}=\delta$.  As  we did in (\ref{kcc6}), we set the maximal operator $\widetilde{\mathcal{M}_\ell}$ associated with the symbol $\widetilde{\sigma}_\ell$ below, that is controlled as  $$\widetilde{\sigma}_\ell(\xi):=\widehat{d\sigma}(\xi) \widehat{\psi}\left(  \frac{\xi}{2^j}\right)
\chi\left(  \frac{\xi}{2^\ell}\right)=\begin{cases} O\left( \widehat{d\sigma}(\xi)  \chi\left(  \frac{\xi}{2^\ell}\right) \right) \ \text{if $2^\ell\le 2^j$ }  \\
 O\left(2^{-N(\ell-j)}  \widehat{d\sigma}(\xi)  
\chi\left(  \frac{\xi}{2^\ell}\right)     \right) \ \text{if $2^\ell>2^j$}
\end{cases}$$
for $N\gg 1$, and  $\widetilde{\sigma}_0(\xi)=\widehat{d\sigma}(\xi)  \widehat{\psi}\left(  \frac{\xi}{2^j}\right) \psi(\xi) =O(\widehat{d\sigma}(\xi)  \psi(\xi) ) $. This yields   for $\delta=2^{-j}$,  
\begin{align*}
\|\mathcal{M}^{\delta}_{S^{d-1}(A)}\|_{L^2\rightarrow L^2} &\lesssim \sum_{\ell=0}^\infty \|\widetilde{\mathcal{M}}_\ell\|_{L^2\rightarrow L^2}\\
&\lesssim \sum_{\ell: 2^0\le 2^{\ell}\le 2^j} \| \mathcal{M}_\ell\|_{L^2\rightarrow L^2}+ \sum_{\ell: 2^j<2^{\ell}  } 2^{-N(\ell-j)}\|\mathcal{M}_j\|_{L^2\rightarrow L^2}  
\end{align*}
which is bounded by $2^{2\epsilon j}   2^{jc(A)/2+\epsilon j}$ due to (\ref{kc4}). This leads
 (\ref{s82}) for $\delta=2^{-j}$.  Therefore, we finish the proof of Proposition \ref{prop44}. 
\end{proof}
Hence, for a proof of Main Theorem \ref{main3}, we   show  (\ref{kc4}) under the assumption of  (\ref{s100}).  \subsection{Symbol Representations}\label{Sec42}
Our next step is to replace the average operator  (\ref{kcc1}) with the  Fourier integral operator associated with the wave propagation $e^{2\pi it|\xi|}$ arising from the Euclidean Fourier transform of the spherical measure.
 By  applying the Fourier inversion formula of the Euclidean space $\mathbb{R}^{d+1}$ for  $f$, we  express $\mathcal{A}_{j} f$ in (\ref{kcc1}) as
\begin{align*}
\mathcal{A}_{j} f(x,x_{d+1},t)
& =\int_{\mathbb{R}^{d+1}} e^{2\pi i \langle (x,x_{d+1}),(\xi,\xi_{d+1})\rangle}\widehat{d\sigma_{j}}\left(t(\xi+ \xi_{d+1} A(x))\right)\widehat{f}(\xi,\xi_{d+1})d\xi d\xi_{d+1}.
\end{align*}
Here $\widehat{d\sigma_j}$ is the Euclidean Fourier transform of $d\sigma_j$ of (\ref{kc1}) in $\mathbb{R}^{d}$ given by
\begin{align}\label{pp04}
 \widehat{d\sigma_{j}}\left(t\big(\xi+ \xi_{d+1} A(x)\big)\right)    =\widehat{d\sigma}\left(t \big(\xi+ \xi_{d+1} A(x)  \big) \right)\chi\left(\frac{t(\xi+ \xi_{d+1} A(x)}{2^j}\right).
 \end{align}
On the other hand, the Fourier transform $\widehat{d\sigma}$ of the sphere measure on $S^{d-1}$ is written in terms of the Bessel function $J_{(d-2)/2}$ as
$$\widehat{d\sigma}(\xi)= 2\pi J_{(d-2)/2}(2\pi |\xi|)|\xi|^{-(d-2)/2} $$
where  the Bessel function $ J_{(d-2)/2}$ for large $|\xi|$ has
  the asymptotic expansion of   $$ J_{(d-2)/2}(|\xi|)=|\xi|^{-1/2}\sum_{\ell=0}^{N} (a_\ell e^{  2\pi i|\xi|}+   b_\ell e^{  -2\pi i|\xi|}) |\xi|^{-\ell}+O(|\xi|^{-(N+\frac{3}{2})}).$$ 
 Thus,
  $\widehat{d\sigma_j}(\xi)= \widehat{d\sigma}(\xi)\chi\left(\frac{\xi}{2^j}\right)$  for $j\ge 1$ in (\ref{kc1})   is expressed as  
\begin{align}\label{kc3}
 \widehat{d\sigma_j}(\xi)=\sum_{\ell =0}^N\frac{( a_\ell e^{  2\pi i|\xi|}+   b_\ell e^{  -2\pi i|\xi|})\chi_\ell\left(\frac{\xi}{2^j}\right) }{2^{\ell j}}\frac{2\pi}{2^{j(d-1)/2}}+O\left( \frac{\tilde{\chi}\left(\frac{\xi}{2^j}\right)}{2^{(d-1) j/2} 2^{(N+\frac{1}{2})j}}\right) .
\end{align}
where $\chi_\ell (\xi):=\chi(\xi)/|\xi|^{\frac{d-1}{2}+\ell}$, which  is similar to $\chi(\xi)$ due to $\text{supp}(\chi)=\{|\xi|\approx 1\}$.
In the summation,  the first term ($\ell=0$)  is dominating, as the sum over $1\le \ell\le N$  with the error   term  behaves   better. So, we   work with only the first term   in  (\ref{kc3}) given by $$    \frac{e^{  2\pi i|\xi|} }{2^{j(d-1)/2}} \chi_0\left(\frac{\xi}{2^j}\right) $$   for computing
 $ \|\mathcal{M}_jf\|_{L^2(\mathbb{R}^{d+1})}$  in (\ref{kc4}) of Proposition \ref{prop44}. We can treat $e^{  -2\pi i|\xi|} $ similarly. Hence, we now replace the above average operator $  \mathcal{A}_j $ with the Fourier integral operator $  \frac{1}{2^{j(d-1)/2}}\mathcal{T}_{m_{j}},$ where  $\mathcal{T}_{m_{j}}:L^2(\mathbb{R}^{d+1})\rightarrow L^2(\mathbb{R}^{d+1}\times\mathbb{R}) $ is defined as
 \begin{align}
& \mathcal{T}_{m_{j}} f(x,x_{d+1},t) \nonumber\\
 &\qquad =  \int e^{2\pi i \langle (x,x_{d+1}),(\xi,\xi_{d+1})\rangle }m_{j} (x,x_{d+1},t,\xi,\xi_{d+1})\widehat{f}(\xi,\xi_{{d+1}})d\xi d\xi_{{d+1}}  \label{3hf}
 \end{align}
whose symbol is $$m_{j} (x,x_{d+1},t,\xi,\xi_{d+1})=   e^{  2\pi it|\xi+\xi_{d+1}A(x)|}  \chi \left(\frac{t(\xi+\xi_{d+1}A(x))}{2^j}\right) .$$
Furthermore,  to treat the supremum over $t>0$ conveniently as in \cite{MS}, we restrict   time variable $|t|\approx 2^{-k}$ to put a symbol $m_{j,k}$ as
\begin{align}\label{955t} 
m_{j,k}(x,x_{d+1},t,\xi,\xi_{d+1})= e^{2\pi i     t|\xi+\xi_{{d+1}}A(x)| } \chi\left(\frac{t(\xi+\xi_{{d+1}}A(x))}{2^j}\right)\chi(2^kt)  .
 \end{align} 
Associated with the sequence $ (m_{j,k})_k$,  we   set the maximal operator $  \mathfrak{M}_{j} $  defined by
 \begin{align} \label{115a}
   \mathfrak{M}_{j} f(x,x_{{d+1}})&:= \sup_{k\in\mathbb{Z}}\sup_{t>0} \left|2^{-j(d-1)/2}\mathcal{T}_{m_{j,k}}f(x,x_{{d+1}},t)\right|  \end{align} 
satisfying that   
\begin{align}\label{11p}
\| \mathcal{M}_j\|_{L^2(\mathbb{R}^{d+1})\rightarrow L^2(\mathbb{R}^{d+1})}\lesssim \|  \mathfrak{M}_{j}\|_{L^2(\mathbb{R}^{d+1})\rightarrow L^2(\mathbb{R}^{d+1})}.
\end{align}
        \subsection{Majorizing $L^\infty(\mathbb{R}_+)$ by $L^2_{1/2}(\mathbb{R}_+)$}
  Let $m_{j,k}$ be the symbol in (\ref{955t}).  Then  by applying the fundamental theorem of calculus to $|\mathcal{T}_{m_{j,k}}f(x,x_{{d+1}},t)|^{2}$ in (\ref{115a}), we obtain that
\begin{align}
\sup_{t>0} |\mathcal{T}_{m_{j,k}}f(x,x_{{d+1}},t)|^{2}&=\sup_{t>0}\int_{0}^t \partial_s  \left[\left(\mathcal{T}_{m_{j,k}}f(x,x_{{d+1}},s)  \overline{\mathcal{T}_{m_{j,k}}f(x,x_{{d+1}},s) } \right) \right]ds\nonumber\\
&\le  2\int_{0}^\infty \left|\partial_t \left(\mathcal{T}_{m_{j,k}}f(x,x_{{d+1}},t)  \right) \right|  \left|\mathcal{T}_{m_{j,k}}f(x,x_{{d+1}},t) \right| dt. \label{nv22}
\end{align}
We can   majorize  $\partial_t  \mathcal{T}_{m_{j,k}}f$ in the RHS of the above inequality as
\begin{align}\label{kc9}
|\partial_t  \mathcal{T}_{m_{j,k}}f(x,x_{{d+1}},t)|\lesssim 2^{j+k}|\mathcal{T}_{\tilde{m}_{j,k}}f(x,x_{{d+1}},t)|
\end{align}
where   $\tilde{m}_{j,k}$   a slight modification of $m_{j,k}$  with $\chi(u)$    replaced by $u\chi(u)$ or $\chi(u)/u$ in (\ref{955t}).
\begin{proof}[Proof of (\ref{kc9})]
It holds that  $\partial_t\mathcal{T}_{m_{j,k}}f=\mathcal{T}_{\partial_t m_{j,k}}f $ in (\ref{3hf}).  The $t$-derivative of $m_{j,k}$ in (\ref{955t}) consists of the following factors
\begin{align*}
\partial_t \left[e^{2\pi i    t|\xi+\xi_{{d+1}}A(x)| }\right]&=2\pi  i |\xi+\xi_{{d+1}}A(x)|  e^{2\pi i  t|\xi+\xi_{{d+1}}A(x)| }=O(2^{j+k})\\
\partial_t \left[\chi(2^kt)  \chi\left(\frac{|\xi+\xi_{{d+1}}A(x)| t}{2^j } \right)\right]&= O(2^k ).\end{align*}
From this, we write the   main additional factor multiplied to $m_{j,k}$ in the above as
$$ |\xi+\xi_{{d+1}}A(x)|=2^{j+k} \left[\left(\frac{|\xi+\xi_{{d+1}}A(x)| t}{2^j } \right) \frac{1}{(2^kt)}\right]. $$ 
Here $ \left[\left(\frac{t|\xi+\xi_{{d+1}}A(x)|  }{2^j } \right) \frac{1}{(2^kt)}\right]$   absorbed into $ \chi\left(\frac{t  |\xi+\xi_{{d+1}}A(x)| }{2^j}\right)\chi(2^kt) $ of $m_{j,k}$ in (\ref{955t}). This  verifies (\ref{kc9}).
 \end{proof}
From (\ref{115a})-(\ref{kc9}), it holds that
      \begin{align}
\|\mathfrak{M}_jf\|_{L^2(\mathbb{R}^{d+1})}  &\le 2^{-(d-1)j/2} 2^{j/2} \left(\sum_{k\in\mathbb{Z}}\left\|2^{k/2}\mathcal{T}_{\tilde{m}_{j,k}}f\right\|_{L^2(\mathbb{R}^{d+1}\times \mathbb{R})}^2\right)^{1/2}. \label{oj11}
\end{align} 
From now on, we regard $\tilde{m}_{j,k}$ in the RHS of (\ref{oj11}) as $m_{j,k}$ of  (\ref{955t}).
\subsection{Littlewood-Paley Decomposition}
\begin{definition}\label{de31}
Note that $\psi$  is supported in $|(\xi,\xi_{d+1})|\le1$ in $\mathbb{R}^{d+1}$ and
 $\psi(\xi,\xi_{d+1})\equiv 1$ on $|(\xi,\xi_{d+1})|\le1/2$. Using the non-isotropic dilation, we set
\begin{align}\label{10061}
\widehat{P_j}(\xi,\xi_{d+1})=\psi\left( \frac{\xi}{2^{j+1}},\frac{\xi_{d+1}}{2^{2(j+1)}}\right)- \psi\left( \frac{\xi}{2^{j}},\frac{\xi_{d+1}}{2^{2j}}\right),
\end{align}
where $\widehat{}$ is the Fourier transform in the Euclidean space $\mathbb{R}^{d+1}$.
Define the Littlewood-Paley projection $\mathcal{P}_jf$ for $f\in \mathcal{S}(\mathbb{R}^{d+1})$ by
\begin{align*}
\mathcal{P}_jf(y,y_{{d+1}})&=\int_{  (\eta,\eta_{d+1})\in \mathbb{R}^d\times \mathbb{R}}e^{2\pi i \langle (\eta,\eta_{{d+1}}),(y,y_{{d+1}}) \rangle}\widehat{P_j}(\eta+\eta_{d+1}A(y),\eta_{d+1})\widehat{f}(\eta,\eta_{{d+1}})d\eta d\eta_{{d+1}}. 
\end{align*}
\end{definition}
From $\sum_{\ell\in\mathbb{Z}} \mathcal{P}_{j+k+\ell}=Id$ and the triangle inequality in (\ref{oj11}), it holds that
\begin{align}
&\| \mathfrak{M}_jf \|_{L^2(\mathbb{R}^{{d+1}})}\nonumber\\
&\qquad  \lesssim 2^{j(1/2-(d-1)/2)} \sum_{\ell\in\mathbb{Z}} \left\| \left(\sum_{k\in \mathbb{Z}}|2^{k/2}\mathcal{T}_{m_{j,k}}\mathcal{P}_{j+k+\ell} f|^2\right)^{1/2}\right\|_{L^2(\mathbb{R}^{{d+1}}\times \mathbb{R})}.\label{0453}
\end{align}
From this with   (\ref{11p}), in order to prove Main Theorem  \ref{main3}, we need to   prove Proposition \ref{prop03}, and estiamte the operator norm $2^{k/2}\mathcal{T}_{m_{j,k}}  $ below. 
\begin{proposition}\label{prop03} 
For an arbitrary small $\epsilon>0$, it holds that
\begin{align*}
& \sum_{\ell\in\mathbb{Z}} \left\| \left(\sum_{k\in \mathbb{Z}}|2^{k/2}\mathcal{T}_{m_{j,k}}\mathcal{P}_{j+k+\ell} f|^2\right)^{1/2}\right\|_{L^2(\mathbb{R}^{{d+1}}\times \mathbb{R})}\\
&\qquad\qquad\qquad \lesssim  2^{\epsilon j} \sup_k  \left\| 2^{k/2}\mathcal{T}_{m_{j,k}}  \right\|_{ L^2(\mathbb{R}^{d+1})\rightarrow L^2(\mathbb{R}^{d+1}\times \mathbb{R})} \|f\|_{L^2(\mathbb{R}^{d+1})}.
\end{align*}
 \end{proposition}
 To prove  Proposition \ref{prop03},  
we shall use the following preliminary estimates.  
 \begin{lemma}\label{prop1}
Recall $\mathcal{T}_{m_{j,k}}$ in  (\ref{3hf}).  If $|\ell|\ge C j$ for a large $C>0$, then there is  $c>0$ such that
\begin{align}
\left\| 2^{k/2}\mathcal{T}_{m_{j,k}}\mathcal{P}_{j+k+\ell} \right\|_{ L^2(\mathbb{R}^{d+1})\rightarrow L^2(\mathbb{R}^{d+1}\times \mathbb{R})}  &\lesssim 2^{-c|\ell|}  \label{00453}
\end{align}
and
 \begin{align}
\left\|  \mathcal{P}_{k_1} \mathcal{P}_{k_2}^*  \right\|_{ L^2(\mathbb{R}^{d+1})\rightarrow L^2(\mathbb{R}^{d+1})} &\lesssim   2^{-c|k_1-k_2|}. \label{045312}
\end{align}
\end{lemma}
\begin{proof}[Proof of Lemma \ref{prop1}]
As its proof is standard, we place  it  in the appendix (Sec 11.3).
\end{proof}
  \begin{lemma}\label{lem44}
Let $M_1,M_2$ be two measure spaces. Suppose that  there is a family $\{T_k\}_{k\in\mathbb{Z}}$ of the operators   $T_k:L^2(M_1)\rightarrow L^2(M_2)$ whose operator norm is denoted by $\|T_k\|_{op}$.  If $\|T_{k_1}T_{k_2}^*\|_{op}\le C(|k_1-k_2|) $ with $\sum_{k\in \mathbb{Z}}  \sqrt{C(|k|)}$ being finite, then   for all $f\in L^2(M_1)$,
\begin{align}
\left\|  \left( \sum_{k\in \mathbb{Z}} |T_k f|^2\right)^{1/2}\right\|_{L^2(M_2)}\lesssim     \left(  \sup_k \|T_k\|_{op} \sum_{k\in \mathbb{Z}}   \sqrt{C(|k|)}\right)^{1/2}  \|f\|_{L^2(M_1)}.\label{8kg}
\end{align}
\end{lemma}
\begin{proof}[Proof of Lemma \ref{lem44}]
Write the LHS of (\ref{8kg}) as the square root of $  \left\langle\sum_{k} T_k^*T_kf,f\right\rangle.  $  For this summation $\sum_k$,  we can apply the  Cotlar-Stein lemma   with $$\|(T_{k_1}^*T_{k_1})(  T_{k_2}^*T_{k_2})\|_{op}\lesssim  \left( \sup_k \|T_k\|_{op} \right)^2 \| T_{k_1}  T_{k_2}^* \| $$   
to obtain (\ref{8kg}).
\end{proof}
\begin{proof}[Proof of Proposition  \ref{prop03}]
Denote   various operator norms by  $\| \cdot\|_{op}$.  In  Lemma \ref{lem44},   put  
\begin{align*} 
T_k:= 2^{k/2}\mathcal{T}_{m_{j,k}}\mathcal{P}_{j+k+\ell} .
\end{align*} 
By this and (\ref{045312}), we control $\left\| T_{k_1} T_{k_2}^*  \right\|_{ L^2(\mathbb{R}^{d+1}\times\mathbb{R})\rightarrow L^2(\mathbb{R}^{d+1}\times\mathbb{R})}$ by
\begin{align*}
  \left\|  2^{k_2/2}\mathcal{T}_{m_{j,k_1 }}  \right\|_{op}\,\cdot \left\| \mathcal{P}_{j+k_1+\ell} \mathcal{P}_{j+k_2+\ell}^* \right\|_{op}\,\cdot \left\| [2^{k_1/2}\mathcal{T}_{m_{j,k_2}}]^* \right\|_{op} \lesssim   2^{-c|k_1-k_2|}\sup_k\left\|2^{k/2}\mathcal{T}_{m_{j,k}}\right\|_{op} ^2 .
 \end{align*}
From this, we  apply  Lemma \ref{lem44} with $C(|k|)= 2^{-c|k|} \sup_k\left\|2^{k/2}\mathcal{T}_{m_{j,k}}\right\|_{op} ^2$ to obtain
\begin{align*}
&  \left\|  \left( \sum_{k\in \mathbb{Z}} |T_k f|^2\right)^{1/2}\right\|_{L^2(\mathbb{R}^{d+1}\times\mathbb{R})} \lesssim \left( \sup_k\left\|2^{k/2}\mathcal{T}_{m_{j,k}}\mathcal{P}_{j+k+\ell} \right\|_{op} \sum_{k\in \mathbb{Z}}  \sqrt{C(|k|)}\right)^{1/2} \left\|  f\right\|_{L^2(\mathbb{R}^{d+1})}\\
 &\qquad\qquad\qquad\lesssim   \left(\sup_k\left\|2^{k/2}\mathcal{T}_{m_{j,k}}\mathcal{P}_{j+k+\ell} \right\|_{op}\right)^{1/2} \left(\sup_k\left\|2^{k/2}\mathcal{T}_{m_{j,k}} \right\|_{op}\right)^{1/2}  \left\| f\right\|_{L^2(\mathbb{R}^{d+1})}.
 \end{align*}
It follows from this and (\ref{00453})   that  for all $f\in L^2(\mathbb{R}^{d+1}$ with $c>0$ independent of $f$, 
\begin{align*} 
&\left\| \left(\sum_{k\in \mathbb{Z}}|2^{k/2}\mathcal{T}_{m_{j,k}}\mathcal{P}_{j+k+\ell} f|^2\right)^{1/2}\right\|_{L^2(\mathbb{R}^{{d+1}}\times \mathbb{R})}  \\
&\qquad\lesssim \begin{cases} 2^{-c|\ell|/2} \left(\sup_k\left\|2^{k/2}\mathcal{T}_{m_{j,k}} \right\|_{op}\right)^{1/2}    \|f\|_{L^2(\mathbb{R}^{d+1})}  \ \text{if $|\ell|\ge Cj$}\\
\left(\sup_k  \left\| 2^{k/2}\mathcal{T}_{m_{j,k}}  \right\|_{op} \right) \|f\|_{L^2(\mathbb{R}^{d+1})} \ \text{if $|\ell|< Cj$.}
   \end{cases}
\end{align*}  
Since $ \left\| \mathcal{T}_{m_{j,0}} \right\|_{op}\gtrsim 1$ (take $\widehat{f}(\xi,\xi_{d+1})=\chi_{B}(\xi,\xi_{d+1})$ in (\ref{i3}) with a small balls $B$), we can replace the exponent $1/2$ with $1$ in the first line of the above inequality.
By summing over $|\ell|\ge Cj$ and $|\ell|<Cj$ in the above, we obtain   Proposition \ref{prop03}. 
\end{proof} 
 
\subsection{Dilation and Localization}
Our matters   now is to estimate $\left\| 2^{k/2}\mathcal{T}_{m_{j,k}}  \right\|_{op}$   in Proposition \ref{prop03}. Then the next lemma tells that  it suffices  to fix  $k=0$ for $\mathcal{T}_{m_{j,k}}f$. 
 \begin{lemma}[Dilation]\label{prop23}
For every  $(j,k)\in\mathbb{Z}_+\times \mathbb{Z}$ and  $p\ge 1$,   it holds that 
\begin{align}\label{yb1}
\left\|2^{k/p}\mathcal{T}_{m_{j,k}}\right\|_{L^p(\mathbb{R}^{{d+1}})\rightarrow L^p(\mathbb{R}^{{d+1}}\times\mathbb{R})} = \left\|\mathcal{T}_{m_{j,0}}\right\|_{L^p(\mathbb{R}^d\times\mathbb{R})\rightarrow L^p(\mathbb{R}^{d+1}\times\mathbb{R})}
\end{align} 
where $C$ is independent of $j,k$, and   $\mathcal{T}_{m_{j,0}}f(x,x_{d+1},t) $  is expressed as
 \begin{align}\label{i3}
 \mathcal{T}_{m_{j,0}}f(x,x_{d+1},t) 
 &=\int
e^{2\pi i\left( \langle (x,x_{d+1}),(\xi,\xi_{d+1})\rangle+t|\xi+\xi_{d+1}A(x)|\right)}\nonumber \\
&\times \chi(t) \chi\left(\frac{t|\xi+\xi_{d+1}A(x)|}{2^j}\right) \widehat{f
 }(\xi,\xi_{{d+1}})d\xi d\xi_{{d+1}}.
\end{align}
 \end{lemma}
  \begin{proof}
 Our proof is   based on non-isotropic dilations  in the appendix (Sec  \ref{app11.2}).
 \end{proof}

 \subsection{Reduction to Uniform $L^2$ estimates}\label{Sec46}
Put $\xi_{{d+1}}=\lambda$ and   rewrite   (\ref{i3}) as 
\begin{align*}
&\mathcal{T}_{m_{j,0}}f(x,x_{{d+1}},t)\\
&\qquad= \int_{\mathbb{R}}
  \left[\chi\left(t\right)   \int_{\mathbb{R}^d}
e^{2\pi i\left(\langle x, \xi \rangle+t|\xi+\lambda A(x)|\right)} 
\chi\left(\frac{t|\xi+\lambda A(x)|}{2^j}\right)\widehat{f
 }(\xi,\lambda)d\xi   \right]  e^{2\pi i x_{d+1} \lambda}d\lambda.
\end{align*}
Let $ \widehat{g_{\lambda}}(\xi)=\lambda^{d/2}\widehat{f}(\lambda \xi,\lambda) $. Then  by using the change of variable $\xi\rightarrow \lambda \xi$ for the above integral, 
$$ \mathcal{T}_{m_{j,0}}f(x,x_{{d+1}},t)= \int_{\mathbb{R}}
 \left[ \mathcal{T}_{j}^{\lambda}g_{\lambda} (x,t ) \right]   e^{2\pi i x_{d+1} \lambda}d\lambda  $$
 where  
\begin{align}
\label{5rg}\mathcal{T}_{j}^{\lambda}g (x,t )&= \lambda^{d/2}  \chi\left(t\right)  \int_{\mathbb{R}^d}
e^{2\pi i\lambda \left(\langle x,  \xi \rangle+t| \xi+A(x)|\right)}
\chi\left(\frac{\lambda t|\xi+A(x)|}{2^j}\right) \widehat{g}(  \xi)d\xi \end{align} 
which is  same   as (\ref{5rpp}).
 Here we  note that $\widehat{g_{\lambda}}(\cdot)$ is a $d$-variable function for each fixed $\lambda\in\mathbb{R}$.
 Then
by applying the Plancherel theorem with respect to $x_{{d+1}}$ variable, whose frequency is denoted by  $\lambda$,   
\begin{align*}
\int\left[\int_{\mathbb{R}} |\mathcal{T}_{m_{j,0}}f(x,x_{{d+1}},t)|^2dx_{d+1} \right]dxdt  &= \int\left[\int_{\mathbb{R}} |\mathcal{T}_{j}^{\lambda}g_{\lambda} (x,t )|^2 d\lambda \right]dxdt\\
&\le \int_{\mathbb{R}} \| \mathcal{T}_{j}^{\lambda}\|_{L^2(\mathbb{R}^d)\rightarrow L^2(\mathbb{R}^{d}\times \mathbb{R})}^2  \|\widehat{g_{\lambda}}(\cdot)\|_{L^2(d\xi)}^2 d\lambda 
\end{align*}
where $\int \|\widehat{g_{\lambda}}(\cdot)\|_{L^2(d\xi)}^2d\lambda=\|f\|_{L^2(\mathbb{R}^{d+1})}^2$.
Thus, it holds that  \begin{align}\label{011}
\left\|\mathcal{T}_{m_{j,0}}\right\|_{L^2(\mathbb{R}^{d+1})\rightarrow L^2( \mathbb{R}^{{d+1}}\times\mathbb{R})}\le  \sup_{\lambda} \left\| \mathcal{T}_{j}^{\lambda} \right\|_{L^2(\mathbb{R}^d)\rightarrow L^2(\mathbb{R}^d\times \mathbb{R})}.
  \end{align} 
Therefore, we conclude that    the estimate of  $\left\|\mathcal{T}_{m_{j,0}}\right\|_{L^2(\mathbb{R}^{d+1})\rightarrow L^2(\mathbb{R}^{d+1}\times \mathbb{R})}$ is reduced to the uniform estimate  in real parameters $\lambda$ of the above oscillatory integral operators $\mathcal{T}_{j}^{\lambda}$ mapping $g\in L^2(\mathbb{R}^{d})\rightarrow \mathcal{T}_{j}^{\lambda}g\in L^2(\mathbb{R}^{d}\times \mathbb{R})$.
  \begin{proof}[Proof of Main Theorem \ref{main3}]
From    (\ref{115a}),(\ref{0453})   and Proposition  \ref{prop03}  together with  (\ref{yb1}) and (\ref{011}),  we obtain that  
\begin{align*}
 \left\|\mathcal{M}_{j} \right\|_{L^2(\mathbb{R}^{d+1})\rightarrow L^2(\mathbb{R}^{d+1})} &\lesssim\left\| \mathfrak{M}_j  \right\|_{L^2(\mathbb{R}^{{d+1}})\rightarrow L^2(\mathbb{R}^{{d+1}})}
 \\
 &\lesssim 2^{j(1/2-(d-1)/2)} 2^{\epsilon j}
\sup_k\left\|2^{k/2}\mathcal{T}_{m_{j,k}}\right\|_{L^2(\mathbb{R}^{{d+1}})\rightarrow L^2(\mathbb{R}^{{d+1}}\times\mathbb{R})}\\
&  \lesssim  2^{j(1/2-(d-1)/2)} 2^{\epsilon j} \left\|\mathcal{T}_{m_{j,0}}\right\|_{L^2([-1,1]^d\times\mathbb{R})\rightarrow L^2([-10,10]^d\times\mathbb{R}]\times I)} \\
&\le2^{j(1/2-(d-1)/2)} 2^{\epsilon j}  \sup_{\lambda} \left\| \mathcal{T}_{j}^{\lambda} \right\|_{L^2(\mathbb{R}^d)\rightarrow L^2(\mathbb{R}^d\times \mathbb{R})} \\ 
&\le 2^{j(1/2-(d-1)/2)} 2^{\epsilon j} 2^{c(A)j/2}.
\end{align*}
From this combined with Proposition  \ref{prop44}, we have
proved the first part of  Main Theorem \ref{main3}, i.e.,   (\ref{s100}) $\Rightarrow $ (\ref{s81}) and (\ref{s82}).
 \end{proof}

 \section{Proof of Main Theorem \ref{main3}; Reduction to  $\mathcal{T}_{\rm{annulus}}^{\lambda}$}\label{sec5}  We prove the remaining part of Main Theorem \ref{main3}, i.e., that if $\det(A)\ne 0$,
 $$\|\mathcal{T}_{\rm{annulus}}^{\lambda}\|_{L^2(\mathbb{R}^d)\rightarrow L^2(\mathbb{R}^d\times \mathbb{R})}\lesssim_{\epsilon} \lambda^{c(A)/2}\ \text{implies}\ \|\mathcal{T}_{j}^{\lambda}\|_{L^2(\mathbb{R}^d)\rightarrow L^2(\mathbb{R}^d\times \mathbb{R})}\lesssim_{\epsilon} 2^{c(A)j/2}\  $$
 for all $j\in\mathbb{Z}_+$. We shall show this  by  applying the H\"{o}rmander theorem to an appropriate dilations  of $\mathcal{T}_{j}^{\lambda}$.
  Let $A\in M_{d\times d}(\mathbb{R})$ be invertible.  
The change of variable $x\rightarrow A^{-1}(x)$ for   $\mathcal{T}_{j}^{\lambda}g$   in  (\ref{5rg}), yields\begin{align*}
\mathcal{T}_{j}^{\lambda}g(A^{-1}(x),t )& =\frac{ \lambda^{d/2}  \chi\left(t\right) }{\det(A)} \int_{\mathbb{R}^d}
e^{2\pi i\lambda \left(\langle A^{-1}(x), \xi \rangle+t| \xi+x|\right)}
\chi\left(\frac{\lambda t|\xi+x|}{2^j}\right) \widehat{g
 }(  \xi)d\xi.
  \end{align*}
     In view of (iii) of Lemma \ref{lem21}, we are allowed to replace $A^{-1}$ with $A$ in the above oscillatory part  and rewrite it as
\begin{align}\label{5ppa}
  \mathcal{T}_{j}^{\lambda}g(x,t )  = \lambda^{d/2}  \chi\left(t\right) \int_{\mathbb{R}^d}
e^{2\pi i\lambda \left(\langle A(x), \xi \rangle+t| \xi+x|\right)}
\chi\left(\frac{\lambda t|\xi+x|}{2^j}\right) \widehat{g
 }(  \xi)d\xi.
\end{align}
Since $ \mathcal{T}_{\rm{annulus}}^{\lambda} $ in (\ref{apl}) is  defined exactly as  the operator $ \mathcal{T}_{j}^{\lambda} $ with $2^j=\lambda$,
the main part of our proof is to treat $  \mathcal{T}_{j}^{\lambda} $ for the case $\lambda\not\approx 2^j$  as it follows.
 \begin{proposition}[Extreme $\lambda$] \label{pr02}
Suppose that   $A\in M_{d\times d}(\mathbb{R})$ is invertible.   Then, there is $C\gg 1$ such that
\begin{align}
 \left\|\mathcal{T}_j^\lambda   \right\|_{L^2(\mathbb{R}^d)\rightarrow L^2(\mathbb{R}^{d+1}) }&\lesssim 1\   \ \text{ if $\left|\frac{2^j}{\lambda}\right|\ge C$, }
\label{546}\\
 \left\|\mathcal{T}_j^\lambda  \right\|_{L^2(\mathbb{R}^d) \rightarrow L^2(\mathbb{R}^{d+1}) }&\lesssim \left|\frac{2^j}{\lambda}\right|^{d/2} \  \  \text{ if $\left|\frac{2^j}{\lambda}\right|\le  1/C$.}\label{545}
\end{align}
\end{proposition}

\begin{proof}[Proof of (\ref{546})]
By using the $L^2$ norm invariance of the dilation $\left(\frac{2^j}{\lambda}\right)^{d/2} \mathcal{T}_{j}^{\lambda}g(\frac{2^j}{\lambda}x,t ) $  and
  the change of variable $\xi\rightarrow \left(\frac{2^j}{\lambda}\right) \xi$ in (\ref{5ppa}), it suffices to deal with the operator
\begin{align} \label{54ss}
 \mathcal{T}_{j}^{\lambda}g( x,t )  =\left(\frac{2^j}{\lambda}\right)^{d}  \lambda^{d/2}  \chi\left(t\right)  \int_{\mathbb{R}^d}
e^{2\pi i\lambda  \left[\left(\frac{2^j}{\lambda}\right)^2  \langle A(x), \xi \rangle+  \left(\frac{2^j}{\lambda}\right)t| \xi+x|  \right]  }
\chi\left(  t|\xi+x|\right) \widehat{g
 }(  \xi)d\xi.
\end{align}
Observe that the support  of the above integral is contained in $\{|x+\xi|\le 4\}$. Thus the support restriction  $|x-n|\le 1$ implies $|\xi+n|\le 5$. From this,
we decompose  $$ \mathcal{T}_{j}^{\lambda}g( x,t )=\sum_{n\in\mathbb{Z}^d}\left(\frac{2^j}{\lambda}\right)^{d}  \lambda^{d/2}  \int e^{i[\,\cdot\,]} \psi(x-n)\psi\left(\frac{\xi+n}{5}\right)\chi(t)\chi\left(  t|\xi+x|\right) \widehat{g
 }(  \xi)d\xi. $$  
Thus, by localization principle, it suffices to deal with one fixed $n\in\mathbb{Z}^d$ in the above summation.
Apply the change of variables $x\rightarrow x+n$ and $\xi\rightarrow \xi-n$. Then the above integral ($n^{th}$ term) becomes $$e^{-2\pi i \lambda  \left[\left(\frac{2^j}{\lambda}\right)^2  (\langle A(x),n\rangle+\langle A(n),n\rangle\right]}\widetilde{\mathcal{T}}_{j}^{\lambda}h( x,t). $$
Here 
\begin{align} \label{5pp}
\widetilde{\mathcal{T}}_{j}^{\lambda}h( x,t ) = \left(\frac{2^j}{\lambda}\right)^{d}  \lambda^{d/2}  \chi\left(t\right)  \int_{\mathbb{R}^d}
e^{2\pi i\lambda  \left(\frac{2^j}{\lambda}\right)^2 \phi(x,t,\xi) }\psi(x)\psi(\xi/5)
\chi\left(  t|\xi+x|\right) h(\xi)d\xi
\end{align}
and $$h(\xi)=e^{2\pi i \lambda  \left[\left(\frac{2^j}{\lambda}\right)^2  \langle A(n),\xi\rangle\right]}\widehat{g
 }(  \xi-n)\ \text{and}\    \phi(x,t,\xi)= \langle A(x), \xi \rangle+  \left(\frac{2^j}{\lambda}\right)^{-1}t| \xi+x|. $$ 
 Fix $|t|\approx 1$. Then 
  the determinant of the $d\times d$ mixed hessian matrix of   $\phi(x,t,\xi)$ is
\begin{align*}
\det( [\phi_{x_i\xi_j}''(x,t,\xi)])=\det(A)+O\left(\left(\frac{2^j}{\lambda}\right)^{-1}\right)\approx \det(A)\gtrsim 1
\end{align*}
which follows from the multilinearity of $\det$ combined with   $\left(\frac{2^j}{\lambda}\right)^{-1}\ll 1$. This enables us to apply the H\"{o}rmander  theorem for (\ref{5pp}) to obtain that
 $$\|  [\widetilde{\mathcal{T}}_{j}^{\lambda}g](\cdot,t) \|_{L^2(\mathbb{R}^d)}\le C\left(\frac{2^j}{\lambda}\right)^{d}  \lambda^{d/2} \left(\lambda \left(\frac{2^j}{\lambda}\right)^2\right)^{-d/2}  \|g\|_{L^2(\mathbb{R}^d)}=C\|g\|_{L^2(\mathbb{R}^d)} $$
with $C$  independent of $t$. So   integrate $\|  [\widetilde{\mathcal{T}}_{j}^{\lambda}g](\cdot,t) \|_{L^2(\mathbb{R}^d)}^2$ along $dt$ to obtain (\ref{546}).
\end{proof}
\begin{remark}\label{rek1}
If $2^j\gg \lambda$, we can  prove (\ref{546}) for $\mathcal{T}_j^\lambda$ in   (\ref{5rg}) without   $\text{rank}(A) =d$ since the phase function
$\left(\frac{2^j}{\lambda}\right)^2 \left( \langle x, \xi \rangle+  \left(\frac{2^j}{\lambda}\right)^{-1}t| \xi+A(x)| \right)$ has non-degenerate hessian for a fixed $t$.
 See    Proposition \ref{prop111} in the appendix.
\end{remark}

\begin{lemma}\label{lem977}
Let $F(x,t,y)=t|x+y|$. 
On the region $\{(x,t,y): |x_k+y_k|\ge |x+y|/d, |t|\approx 1\ \text{and}\ |x+y|\approx 1\}$, it holds that
$$\det [F''_{(x_1\cdots x_{k-1}x_{k+1}\cdots x_d t)(y_1\cdots y_d)}(x,t,y)]  \approx \frac{t^{d-1}}{|x+y|^{d-1}}.$$
\end{lemma}
\begin{proof}
Let $u=(x+y)$ and $u_i=x_i+y_i$. The   $(d+1)\times d$ mixed hessian matrix of $F$ is
$$[ F_{(x_1\cdots x_d t)(y_1\cdots y_d)}''(x,t,y)]=
  \left(\begin{matrix}
F_{x_1y_1}&\cdots&F_{x_1y_d} \\
F_{x_2y_1}&\cdots&F_{x_2y_d}\\
\vdots\\
F_{x_dy_1}&\cdots&F_{x_dy_d}\\
F_{ty_1}&\cdots&F_{ty_d}\\
 \end{matrix}  \right)=  \left(\begin{matrix}
t|u|^{-3}(\alpha_{1j}) \\
t|u|^{-3}(\alpha_{2j})\\
\vdots\\
t|u|^{-3}(\alpha_{dj})\\
|u|^{-1}(\alpha_{(d+1)j})\\
 \end{matrix}  \right)
$$
where for each fixed $i=1,\cdots,d$,
$$\alpha_{ij}=\begin{cases} -u_iu_j\ \text{if}\ j\ne i\\
|u|^2-u_i^2\ \text{if}\ j=i
\end{cases}\ \text{and}\ \alpha_{(d+1)j}=u_j.$$
Fix $k=1,\cdots,d$. If $u_k=0$, then $(\alpha_{kj})_{j=1}^d=(0,\cdots,0,|u|^2,0,\cdots,0)$ where $|u|^2$ is located on the $k^{th}$ component.
If $u_k\ne 0$, then for the $d\times (d+1)$ matrix $(\alpha_{ij})$, we
apply an elementary row operation replacing $[k^{th}\ \text{row}]$ with $[k^{th}\ \text{row}+u_k\times (d+1)^{th}\ \text{row}]$ as $$(\alpha_{kj})_{j=1}^d+u_k(\alpha_{(d+1)j})_{j=1}^d=(0,\cdots,0,|u|^2,0,\cdots,0)   $$
where $  |u|^2 $ is located on the $k^{th}$ component. Thus it holds that
$$\left(\begin{matrix}
 (\alpha_{1j}) \\
 (\alpha_{2j})\\
\vdots\\
 (\alpha_{dj})\\
 (\alpha_{(d+1)j})\\
 \end{matrix}  \right)
\sim  
\left(\begin{matrix}
|u|^2&0&0 &0 \\
0&|u|^2&0 &0\\
\vdots&\vdots &\ddots&\vdots\\
0&0&0&|u|^2\\
u_1&\cdots &u_{d-1}&u_d\\
 \end{matrix}  \right).
 $$
From this and the support condition $|u_k|\gtrsim |u|$,  we compute the determinant of the above mixed hessian matrix after removing the $k^{th}$ row to obtain that
$$\det [F''_{(x_1\cdots x_{k-1}x_{k+1}\cdots x_d t)(y_1\cdots y_d)}(x,t,y)]=\frac{u_kt^{d-1}|u|^{2(d-1)}}{|u|\cdot |u|^{3(d-1)}}\approx \frac{t^{d-1}}{|u|^{d-1}}.$$
This finish the proof of Lemma \ref{lem977}.
\end{proof}
\begin{proof}[Proof of (\ref{545})]
This is the case $    \frac{2^j}{\lambda}\ll 1$. By the similar localization as above, it suffices to work with
\begin{align}\label{445}
 \mathcal{T}_{j}^{\lambda}g( x,t )  =\left(\frac{2^j}{\lambda}\right)^{d}  \lambda^{d/2}  \chi\left(t\right)  \int_{\mathbb{R}^d}
e^{2\pi i\lambda \frac{2^j}{\lambda}  \phi(x,t,\xi)   }
\psi(x)\psi(\xi/5)\chi\left(  t|\xi+x|\right) \widehat{g
 }(  \xi)d\xi
\end{align}
where 
$$ \phi(x,t,\xi)=\frac{2^j}{\lambda}  \langle A(x), \xi \rangle+ t| \xi+x|. $$
It suffices to work with the restriction $|\xi_k+x_k|\ge |\xi+x|/d$ by inserting the cutoff function
$\psi\left(\frac{|\xi+x|}{d|\xi_k+x_k|}\right)$. On the support of this cutoff function for $u=x+\xi$ and $u_k=x_k+\xi_k$ in (\ref{445}), we apply Lemma \ref{lem977} combined with the multi-linearity of $\det$ to obtain that for a fixed $x_k$,
$$ \det [ \phi'' _{(x_1\cdots x_{k-1}x_{k+1}\cdots x_d t)(\xi_1\cdots \xi_d)}(x,t,\xi)]=O\left(\frac{2^j}{\lambda}\right) + \frac{t^{d-1}}{|u|^{d-1}}   \approx 1.$$
This enables us to apply the H\"{o}rmander  theorem to obtain that
 $$\|  [\mathcal{T}_{j}^{\lambda}g](x_k,\cdot) \|_{L^2(\mathbb{R}^d)}\le C\left(\frac{2^j}{\lambda}\right)^{d}  \lambda^{d/2}   2^{-jd/2}  \|g\|_{L^2(\mathbb{R}^d)}=C\left(\frac{2^j}{\lambda}\right)^{d/2} \|g\|_{L^2(\mathbb{R}^d)}. $$
 Since $C$ is independent of $x_k$, we integrate $\|  [\mathcal{T}_{j}^{\lambda}g](\cdot,x_k) \|_{L^2(\mathbb{R}^d)}^2$ with respect to $dx_k$  in the support of (\ref{445}) to obtain (\ref{545}).
\end{proof}
Therefore we finished the proof of Proposition \ref{pr02}. 
\subsection{Reduction to $\lambda=2^j$}
Suppose that $\text{rank}(A)=d$.  We have finished the proof of   the proposition \ref{pr02} treating $\lambda\not\approx 2^j$. So as to estimate $\|\mathcal{T}^\lambda_j\|_{op}$,   there remains the   case $\lambda\approx 2^j$   in (\ref{5pp}).  Apply the dilation invariance  under the norms of $L^2(dx)$ and $L^2(d\xi)$ as $$\left(\frac{2^j}{\lambda}\right)^{d/2}\mathcal{T}_j^\lambda g\left( \frac{2^j}{\lambda}x,t  \right)\ \text{and}\ \widehat{h}(\xi)=\left(\frac{2^j}{\lambda}\right)^{d/2}\widehat{g}\left( \frac{2^j}{\lambda}\xi \right).$$   Then by the change of variable $\xi\rightarrow \frac{2^j}{\lambda}\xi$  in  (\ref{5pp}), we work with  the $L^2$ estimate for
 \begin{align*}
\left(\frac{2^j}{\lambda}\right)^{d/2}\mathcal{T}_j^\lambda g\left( \frac{2^j}{\lambda}x,t  \right)&=  \left(\frac{2^j}{\lambda}\right)^{d/2} \lambda^{d/2}  \chi\left(t\right)  \int_{\mathbb{R}^d}
e^{2\pi i 2^j \left(\left\langle  A(x), \xi \right\rangle+t| \xi+x|\right)}
\chi\left(  t|\xi+x| \right) \widehat{h
 }(  \xi)d\xi.  \end{align*} 
Due to $\frac{2^j}{\lambda}\approx 1$,  it suffices to   fix the number $2^j=\lambda  $ for the $L^2$ estimation  to rewrite it as
\begin{align*}
\mathcal{T}_{\rm{annulus}}^\lambda g(x,t)= \lambda^{d  /2}  \chi\left(t\right)    \int_{\mathbb{R}^d}
e^{2\pi i \lambda \left(\langle A(x), \xi \rangle+t| \xi+x|\right)}
\chi\left(  t|\xi+x| \right) \widehat{g
 }(  \xi)d\xi
 \end{align*}
 where we can insert $\psi(x) $ to localize $|x|\lesssim 1$ as in (\ref{5pp}). This operator is what we  have considered in (\ref{apl}). Therefore,  when $A$ invertible $A$,
  the estimate of  (\ref{s100}) follows from  \begin{align*} 
  \|\mathcal{T}_{\rm{annulus}}^{\lambda}\|_{L^2(\mathbb{R}^d)\rightarrow L^2(\mathbb{R}^d\times \mathbb{R})}\le C\lambda^{\epsilon} \lambda^{ c(A)/2} \ \text{ for  all large $\lambda\gg 1$ }
  \end{align*}
which is (\ref{s83}). This completes the proof of Main Theorem \ref{main3}.

 \subsection{Statement of the Main Estimates} 
We have proved Main Theorem \ref{main3} stating that  the sufficient part of Main Theorems 1 and 2 follows from the estimate (\ref{s100}) or simply (\ref{s83}).
  We now state the  estimates regarding  (\ref{s100}) and (\ref{s83}).  At first, the spherical maximal   theorem   (Main Theorem 2) follows from Theorem \ref{th44} below   
 \begin{theorem}\label{th44}
  If all eigenvalues of $A\in M_{d\times d}(\mathbb{R})$  has nonzero imaginary part, then  it holds 
\begin{align*} 
 \left\| \mathcal{T}_{\rm{annulus}}^\lambda  \right\|_{ L^2(\mathbb{R}^{d})\rightarrow L^2(\mathbb{R}^{d}\times \mathbb{R})} &\lesssim_{\epsilon}   \lambda^{0}\ \text{for  $\lambda\ge 1$. } 
 \end{align*}
\end{theorem}
We can obtain the annulus maximal result for  $\text{rank}(A)=2$ in Main Theorem 1     from 
\begin{theorem}\label{th4}
Let $A\in M_{2\times 2}(\mathbb{R})$  with $\text{rank}(A)= 2$ and $\lambda\ge 1$. Then it holds that
\begin{align*}
\text{If $\text{rank}\left(JA+(JA)^T\right)=2$,\ then}\  &\left\| \mathcal{T}_{\rm{annulus}}^\lambda  \right\|_{ L^2(\mathbb{R}^{2})\rightarrow L^2(\mathbb{R}^{2}\times \mathbb{R})}  \lesssim_{\epsilon} \lambda^{0},
 \\
  \text{If $\text{rank}\left(JA+(JA)^T\right)=0$,\ then}\  &\left\| \mathcal{T}_{\rm{annulus}}^\lambda  \right\|_{ L^2(\mathbb{R}^{2})\rightarrow L^2(\mathbb{R}^{2}\times \mathbb{R})}  \lesssim_{\epsilon} \lambda^{1/2},
 \\
    \text{If $\text{rank}\left(JA+(JA)^T\right)=1$,\ then}\  &\left\| \mathcal{T}_{\rm{annulus}}^\lambda  \right\|_{ L^2(\mathbb{R}^{2})\rightarrow L^2(\mathbb{R}^{2}\times \mathbb{R})}  \lesssim_{\epsilon} \lambda^{1/6}. 
  \end{align*}
These are the growth rates    $c(A)/2=0,1/6,1$  in (\ref{314a}).
  \end{theorem}
Finally,   the annulus maximal case of $\text{rank}(A)=1$ (Main Theorem 1)  follows from 
  \begin{theorem}\label{th441}  
Let $\mathcal{T}_j^{\lambda}$  be defined as in (\ref{5rpp}). Suppose that  $\text{rank}(A)=1$.    Then
    \begin{align*}
 \left\| \mathcal{T}_j^{\lambda} \right\|_{L^2(\mathbb{R}^2 )\rightarrow L^2(\mathbb{R}^2\times\mathbb{R})} &\lesssim 2^{\epsilon j}\ \text{uniformly in $\lambda\in\mathbb{R}$}.
\end{align*}
This is the growth rate $c(A)/2=0$ in  (\ref{314a}).
\end{theorem}

 \section{Proof of Theorem \ref{th44} } \label{secc6}
 
 \subsection{Kernel of  $[\mathcal{T}_{\rm{annulus}}^\lambda]^* \mathcal{T}_{\rm{annulus}}^\lambda$}
 Write the  integral kernel $K(\xi,\eta)$ of $[\mathcal{T}_{\rm{annulus}}^\lambda]^*\mathcal{T}_{\rm{annulus}}^\lambda$  of  (\ref{5pp}) as in (\ref{28n})
\begin{align}\label{10029}
K(\xi,\eta) &=  \lambda^{d} \int_{\mathbb{R}^{d+1}} e^{i\lambda \Phi(x,t,\xi,\eta) }\Psi(x,t,\xi,\eta) dxdt
 \end{align}
where  
\begin{itemize}
\item the phase function $\Phi(x,t,\xi,\eta)=\langle A^T(\xi-\eta), x\rangle +t\left(|\xi+x|-|\eta+x|\right)$,
\item   the amplitude $\Psi(x,t,\xi,\eta)= \chi(t)\psi\left(x\right)  \chi\left( t|\xi+x| \right) 
 \chi\left( t|\eta+x| \right)    $
 \end{itemize}
have the derivatives  
\begin{align*}
& \nabla_{x} \Phi(x,t,\xi,\eta)  =  A^T(\xi-\eta) + t \left( \frac{\xi+x}{|\xi+x|}-  \frac{\eta+x}{|\eta+x|} \right)\ \text{and}\
  \partial_{t} \Phi(x,t,\xi,\eta)  =   |\xi+x|-|\eta+x|,   \\
 &\langle v, \nabla_x\rangle^m \left(\Psi(x,t,\xi,\eta) \right) =O(1)\ \text{for $|v|=1$}\ \text{and}\ 
 (\partial_t)^m   \left(\Psi(x,t,\xi,\eta) \right) =O(1).  
 \end{align*}  
 We apply Schur's test for the operator  $[\mathcal{T}_{\rm{annulus}}^\lambda]^*\mathcal{T}_{\rm{annulus}}^\lambda$, estimating the $L^1$ norm of the kernel along one side with the other  fixed,
 $$\int |K(\xi,\eta)|d\xi\  \ \text{and}\ \int |K(\xi,\eta)|d\eta.$$
 Indeed, it suffices to estimate one of the above two integrals due to symmetry of $\xi$ and $\eta$ in the definition of $K$. Moreover, without loss of generality, we can assume that $\lambda>0$ in this paper.

\subsection{Estimate  for large $\partial_t\Phi$}\label{sec99} 
We first reduce the support of  (\ref{10029}) to the region:
  $$ \left| |\xi+x| -|\eta+x|\right| \le \frac{\lambda^{\epsilon}}{\lambda}   \ \text{and}\    |\xi-\eta|\ge \frac{\lambda^{2\epsilon}}{\lambda} . $$
 
    \begin{proposition}\label{prop700}
Suppose that  $A\in M_{d\times d}(\mathbb{R})$ is invertible. Recall that 
\begin{align}\label{pc98}
K_{\rm{main}}(\xi,\eta)&:= \lambda^{d}\int_{\mathbb{R}^{d+1} } e^{2\pi i \lambda
\Phi(x,t,\xi,\eta)}
 \Psi(x,t,\xi,\eta)\nonumber\\
&
\times\psi\left(\frac{ |\xi+x| -|\eta+x|
  }{\frac{\lambda^{\epsilon}}{\lambda} }\right)  dxdt \left(1-\psi\left( \frac{|\xi-\eta|}{\frac{\lambda^{2\epsilon}}{\lambda}} \right)\right).   
 \end{align}
 Given the  kernel $K$ in (\ref{10029}),    it holds that for any small  
 $\epsilon>0$,
\begin{align}\label{kc201}
\sup_{\eta}\int_{\mathbb{R}^{d}} \left|K (\xi,\eta) \right|d\xi     \lesssim  \lambda^{2d\epsilon  }+\sup_{\eta}\int_{\mathbb{R}^{d}} \left|K _{\rm{main}}(\xi,\eta) \right|d\xi 
 \end{align}
  where  $\xi$ and $\eta$   can be switched.
  \end{proposition}

\begin{proof}[Proof of Proposition \ref{prop700}]
We first treat   a good portion of $K(\xi,\eta)$ in (\ref{10029}) defined by
\begin{align}
&K_{\rm{time osc}}(\xi,\eta)\nonumber\\
&\qquad:= \lambda^{d}\int_{\mathbb{R}^{d+1} } e^{2\pi i \lambda
\Phi(x,t,\xi,\eta)}
 \Psi(x,t,\xi,\eta)(1- \psi)\left(\frac{ |\xi+x| -|\eta+x|
  }{\frac{\lambda^{\epsilon}}{\lambda } }\right)  dxdt      \label{h03}
 \end{align}
which has a enough oscillation    as its time derivative $\partial_t\Phi$ has the lower bound $ \lambda^\epsilon/\lambda$.
The lower bound of $|\partial_t\Phi(x,t,\xi,\eta)|= \big| |\xi+x| -|\eta+x|\big|\gtrsim \lambda^{\epsilon}/\lambda  $ leads  the    estimates:
 \begin{align}
  \sup_{\eta}\int_{\mathbb{R}^{d}} \left|K_{\rm{timeosc}}(\xi,\eta) \right|d\xi     &\lesssim  \lambda^{-N} \ \text{ where we can switch $\xi$ and $\eta$.}
  \label{5kg}
  \end{align}
  \begin{proof}[Proof of (\ref{5kg})]
 For this case,  
 \begin{itemize}
 \item derivatives   $|\partial_t\Phi|\gtrsim \frac{\lambda^{\epsilon}}{\lambda} $ in (\ref{10029}) and (\ref{h03})
 \item derivative $\partial_t$ of cutoff function $=O(1)$ in (\ref{10029}) and  (\ref{h03})
 \item   the measures   $dx=O(1)$  and   $d\xi =O(1)$  in (\ref{10029}) . 
  \end{itemize}
So, we   apply integration by parts $N>\frac{M}{\epsilon} $ times with respect to $dt$ 
in   (\ref{h03}),
\begin{align*}
\sup_{\eta}\int |K_{\rm{timeosc}}(\xi,\eta)|d\xi&\lesssim  \lambda^{d}\int_{\mathbb{R}^d} \int_{\mathbb{R}^{d}\times\mathbb{R}^{1}} \lambda^{-\epsilon N} 
 \Psi(x,t,\xi,\eta)  dxdt d\xi\lesssim \lambda^{-(M-d)}\end{align*}
where $M> d$. 
\end{proof}
From (\ref{5kg}), it suffices to estimate $\int |K(\xi,\eta)-K_{\rm{timeosc}}(\xi,\eta)|d\xi$. It holds that 
\begin{align}
\sup_{\eta}\int \left|(K(\xi,\eta)-K_{\rm{timeosc}}(\xi,\eta))\psi\left( \frac{|\xi-\eta|}{\frac{\lambda^{2\epsilon}}{\lambda}} \right)\right|d\xi&\lesssim  \lambda^{2d\epsilon}.\label{43k7}
\end{align}
because the measure   $d\xi  $  is $  = O\left(\left| (\lambda^{2\epsilon}/\lambda)\right|^d\right)$  with $dx=O(1)$.
Note from (\ref{pc98}) and (\ref{h03}),  $$(K(\xi,\eta)-K_{\rm{timeosc}}(\xi,\eta))\left(1- \psi\left( \frac{|\xi-\eta|}{\frac{\lambda^{2\epsilon}}{\lambda}} \right)\right)=K_{\rm{main}}(\xi,\eta) .$$ This together with (\ref{5kg}) and (\ref{43k7}) yields (\ref{kc201}). 
We finish the proposition \ref{prop700}.  
\end{proof}
\subsection{Derivative of Phase Function}
We express the $x$-derivative of the phase function $\Phi$  of the integral kernel  of the operator of (\ref{pc98}) as the following simple form.
 \begin{proposition}\label{prop701}
For $(\xi,\eta)\in \text{supp}(K_{\rm{main}})$ in (\ref{pc98}), the phase   $\Phi $   satisfies that 
   \begin{align} \label{409}  
 \nabla_x \Phi(x,t,\xi,\eta)& =  A^T(\xi-\eta) + t \left( \frac{\xi-\eta}{|\xi+x|}+(\eta+x)  \left(  \frac{1}{|\xi+x|} -   \frac{1}{|\eta+x|}\right)   \right)\nonumber\\
 &= \left(A^T+    \frac{ t }{|\xi+x|} I\right)(\xi-\eta)   +O\left(\frac{\lambda^{\epsilon}}{\lambda}\right)\ \text{with}\ |\xi-\eta|\ge \lambda^{2\epsilon}/\lambda.
   \end{align} 
  \end{proposition}

\begin{proof}[Proof of (\ref{409})]
From the support condition of  $K_{\rm{main}}$   in   (\ref{pc98}),
\begin{align*}
|x|\lesssim 1,\ |\xi+x|\approx|\eta+x|\approx  t\approx 1\ \text{and}\ 
\bigg| |\xi+x|- |\eta+x|\bigg|\le\frac{\lambda^{\epsilon }}{\lambda}. 
\end{align*}
This condition implies that
\begin{align*}  
&\qquad   \left|(\eta+x)  \left( \frac{1}{|\xi+x|} - \frac{1}{|\eta+x|}\right)\right| \approx   \bigg| |\xi+x|- |\eta+x|\bigg|\le  \frac{\lambda^{ \epsilon}}{\lambda}.  
 \end{align*}  
This is the error term of  $  \nabla_x \Phi(x,t,\xi,\eta) $ in the second line below    
  \begin{align*}
 \nabla_x \Phi(x,t,\xi,\eta)&=   A^T(\xi-\eta) + t \left( \frac{\xi+x}{|\xi+x|}-  \frac{\eta+x}{|\eta+x|} \right)\\
 &= A^T(\xi-\eta) +t\left( \frac{\xi-\eta}{|\xi+x|}  +(\eta+x)  \left( \frac{1}{|\xi+x|} - \frac{1}{|\eta+x|} \right) \right)    \\
&= \left(A^T+  
  \frac{t}{|\xi+x|}\right) \left(\xi-\eta\right)+O\left(\frac{\lambda^{\epsilon}}{\lambda}\right).  
 \end{align*}
  Here   $|\xi-\eta|\ge \lambda^{2 \epsilon}/\lambda$ from the support condition of (\ref{pc98}) .     \end{proof}

  \subsection{Estimate of $\int |K_{\rm{main}}(\xi,\eta)|d\xi$ for Proof of Theorem \ref{th44} }
To prove Theorem 
\ref{th44}, from Propositions  \ref{pr02} and \ref{prop700}, it suffices to   claim that 
\begin{align}
\sup_{\eta}\int \left| K_{\rm{main}}(\xi,\eta) \right|d\xi&\lesssim  1.\label{5kgkg}
\end{align}
   \begin{proof}[Proof of (\ref{5kgkg})]
Since $A$ has no real eigenvalue and $\frac{t}{|\xi+x|}\approx 1$ in (\ref{pc98}),   the matrix $A^T+\frac{t}{|\xi+x|}I$   in   (\ref{409})    is     invertible.  Thus in (\ref{409}),
  $$ \left| \left(A^T+ 
  \frac{ t}{|\xi+x|}I\right)\left(\xi-\eta\right)\right|\approx|\xi-\eta|\ge \lambda^{2 \epsilon}/\lambda.$$
On the other hand, we can decompose the support $\big\{(x,t): \big|\frac{ t}{|\xi+x|}\big|\approx 1 \big\}$ into   finer but  finitely many  separate pieces so that $ t/|\xi+x|$ is close to one fixed number $\approx 1$ on each slice.  Thus, on each piece, we can find a unit vector $v$, almost parallel to $\left(A^T+
  \frac{ t}{|\xi+x|}I\right)\left(\xi-\eta\right) $, not  depending on $x$ satisfying that  
   \begin{align} 
 |\langle v,\nabla_x \rangle \Phi(x,t,\xi,\eta)|& \approx \left| \left(A^T+
  \frac{ t}{|\xi+x|}I\right)\left(\xi-\eta\right)\right|\approx|\xi-\eta| \gtrsim \lambda^{2\epsilon }/\lambda.  \label{pc202} \end{align}  
 The higher derivatives of  $\langle v,\nabla_x \rangle \Phi(x,t,\xi,\eta)$ along any unit vector $v$ is majorized by
  $C|\xi-\eta|$.  As in (\ref{409}), it holds that  
\begin{align}\label{6613}
\nabla_x(|\xi+x|-|\eta+x|)=\frac{\xi+x}{|\xi+x|}-\frac{\eta+x}{|\eta+x|} =t\frac{\xi-\eta}{|\xi+x|}+O(\lambda^{\epsilon-1}).
 \end{align}
  So, we can compute that
  \begin{align*} 
 \left| \langle v,\nabla_x \rangle \psi\left(\frac{ |\xi+x| -|\eta+x|
  }{\frac{\lambda^{\epsilon}}{\lambda}}\right)\right|&=\left|  \psi'\left(\cdot  \right) \left\langle v,\,  t\frac{\xi-\eta}{|\xi+x|}+O(\lambda^{\epsilon-1}) \right\rangle\right|\nonumber\\
  &\lesssim \left|\frac{\lambda}{\lambda^{\epsilon}} \left(  |\xi-\eta| +  \frac{\lambda^{\epsilon}}{\lambda}\right) \right|
   \end{align*}  
leading its higher derivatives along any unit vector $v$
   \begin{align} \label{613}
\left|\langle v, \nabla_x \rangle^m \left[\psi\left(\frac{ |\xi+x| -|\eta+x|
  }{\frac{\lambda^{\epsilon}}{\lambda}}\right) \right]\right|=O\left(  \left[\frac{\lambda}{\lambda^{\epsilon}}  | \xi-\eta |+1\right]^m \right).
   \end{align} 
Thus, for $\Psi= \text{amplitidue of (\ref{pc98})} $, we have an upper bound   as
\begin{align}\label{61kk}
  | \langle v, \nabla_x \rangle^m \Psi(x,t,\xi,\eta)|\lesssim  \left[\frac{\lambda}{\lambda^{\epsilon}}  | \xi-\eta |+1\right]^m.
  \end{align}
So, we  utilize  the product rule  with (\ref{61kk}) and (\ref{pc202}) to compute  possible upper bounds to obtain   that
  $$ \left|\left[\langle v,\nabla_x \rangle \frac{1}{\lambda \langle v,\nabla_x \rangle \Phi(x,t,\xi,\eta)}\right]^m\Psi(x,t,\xi,\eta)\right|\le      
\left| \frac{\lambda^{1-\epsilon}   |\xi-\eta|}{\lambda|\xi-\eta|} \right|^{m} +   
\left| \frac{ 1}{\lambda|\xi-\eta| +1} \right|^{m}.$$
This enables  us to apply the integration by parts $N $  times with $ N\epsilon\ge 2d+1$ to obtain that
\begin{align*}
  \left|K_{\rm{main}}(\xi,\eta)\right| &\lesssim  \lambda^{d}   
\left| \frac{\lambda^{1-\epsilon}  }{\lambda } \right|^{N} +\lambda^{d}   
\left| \frac{ 1}{\lambda|\xi-\eta| +1} \right|^{N}. 
  \end{align*}
This with   $|\xi|,|\eta|\le 10$ in (\ref{pc98}) gives the desired bound $\int   \left|K_{\rm{main}}(\xi,\eta)\right|  d\xi= O(1)$ in (\ref{5kgkg}).
\end{proof} 
Hence,  we finish the proof of  Theorem \ref{th44} from  (\ref{5kgkg}) and (\ref{kc201}).

 \section{Proof  of Theorem \ref{th4}  }\label{sec6}
\subsection{Proof  of Theorem \ref{th4} for $\text{rank}(JA+(JA)^T)=2,0$.} Remind that in Theorem \ref{th4}, we assumed  $\text{rank}(A)=2$. 
From  
 Proposition  \ref{prop700}  for $d=2$, it suffices to show that,
\begin{align}\label{cb11}
\sup_{\eta}\int \left|K_{\rm{main}}(\xi,\eta)  \right|d\xi&\lesssim_{\epsilon} \begin{cases}  \lambda^{0}\ \text{if $\text{rank}\left(JA+(JA)^T\right)=2$},
\\
  \lambda^{1} \ \text{if $\text{rank}\left(JA+(JA)^T\right)=0$.}
\end{cases}
\end{align}
  Recall that $K_{\rm{main}}(\xi,\eta)$   in Propositions \ref{prop700} and \ref{prop701} for $d=2$ is
 \begin{align}\label{pc981}
K_{\rm{main}}(\xi,\eta)&:= \lambda^{2}\int_{\mathbb{R}^{2+1} } e^{2\pi i \lambda
\Phi(x,t,\xi,\eta)}
 \psi(x)\chi(t)\chi\left( 
t |\xi+ x| \right)\chi\left(  t |\eta+
x| \right)\nonumber\\
&
\times\psi\left(\frac{ |\xi+x| -|\eta+x|
  }{\frac{\lambda^{\epsilon}}{\lambda} }\right)  dxdt \left(1-\psi\left( \frac{|\xi-\eta|}{\frac{\lambda^{2\epsilon}}{\lambda}} \right)\right),\\   
\Phi(x,t,\xi,\eta)&=\langle A^T(\xi-\eta), x\rangle +t\left(|\xi+x|-|\eta+x|\right),\nonumber\\
 \nabla_x \Phi(x,t,\xi,\eta)& = \left(A^T+    \frac{ t }{|\xi+x|} I\right)(\xi-\eta)   +O\left(\frac{\lambda^{\epsilon}}{\lambda}\right)\ \text{with}\ |\xi-\eta|\ge \lambda^{2\epsilon}/\lambda.\nonumber
   \end{align} 
This combined with  the observation
$$  \left\langle v,  \frac{J(\xi-\eta)}{|\xi-\eta|}\right\rangle=0\  \text{for}\ v=\frac{J(\xi-\eta)}{|\xi-\eta|}$$
implies that   
\begin{align}
\langle v,\nabla_x\rangle \Phi&=  \left\langle A^T\left(\xi-\eta\right),   \frac{J(\xi-\eta)}{|\xi-\eta|} \right\rangle +O(\lambda^{\epsilon}/\lambda).  \label{att} 
\end{align}
 According to the size  $ \left\langle A^T\left(\xi-\eta\right),   \frac{J(\xi-\eta)}{|\xi-\eta|} \right\rangle$ in (\ref{att}), 
we split     (\ref{pc981}): $$K_{\rm{main}}(\xi,\eta) =K_{\rm{main}}^{\rm{good}}(\xi,\eta)+K_{\rm{main}}^{\rm{bad}}(\xi,\eta)$$ where
\begin{align*}
K_{\rm{main}}^{\rm{good}}(\xi,\eta)&=K_{\rm{main}}(\xi,\eta)(1-\psi)\left( \frac{ \left\langle A^T\left(\xi-\eta\right),   \frac{J(\xi-\eta)}{|\xi-\eta|} \right\rangle }{2^{10}\lambda^{2\epsilon  }\lambda^{-1} }\right),  \\
  K_{\rm{main}}^{\rm{bad}}(\xi,\eta)&=K_{\rm{main}}(\xi,\eta) \psi\left( \frac{ \left\langle A^T\left(\xi-\eta\right),   \frac{J(\xi-\eta)}{|\xi-\eta|} \right\rangle }{ 2^{10} \lambda^{2\epsilon } \lambda^{-1}   }\right).   \end{align*}
 {\bf Case  $K^{\rm{good}}_{\rm{main}} $}. 
The support   of $K^{\rm{good}}_{\rm{main}}(\xi,\eta)$ in (\ref{att}) gives the lower bound 
\begin{align*} 
|\langle v,\nabla_x\rangle \Phi|\gtrsim \lambda^{2\epsilon } \lambda^{-1}.
\end{align*}
 On the other hand,  $v\perp (\xi-\eta)$ also implies $[\langle v,\nabla_x \rangle]^m\Psi=O(1)$, which follows from 
\begin{align*} 
  \langle v,\nabla_x \rangle \psi\left(\frac{ |\xi+x| -|\eta+x|
  }{\frac{\lambda^{\epsilon}}{\lambda}}\right) &=  \psi'\left(\cdot  \right)\frac{\lambda}{\lambda^{\epsilon}} \left\langle v,\,  \frac{\xi-\eta}{|\xi+x|}+(\eta+x)  \left( \frac{1}{|\xi+x|} - \frac{1}{|\eta+x|} \right)  \right\rangle \nonumber\\
  &=  \psi'\left(\cdot  \right) \frac{\lambda}{\lambda^{\epsilon}}\left\langle v, \eta+x\right\rangle  \left( \frac{1}{|\xi+x|} - \frac{1}{|\eta+x|} \right) =O(1)
   \end{align*}  
Thus,  by applying  integration  by parts $N\gg 1$  times with respect to $dx$ for  (\ref{pc981}),
\begin{align*}
\sup_{\eta}\int |K_{\rm{main}}^{\rm{good}}(\xi,\eta)|d\xi \lesssim \frac{\lambda^{2}}{|\lambda \left( \lambda^{2\epsilon }\lambda^{-1}\right)|^{N} }\lesssim 1\ \text{for $2\epsilon N\ge 2$}. 
\end{align*}
 {\bf Case} $K_{\rm{main}}^{\rm{bad}} $.  
We shall claim     that  
\begin{align}\label{fp1}
\sup_{\eta}\int |K^{\rm{bad}}_{\rm{main}}(\xi,\eta)|d\xi \lesssim_{\epsilon} \begin{cases}  \lambda^{0}  \  \text{if 
$\text{rank}\left(JA+(JA)^T\right)=2$,} \\
 \lambda^{1} \ \text{ if  $\text{rank}\left(JA+(JA)^T\right)=0$.}
  \end{cases}
\end{align}
\begin{proof}[Proof of (\ref{fp1})]  
  It suffices to deal with $K^{\rm{bad}}_{\rm{main}}$ supported on the set
\begin{align*} \Gamma_{\rm{main}}^{\rm{bad}}=\left\{(\xi,\eta)\in\mathbb{R}^2\times\mathbb{R}^2:  \left\langle A^T\left(\xi-\eta\right),   J(\xi-\eta)  \right\rangle  \le |\xi-\eta|  2^{10}\lambda^{2\epsilon}\lambda^{-1} \right\}. 
\end{align*}
By $\left\langle  \left(\xi-\eta\right),   AJ(\xi-\eta)\right\rangle = \left\langle  AJ\left(\xi-\eta\right),   (\xi-\eta)\right\rangle $, 
\begin{align*}
  \left\langle A^T\left(\xi-\eta\right),    J(\xi-\eta)  \right\rangle& =  \left\langle  \frac{ JA+(JA)^T}{2} \left(\xi-\eta\right),    (\xi-\eta)\right\rangle. 
  \end{align*}
Let $\Gamma_m$ be the intersection  $\Gamma_{\rm{main}}^{\rm{bad}}\cap\{(\xi,\eta):|\xi-\eta|\approx 2^{-m}\le 1\}$ defined  by
$$
 \Gamma_m:= \left\{(\xi,\eta):  \left\langle  \frac{ JA+(JA)^T}{2} \left(\xi-\eta\right),    (\xi-\eta)\right\rangle   \le 2^{-m} \lambda^{2\epsilon}\lambda^{-1}\ \text{and}\ |\xi-\eta|\approx 2^{-m}\right\}. $$
  Then we can observe that $\Gamma^{\rm{bad}}_{\rm{main}} \subset\bigcup_{m=0}^\infty \Gamma_m$. For $\text{rank}(A)=2$,  we claim that
   \begin{align}\label{pc991}
 \left|\int\chi_{\Gamma_m}(\xi,\eta)d\xi\right|\le \begin{cases} C2^{-m}\lambda^{2\epsilon}\lambda^{-1}\ \text{if $\text{rank}(JA+(JA)^T)=2$},\\
C2^{-2m} \ \text{if $\text{rank}(JA+(JA)^T)=0$.}\end{cases}
\end{align}
\begin{proof}[Proof of (\ref{pc991})]
Since   a symmetric matrix is orthogonally diagonalizable, it holds that   $  (JA+(JA)^T)/2 = QDQ^T$ with an orthogonal matrix $Q$ and   two eigenvalues $\mu_1,\mu_2$ in the diagonal   of $D$. We use a change of variable $\xi-\eta\rightarrow \xi$ and next a rotation via $Q$ to obtain 
\begin{align*}
  \left|\int\chi_{\Gamma_m}(\xi,\eta)d\xi\right|&\le|\{\xi=(\xi_1,\xi_2):   | \mu_1\xi_1^2+\mu_2\xi_2^2|   \le 2^{-m} \lambda^{2\epsilon}\lambda^{-1}\ \text{and}\ |\xi|\approx 2^{-m}\} | 
  \\
  &=\begin{cases}O(2^{-m}\lambda^{2\epsilon}\lambda^{-1}) \ \text{ if  $\mu_1,\mu_2\ne 0$ namely $\text{rank}(JA+(JA)^T)=2$}\\
 O(2^{-2m})\ \text{ if   $\mu_1=\mu_2=0$ namely  $\text{rank}(JA+(JA)^T)=0$.}
 \end{cases}
  \end{align*}  
This follows from the elementary sub-level set estimates. So we proved (\ref{pc991}).
\end{proof}
In (\ref{pc981}), by another  sub-level set estimate with respect to $dx$ again,  we have 
\begin{align*} 
|K^{\rm{bad}}_{\rm{main}}(\xi,\eta) |\lesssim \lambda^{2}\int_{\mathbb{R}^2}  \psi\left(\frac{ |\xi+x| -|\eta+x|
  }{\lambda^{\epsilon}\lambda^{-1} }\right)dx\lesssim \lambda^{2}\frac{\lambda^{\epsilon}\lambda^{-1} }{|\xi-\eta| } 
  \end{align*} 
because
   $\left|\nabla_{x}\left( |\xi+x| -|\eta+x|\right)\right|  \approx |\xi-\eta|$ due to (\ref{6613}).    Thus, we obtain that
   \begin{align}\label{fp4}
\int_{\mathbb{R}^2} |K_{\rm{main}}^{\rm{bad}}(\xi,\eta)|d\xi&\lesssim\int_{\mathbb{R}^2} \lambda^{2}\frac{\lambda^{\epsilon}\lambda^{-1} }{|\xi-\eta| }  \sum_{\lambda^{-1j}\le 2^{-m}\le 1}\chi\left(\frac{|\xi-\eta|}{2^{-m}}\right)\chi_{\Gamma_m} (\xi,\eta)d\xi.   
 \end{align}
 If  $\text{rank}(JA+(JA)^T)=2$, then (\ref{pc991}) implies that $RHS$ of (\ref{fp4}) is bounded by
 $$ \sum_{\lambda^{-1}\le 2^{-m}\le 1}\lambda^{2}   \left(\frac{\lambda^{\epsilon}\lambda^{-1}}{2^{-m} } \right) 2^{-m}\lambda^{2\epsilon}\lambda^{-1}\le C\lambda^{4\epsilon}. $$
 If  $\text{rank}(JA+(JA)^T)=0$, then (\ref{pc991}) implies that $RHS$ of (\ref{fp4}) is bounded by
 $$ \sum_{\lambda^{-1}\le 2^{-m}\le 1}\lambda^{2}   \left(\frac{\lambda^{\epsilon}\lambda^{-1}}{2^{-m} } \right)   2^{-2m} \le C\lambda\lambda^{2\epsilon}. $$
 Hence we proved (\ref{fp1}).  \end{proof}
Therefore we finish the proof of (\ref{cb11}).

\subsection{Proof  of Theorem \ref{th4} for $\text{rank}\left(JA+(JA)^T\right)=1$ }\label{sec8.2}
Let $\phi(x,t,\xi,\eta)=\langle A(x), \xi \rangle+t| \xi+x|$ for $ A=\left(\begin{matrix}1&c\\ 0&1\end{matrix}  \right)$ as  (1-2) of Proposition \ref{lem10044}  and Lemma \ref{lem2.2} and recall that
\begin{align*} 
\mathcal{T}^{\lambda}_{\rm{annulus}}g(x,t )&=\lambda  \chi\left(t\right)\psi\left(x\right) \int_{\mathbb{R}^2}
e^{2\pi i\lambda\phi(x,t,\xi,\eta)}
\chi\left(  t|\xi+x| \right) \widehat{g
 }(  \xi)d\xi\ \text{}.\end{align*} 
We claim that under the conditions  $\text{rank}\left(JA+(JA)^T\right)=1$ and $\text{rank}(A)=2$,
\begin{align}
\|\mathcal{T}^{\lambda}_{\rm{annulus}}\|_{L^2(\mathbb{R}^2)\rightarrow L^2(\mathbb{R}^{2+1})}\lesssim \lambda^{1/6+\epsilon }.\label{pp89}
\end{align} 
For this purpose,  we shall apply the following well-known  results of the  two-sided fold singularities.
\begin{lemma}[\cite{GS2} Two--Sided Fold Singularity]\label{pr9}
Let  $\psi\in C_c(\mathbb{R}^d\times\mathbb{R}^d)$ and $\phi$ be a smooth real valued phase function defined in $\mathbb{R}^d\times\mathbb{R}^d$.  We consider a family of operators $\{\mathcal{S}^{\lambda}\}$ with $\lambda>0$:
$$\mathcal{S}^{\lambda}f(x)=\int e^{i\lambda \phi(x,y)} \psi(x,y)f(y)dy\ \text{for $f\in L^2(\mathbb{R}^d)$}.$$By $[\phi_{xy}''(x_0,y_0)]\in M_{d\times d}(\mathbb{R})$, we denote   the  mixed hessian matrix of $\phi$ at $(x_0,y_0)$. 
Assume that $C_{\phi}=\left\{(x,\phi'_x,y,-\phi'_y)\right\}$ is a two-sided folding canonical relation, that is, 
\begin{itemize}
\item for each point $(x_0,y_0)\in \text{supp}(\psi)$ satisfying $\det[\phi_{xy}''(x_0,y_0)]=0$, it holds that
\begin{align}\label{0643}
\text{rank}([\phi_{xy}''(x_0,y_0)])= d-1, 
\end{align}
\item for unit vectors  $U$ and $V$ in $\mathbb{R}^d$, it holds that
\begin{align}\label{pp88}
\begin{cases}
$$\text{If $[\phi_{xy}''(x_0,y_0)]V=0$,}\ \text{then}\ \left| \langle V,\nabla_y\rangle  \det([\phi_{xy}''(x,y)])\big|_{(x_0,y_0)}\right| >0,$$\\
 $$\text{If $U^T[\phi_{xy}''(x_0,y_0)]=0$,}\ \text{then}\  \left| \langle U,\nabla_x\rangle \det( [\phi_{xy}''(x,y)])\big|_{(x_0,y_0)}\right| >0.$$
 \end{cases}
 \end{align}
\end{itemize}
 Then there is an additional decay $\lambda^{-1/3}$ from the third derivative of the phase as
$$\|\mathcal{S}^{\lambda}\|_{L^2(\mathbb{R}^d)\rightarrow L^2(\mathbb{R}^d)}\lesssim \lambda^{-(d-1)/2}\lambda^{-1/3}.$$
\end{lemma}
We  utilize  Lemma \ref{pr9} to prove   Proposition \ref{lem81} below.
\begin{proposition} \label{lem81}
Let $\phi(x,t,y)=\langle A(x), y\rangle+t|x+y|$ where $A= \left(\begin{matrix}1&c\\ 0&1\end{matrix}  \right)$. Set
\begin{align*} 
\mathcal{S}^{\lambda}f(x,t)= \chi(t)\psi(x)\int e^{i\lambda  \phi(x,t,y) } \chi(|x+y|)f(y)dy.
\end{align*}
  Then  it holds that
\begin{align}
\|\mathcal{S}^{\lambda}\|_{L^2(\mathbb{R}^2)\rightarrow L^2(\mathbb{R}^2\times\mathbb{R})}\lesssim \lambda^{-1/2}\lambda^{-1/3}.\label{807}\end{align}
\end{proposition}
\begin{proof}[Proof of (\ref{807})]
  Split $\mathcal{S}^\lambda=\mathcal{S}^\lambda_1+\mathcal{S}^{\lambda}_2$ where
  \begin{align*}
  \mathcal{S}^{\lambda}_1f(x,t)&=\chi(t)\psi(x)\int e^{i\lambda\phi(x,t,y)} \psi\left(\frac{|x_1+y_1|}{\epsilon}\right)\psi\left(\frac{|x+y|+t}{\epsilon}\right)\chi(|x+y|)f(y)dy, \\
   \mathcal{S}^{\lambda}_2f(x,t)&=\chi(t)\psi(x)\int e^{i\lambda\phi(x,t,y)} \left(1-\psi\left(\frac{|x_1+y_1|}{\epsilon}\right)\psi\left(\frac{|x+y|+t}{\epsilon}\right)\right)\chi(|x+y|)f(y)dy.   
  \end{align*}
  In the support of the integral above, we use the letters  $u=(u_1,u_2)$
\begin{align*}
\text{$(u_1,u_2)=(x_1+y_1,x_2+y_2)$ where
  $|u|=|x+y|\approx 1$.}
  \end{align*} 
We first claim that
\begin{align}\label{100232}
\|\mathcal{S}^{\lambda}_2\|_{op}\lesssim \lambda^{-1}.
\end{align}
    \begin{proof}[Proof of (\ref{100232})]
We  apply H$\ddot{o}$rmander's theorem according to  the non-vanising mixed hessians of  $2\times 2$ submatrices of the $2\times 3$ full mixed derivative matrix of $ \phi(x,t,y)=\langle A(x), y\rangle+t|x+y|$ in (\ref{a72}).  Recall that for $A=\left(   \begin{matrix}a_{11} &a_{12} \\ a_{21}&a_{22}\end{matrix}  \right) =\left(\begin{matrix}1&c\\ 0&1\end{matrix}  \right)$ and $u=x+y$,
\begin{align*}
\det\left([\phi''_{(x_1t)(y_1y_2)}]\right)&=\frac{1}{|u|} \left(\det\left(   \begin{matrix}a_{11} &a_{12} \\ u_1&u_2\end{matrix}  \right) +\frac{tu_2}{|u|} \right)= \frac{1}{|u|}\left(u_2\left(1+ \frac{t}{|u|}\right)-cu_1 \right) \\
\det\left([\phi''_{(x_2t)(y_1y_2)}]\right)&=\frac{1}{|u|} \left(\det\left(   \begin{matrix}a_{21} &a_{22} \\ u_1&u_2\end{matrix}  \right) -\frac{tu_1}{|u|} \right) =\frac{-u_1}{|u|}\left(1+ \frac{t}{|u|}\right). 
\end{align*}
Then in the support of
$$ \left(1-\psi\left(\frac{|x_1+y_1|}{\epsilon}\right)\psi\left(\frac{|x+y|+t}{\epsilon}\right)\right)\chi(|x+y|)= \left(1-\psi\left(\frac{|u_1|}{\epsilon}\right)\psi\left(\frac{|u|+t}{\epsilon}\right)\right)\chi(|u|) ,$$ we choose the non-vanishing   mixed hessian according to the following three cases:
 \begin{itemize}
 \item If $|u_1|\le \frac{\epsilon}{10(1+|c|) }$, then $||u|+t|\ge \epsilon/2$ with $1/2\le |u|\le 2$ and $|u_2|\approx |u|$. This implies $$\left|\det\left([\phi''_{(x_1t)(y_1y_2)}]\right)  \right|\approx\left| \left(1+ \frac{t}{|u|}\right) \right|\gtrsim \epsilon. $$
  \item Let $|u_1|\ge \frac{\epsilon}{10(1+|c|) }$ and $||u|+t|\le  \frac{|c|\epsilon}{100(1+|c|) }  $.   This with  $|u|\approx 1$ implies
$$\left|\det\left([\phi''_{(x_1t)(y_1y_2)}]\right)  \right|\approx \frac{c|u_1|}{|u|}\gtrsim \epsilon.$$
\item Let  $|u_1|\ge \frac{\epsilon}{10(1+|c|) }$ and $||u|+t|\ge  \frac{|c|\epsilon}{100(1+|c|) }  $. Then $$\left|\det\left([\phi''_{(x_2t)(y_1y_2)}]\right)  \right|\gtrsim \epsilon^2.
$$
\end{itemize} 
Notice that the derivatives of cutoff functions separating the above three regions do not blow up.  We regard  $\mathcal{S}^{\lambda}f $ as a function of $(x_1,t)$ or $(x_2,t)$ respectively, whose phase function satisfies the non-degeneracy assumption of the  $2\times 2$ mixed hessian matrix. Then the desired estimate (\ref{100232}) follows from  the H$\ddot{o}$rmander  theorem.  
\end{proof}
We next   claim that
\begin{align}\label{100233}
\|\mathcal{S}^{\lambda}_1\|_{op}\lesssim \lambda^{-1/2}\lambda^{-1/3}.
\end{align}
\begin{proof}[Proof of (\ref{100233})]
Here, we freeze $t$ with $|t|\approx 1$ to treat $\mathcal{S}^\lambda_1f(x_1,x_2,t)$ and apply the two-sided fold singularity in Proposition \ref{pr9}.  From the support condition 
\begin{align}\label{100235}
  \psi\left(\frac{|x_1+y_1|}{\epsilon}\right)\psi\left(\frac{|x+y|+t}{\epsilon}\right),
  \end{align}
it holds that for    $u=(u_1,u_2)=(x_1+y_1,x_2+y_2)$,
 \begin{align}\label{p44}
\text{$ |u_1| =O(\epsilon)$,\ \text{that is} \  $  |u_2|^2=|u|^2+O(\epsilon^2)$,
and $1\approx |u|= -t+O(\epsilon)$}.
\end{align} 
At the points $(x_1,x_2),(y_1,y_2)$ satisfying (\ref{p44}) and $\det\left(\phi_{(x_1x_2)(y_1y_2)}''\right)=0$, we shall check the fold singularity of  $(x,\phi_x, y,-\phi_y)$.  
From $A= \left(\begin{matrix}a_{11}&a_{12}\\
a_{21}&a_{22}
\end{matrix}\right)= \left(\begin{matrix}1&c\\ 0&1\end{matrix}  \right)$ and  (\ref{p44}), we compute the mixed hessian matrix of  $ \phi(x,t,y)=\langle A(x), y\rangle+t|x+y|$ in  (\ref{a72}) at $(x,y)$ for a fixed $t$ as
\begin{align}\label{08b}
[\phi_{(x_1x_2)(y_1y_2)}''](x,y)= \left(\begin{matrix}1+t\frac{u_2^2}{|u|^3}&c- t\frac{u_1u_2}{|u|^3} \\ 0-
t\frac{u_1u_2}{|u|^3} &1+t\frac{u_1^2}{|u|^3} \end{matrix}  \right)= \left(\begin{matrix}0&c 
\\ 0 &1 \end{matrix}  \right) +O(\epsilon).
\end{align}
Here $O(\epsilon)$ means that every entry of the $2\times 2$ matrix is $O(\epsilon)  $ and
$
\det\left([\phi_{(x_1x_2)(y_1y_2)}''](x,y)\right)  =O(\epsilon).
$
Let $(x_0,y_0)$ be the point in the support of the kernel in  (\ref{100235}) for (\ref{08b}) to satisfy $$\det\left([\phi_{(x_1x_2)(y_1y_2)}''](x_0,y_0)\right)=0.$$ Then  from (\ref{08b}) it follows that
$\text{rank}([\phi_{xy}''(x_0,y_0)])=1$ satisfying (\ref{0643}).
The  kernel vector fields $V$ and $U$ in (\ref{pp88})  for the matrix $[\phi_{(x_1x_2)(y_1y_2)}''(x_0,y_0)]$ of (\ref{08b}) satisfying
$$  [\phi_{(x_1x_2)(y_1y_2)}'' (x_0,y_0)]V={\bf 0}\ \text{and}\ U^T [\phi''_{(x_1x_2)(y_1y_2)} (x_0,y_0)]={\bf 0}\ \text{with  $|U|=|V|=1$}$$
are parallel to $ (1,0)+O(\epsilon)$ and $ (-1,c)+O(\epsilon)$ respectively.   
On the other hand, we can evaluate the determinant of the hessian matrix  (\ref{08b})  at $(x,y)$ for a fixed $t$  as
\begin{align*} 
\det \left([\phi_{(x_1x_2)(y_1y_2)}''(x,y)]\right)&= 1+t\left(\frac{ u_1^2+ u_2^2+cu_1u_2}{|u|^3}\right)  
\end{align*}  
 for $u=x+y$ as in  (\ref{a72}).
Then we can see that
 $$\nabla_{x}\det\left([\phi_{(x_1x_2)(y_1y_2)}''(x,y)]\right)=\nabla_y\det\left([\phi_{(x_1x_2)(y_1y_2)}''(x,y)]\right)$$ whose components are  
\begin{align*}
\frac{\partial}{\partial y_1}&\det\left([\phi_{(x_1x_2)(y_1y_2)}''(x,y)]\right) =\frac{\partial}{\partial x_1}\det\left([\phi_{(x_1x_2)(y_1y_2)}''(x,y)]\right)\\
&\qquad\qquad= \frac{t}{|u|^5}\left( |u|^2( 2u_1+cu_2)-3 u_1( u_1^2+cu_1u_2+ u_2^2)\right)=ct  u_2/|u|^3+O(\epsilon)
\end{align*}
\begin{align*}
\frac{\partial}{\partial y_2}&\det\left([\phi_{(x_1x_2)(y_1y_2)}''(x,y)]\right) =\frac{\partial}{\partial x_2}\det\left([\phi_{(x_1x_2)(y_1y_2)}''(x,y)]\right)\\ 
&\qquad\qquad= \frac{t}{|u|^5} \left(|u|^2( 2u_2+cu_1)-3u_2( u_1^2+cu_1u_2+ u_2^2)\right) =-t u_2 /|u|^3+O(\epsilon).
\end{align*}
 Thus we utilize the two kernel vector field $V= (1,0)+O(\epsilon)$ and $U= \frac{(-1,c)}{\sqrt{1+c^2}}+O(\epsilon)  $ with the support condition (\ref{p44}) to compute the directional derivative in   (\ref{pp88}), 
\begin{align*}
 \langle V,\nabla_y\rangle  \det([\phi_{xy}''(x ,y )])\bigg|_{(x_0,y_0)}&=\frac{tu_2}{|u|^3}\left(  c,  -1\right)\cdot  (1,0)+O(\epsilon)=\frac{ ct u_2}{|u|^3}+O(\epsilon) \ne 0, \\
 \langle U,\nabla_x\rangle \det( \phi_{xy}''(x_0,y_0))\bigg|_{(x_0,y_0)}&=\frac{tu_2}{|u|^3}\left( c,  -1\right)\cdot  \frac{(-1,c)}{\sqrt{1+c^2}}+O(\epsilon)=\frac{-2ctu_2 }{\sqrt{1+c^2 }|u|^3}+O(\epsilon)  \ne 0.
\end{align*}
where $u=x_0+y_0$,
satisfying the condition (\ref{pp88}) for a fixed $t$  with $|t|\approx 1$.
Hence
$(x,\phi'_x, y,-\phi'_y)$ on the support of (\ref{100235}) has two sided fold singularity. Therefore, we freeze $t$ and apply Lemma \ref{pr9} for $d=2$ to obtain (\ref{100233}).
\end{proof}
We have proved    (\ref{100232}) and (\ref{100233}),  to  finish the proof of (\ref{807}).
\end{proof}
The proposition \ref{lem81}  implies that $\|\mathcal{T}^\lambda_{\rm{annulus}} \|_{op}=\|\lambda \mathcal{S}^{\lambda}\|_{op}\lesssim\lambda^{1/6}$   to   show (\ref{pp89}),  which is  the   case $\text{rank}(A)=2$ and $\text{rank}((JA)+(JA)^T)=1$.  Therefore we  finished the proof of Theorem \ref{th4}.

 \section{Proof of Theorem \ref{th441} ($\text{rank}(A)=1$)}\label{sec90}
 The goal of this section is to prove Theorem \ref{th441}.  
 Under the assumption of $\text{rank}(A)=1$,  we  shall  show
\begin{align}\label{nv2}
\| \mathcal{T}_{j}^{\lambda}\|_{L^2(\mathbb{R}^2)\rightarrow L^2(\mathbb{R}^{2+1})} \lesssim 2^{\epsilon j}\ \text{uniformly in $\lambda$}.
\end{align}
First of all, we show the following lemma.
\begin{lemma}\label{lem81} Suppose that 
  $\text{rank}(A)=1$.
 Then there exists an orthogonal matrix $Q$ such that
  $$Q^TAQ=   \left(\begin{matrix} \alpha &0\\  \beta&0\end{matrix}\right)\ \text{or}\  \left(\begin{matrix}0 &\alpha\\  0&\beta\end{matrix}\right).$$
  \end{lemma}
 \begin{proof}
First consider the case that $A$ is diagonalizable. Let  $Av_1=\lambda_1 v_1 $ and $Av_2=\lambda_2 v_2$ where $v_1$ and $v_2$ are two normalized eigenvectors. We may assume that $\lambda_1=0$ because $\det(A)=0$. Thus the other eigenvalue $\lambda_2$ is   a purely real number so that both $v_1,v_2\in \mathbb{R}^2$. Suppose that $v_1,v_2$ are linearly independent. Then by Gram-Schmidt process, there exists an orthonormal vectors $q_1=v_1$ and $q_2=\frac{v_2-\langle v_2,v_1\rangle v_1}{c}$  with $c=|v_2-\langle v_2,v_1\rangle v_1|$.  Put
$Q= \left(\begin{matrix} q_1|q_2\end{matrix}\right)$. Then $$AQ=\left(\begin{matrix}Aq_1|Aq_2\end{matrix}\right)=\left(\begin{matrix} {\bf 0}| \, \frac{1}{c} \lambda_2v_2\end{matrix}\right).$$
Thus
$$Q^TAQ=\left(\begin{matrix} q_1^T\\q_2^T\end{matrix}\right) \left(\begin{matrix} {\bf 0}| \, \frac{1}{c} \lambda_2v_2\end{matrix}\right)= \left(\begin{matrix}0&    \langle q_1,\frac{1}{c}\lambda_2v_2\rangle\\
0& \langle q_2^T, \frac{1}{c} \lambda_2v_2\rangle\end{matrix}\right)=   \left(\begin{matrix}0&   \frac{1}{c} \langle v_1,\lambda_2v_2\rangle\\
0& \frac{\lambda_2}{c}(|v_2|^2-|\langle v_2,v_1\rangle|^2)\end{matrix}\right)$$
Next consider the case  that $A$ is not diagonalizable. Then the only eigenvalue of $A$ is $\lambda=0$ with its eigenvectors $v_1,v_2$   dependent.  Then (1-2) of Proposition \ref{lem10044} for the case $\lambda=0$ implies that $ A=Q^T \left(\begin{matrix}0&c\\ 0&0\end{matrix}\right)    Q$. This completes the proof of Lemma \ref{lem81} 
 \end{proof}
\subsection{Basic Reduction and Decompositions}
In view of (ii) of Lemma \ref{lem2.2} and Lemma \ref{lem81}, we let $A=\left(\begin{matrix} \alpha &0\\  \beta&0\end{matrix}\right)$ and show  (\ref{nv2}). In view of Remark \ref{rek1} and Proposition \ref{prop111}, it suffices to treat the only case
\begin{align}\label{kyo}
\frac{2^j}{\lambda}\lesssim 1.
\end{align}
Consider the oscilatory integral operator $\mathcal{T}_j^\lambda$ in   (\ref{5rpp}). Let
 $G(x_1,\xi)= | \xi+A(x)|$ with $A(x)=(\alpha x_1,\beta x_1)$ and $|(\alpha,\beta)|\ne 0$.
From this  we write the integral   in (\ref{5rg}) as
\begin{align} 
\mathcal{T}_j^\lambda f(x_1,x_2,t)&=\lambda \int e^{2\pi i\lambda \left(\langle  x,  \xi\rangle +tG(x_1,\xi_1,\xi_2)\right)}\chi(t)\chi\left(\frac{tG(x_1,\xi_1,\xi_2)}{\frac{2^j}{\lambda}}\right) \widehat{f}(\xi_1,\xi_2)d\xi_1d\xi_2\nonumber\\
&=\sqrt{\lambda}\int e^{2\pi i \lambda x_2 \xi_2} \left[   \mathcal{S}^{\xi_2}(\widehat{f}(\cdot,\xi_2)) \right](x_1,t) d\xi_2\label{00hh}
 \end{align}
Here
$$ \mathcal{S}^{\xi_2}(h)(x_1,t):=\sqrt{\lambda}\int e^{2\pi i \lambda [x_1 \xi_1+tG(x_1,\xi_1,\xi_2)]} \chi(t)  \chi\left(\frac{tG(x_1,\xi_1,\xi_2)}{\frac{2^j}{\lambda}}\right) h(\xi_1)d\xi_1 $$
It suffices to prove that  there exists $C$  such that for all $h\in L^2(\mathbb{R}^2)$,
\begin{align}\label{ssee}
\|\mathcal{S}^{\xi_2}(h)\|_{L^2(\mathbb{R}^2)}\le\,  C2^{\epsilon j}  \|h\|_{L^2(\mathbb{R})},
\end{align}
since this  with the Plancherel theorem with respect to $x_2$ variable in (\ref{00hh}) implies (\ref{nv2}).
\subsection{Proof of (\ref{ssee})}
For this case the mixed hessian with respect to $x_1\xi_1$ and $t\xi_1$ vanishes simultaneously, we cannot apply directly the H\"{o}rmander  theorem.  
Let us fix $\xi_2=c$ and denote $G(x_1,\xi_1,\xi_2)$ by $G(x_1,\xi_1)=|(\xi_1+\alpha x_1,c+\beta x_1)|$,  we need to control the sizes of derivatives of $G$,
 \begin{align}\label{pnm1}
  \frac{\partial}{\partial \xi_1}G(x_1,\xi_1)  & =
  \frac{ H(x_1,\xi_1)   }{G(x_1,\xi_1)  }   \end{align} 
  where
  $$H(x_1,\xi_1)=\xi_1+\alpha x_1\ \text{and}\ R(x_1,\xi_1) =\alpha ( \xi_1+ \alpha x_1)+ \beta( c+\beta x_1 ).$$
 These two derivatives in (\ref{pnm1}) are bounded by $O(1)$.  
 According to the size of $\frac{\partial}{\partial \xi_1}G(x_1,\xi)=H/G$, we decompose $ \mathcal{S}^{\xi_2} =\sum_{m=0}^{j} \mathcal{S}_{m} $ where  for $2^{-j}< 2^{-m}\le 1$
 \begin{align*}
  \mathcal{S}_{m}(h)(x_1,t) &:=\sqrt{\lambda} \chi(t)\int e^{2\pi i \lambda \left(x_1\xi_1 +tG(x_1,\xi_1)\right)} \chi\left(  \frac{ H(x_1,\xi_1) }{2^{-m} \frac{\lambda}{2^j}} \right)
\chi\left(\frac{tG(x_1,\xi_1) }{\frac{2^j}{\lambda}}\right) h(\xi_1)d\xi_1, 
\\
 \mathcal{S}_{j}(h)(x_1,t) &:=\sqrt{\lambda} \chi(t)\int e^{2\pi i \lambda \left(x_1\xi_1 +tG(x_1,\xi_1)\right)} \psi\left(  \frac{   H(x_1,\xi_1)  }{2^{-j} \frac{\lambda}{2^j}} \right)
\chi\left(\frac{tG(x_1,\xi_1) }{\frac{2^j}{\lambda}}\right) h(\xi_1)d\xi_1, 
\end{align*}
 Here we can take   $H(x_1,\xi_1)>0$ without loss of generality.
We first estimate $ \mathcal{S}_{j}$ above.  In  the   kernel of $\mathcal{S}_j$,    it holds that $|H| \lesssim 2^{-j} 2^j/\lambda$. Thus
  $d\xi_1=O( 1)$. Moreover,
  $dx_1 =O(\frac{2^j}{\lambda})=O(1)$ from $|(\alpha,\beta)|\ne 0$ and (\ref{kyo}). Thus we have $ \|   \mathcal{S}_{j} \|_{op}\lesssim 1$.  It suffices to prove that  for a fixed $m$ in the above,
\begin{align}\label{vbb2}
 \|   \mathcal{S}_m  \|_{op}\lesssim 1.
\end{align}
 The integral kernel $ K(\xi_1,\eta_1)$ of $[\mathcal{S}_{m}]^* \mathcal{S}_{m}$ is
  \begin{align}
 &K(\xi_1,\eta_1)=\lambda  \int e^{2\pi i \lambda (\xi_1-\eta_1)x_1}  e^{2\pi i \lambda\left( G(x_1,\xi)-G(x_1,\eta)\right) t } \Psi(x_1,t,\xi,\eta)
 dx_1dt\label{7j4}
 \end{align}
 where $\Psi(x_1,t,\xi_1,\eta_1)$ is
 \begin{align*}
 \chi(t) \psi(x_1)  \chi\left(  \frac{ H(x_1,\xi_1) }{2^{-m} \frac{\lambda}{2^j}} \right)
\chi\left(\frac{tG(x_1,\xi_1) }{\frac{2^j}{\lambda}}\right)  \chi\left(  \frac{ H(x_1,\eta_1) }{2^{-m} \frac{\lambda}{2^j}} \right)\chi\left(\frac{tG(x_1,\eta_1) }{\frac{2^j}{\lambda}}\right) . 
\end{align*}
In view of (\ref{pnm1})  from the support of (\ref{7j4}), it holds that  
\begin{align}
H=  \xi_1+\alpha x_1 (\text{or}\ \approx \eta_1+\alpha x_1)\approx  2^{-m} \left(\frac{2^j}{\lambda}\right),\  G= |(\xi_1+\alpha x_1,c+\beta x_1)|\approx \frac{2^j}{\lambda}.  \label{kp1}
\end{align}
The derivative of the phase function in (\ref{7j4}) with respect to $t$ is
\begin{align} \label{998} 
 G(x_1,\xi)-G(x_1,\eta_1)&= \frac{G(x_1,\xi)^2-G(x_1,\eta_1)^2}{G(x_1,\xi)+G(x_1,\eta_1)}\nonumber \\
& = \frac{\bigg(H(x_1,\xi)+H(x_1,\eta)\bigg)(\xi_1-\eta_1) }{G(x_1,\xi)+G(x_1,\eta_1)  }  
 \approx 2^{-m}(\xi_1-\eta_1). 
\end{align} 
{\bf Case 1}. Let  $2^{j}/\lambda<C(2^{-m}+ |\alpha|)   $ for  $C\gg 1$. Then 
\begin{itemize}
\item if $\alpha\ne 0$, then $dx_1\lesssim 2^{-m}  \frac{2^j}{\lambda} =O(2^{-m})$ from the first term (\ref{kp1}).
\item if $\alpha=0$, then $dx_1\lesssim \frac{2^j}{\lambda} =O(2^{-m})$  from the second term (\ref{kp1}) and  $2^{j}/\lambda<C 2^{-m} $.
\end{itemize}
With this and (\ref{998}), we apply the integration by parts with respect to $t$ 
to obtain that
 \begin{align}\label{pm33}
& \int   \left| K(\xi_1,\eta_1)  \right|d\xi_1  
 \lesssim \int \frac{ \lambda  2^{-m}    }{(\lambda 2^{-m}|\xi_1-\eta_1|+1)^2}   d\xi_1\lesssim  1.
\end{align} 
{\bf Case 2}. Let  $2^{j}/\lambda\ge C(2^{-m}+ |\alpha|) $ for a large constant $C>0$.  This case occurs only when $\alpha=0$ because   $2^j/\lambda\lesssim 1$ in (\ref{kyo}).  For this case $G(x_1,\xi_1)=|(\xi_1,c+\beta x_1)|$ and $H(x_1,\xi_1)=\xi_1.$
From  the above support  condition  in   (\ref{pnm1}),(\ref{7j4}) and (\ref{998})  with $|G(x_1,\xi)|\approx |G(x_1,\eta_1)|\approx 2^j/\lambda $, it holds that  
\begin{align*} 
 \left| \frac{\partial}{\partial x_1} [G(x_1,\xi_1)-G(x_1,\eta_1)]\right|&=\beta (c+\beta x_1)  \left|\frac{ 1  }{G(x_1,\xi_1)  } -  \frac{ 1  }{G(x_1,\eta_1)  }\right| \nonumber \\
 &\le \frac{2^{-m}|\xi_1-\eta_1| }{ 2^j/\lambda}  \le \frac{|\xi_1-\eta_1|}{C}.  
 \end{align*} 
From  this, we see that  $x_1$-derivative of the phase function in (\ref{7j4}) has the lower bound $$\left|\frac{\partial}{\partial x_1} \left( (\xi_1-\eta_1)x_1+ t(G(x_1,\xi)-G(x_1,\eta_1)) \right)\right| \gtrsim |\xi_1-\eta_1|.$$ 
Moreover, the higher $x_1$-derivatives of the phase function for $N\ge 2$ satisfy  $$\partial_{x_1}^N\left( (\xi_1-\eta_1)x_1+ t(G(x_1,\xi)-G(x_1,\eta_1) )\right)=O \left(\left(\frac{1}{2^j/\lambda}\right)^{N-1}  |\xi_1-\eta_1|\right).   $$
Futhermore, the  $x_1$ derivative of the cutoff functions has the upper bound
$$  \left|\partial^N_{x_1} \chi\left(  \frac{    G(x_1,\xi)  }{  2^j/\lambda } \right)\right|  \lesssim 
\left(\frac{1}{  \frac{2^j}{\lambda} } \right)^N.   $$
Thus, we are able to apply the integration by parts with respect to $dx_1$,
 \begin{align*}
& \int | K(\xi_1,\eta_1) |d\xi_1   \lesssim 
 \int \frac{  \lambda  }{(2^j  |\xi_1-\eta_1|+1)^N}  \times [dx_1\ \text{measure}]\,d\xi_1\lesssim 1
\end{align*}
where $dx_1=O \left(  \frac{ 2^j}{\lambda} \right) $   from the second part of (\ref{kp1}).
This yields (\ref{vbb2}) to finish the proof of  Theorem \ref{th441}.
 \section{Local Smoothing in $L^2$; Proof of Main Theorem \ref{main4}}\label{Sec91}
 In this section, we prove  Main Theorem \ref{main4}. For this 
we verify
 \begin{align} \label{488}
&\|\mathcal{T}_{\rm{annulus}}^\lambda\|_{L^2(\mathbb{R}^{d+1})\rightarrow L^2(\mathbb{R}^{d+1}\times \mathbb{R})}\lesssim_{\epsilon} C\lambda^{c(A)/2}\ \text{implies}\ \nonumber\\
&\qquad\qquad\left\| \mathcal{A}_{S^{d-1}A)} \right\|_{L^2_{\alpha}(\mathbb{R}^{d+1}) \rightarrow L^2(\mathbb{R}^{d+1}) \times [1,2])}\lesssim 1\ 
\text{ for   $\alpha>-(d-1)/2+c(A)/2$} 
\end{align}
and  find the matrix  $A=\left(\begin{matrix}0&1\\ 1&0\end{matrix}\right)$ such  that
 \begin{align}\label{j994}
 \text{$f\rightarrow \mathcal{A}_{S^{1}(A)}f(\cdot,1)$ is unbounded  from  $L^2_{\alpha}(\mathbb{R}^3)$ to $L^2(\mathbb{R}^3)$ for $ -1/4>\alpha>-1/2$.}
 \end{align}
where    $ -1/2 =-\frac{d-1}{2}+\frac{c(A)}{2}$  with $d=2$ and $\frac{c(A)}{2}=0$ since $\text{rank}(JA+(JA)^T)=2$.

\begin{proof}[Proof of (\ref{488})]
It suffices to work with $\alpha\le 0$. Since $t$ is localized as $1\le |t|\le 2$ for $ \mathcal{A}_{S^{d-1}(A)}f(x,x_{d+1},t)$, in (\ref{47kg}), we can localize $x$   as $|x|\le 2$.   
To compare    $|\xi|$ and $|\xi_{3}|$ in the symbol expression of the average in Section \ref{Sec42}, we take  a large $C\ge 1$ and set
  \begin{align}\label{si9}
\psi_1(\xi,\xi_3):= \psi\left(\frac{|\xi|}{C|\xi_3|}\right)\ \text{and}\ 
\psi_2(\xi,\xi_3):= 1-\psi\left(\frac{|\xi|}{C|\xi_3|}\right).
  \end{align}
Fix $\alpha$   above and put 
$$ \widehat{g}(\xi,\xi_{d+1})=(|\xi|+|\xi_{d+1}|+1)^{ \alpha} \widehat{f
 }(\xi,\xi_{{d+1}}).   $$
In view of (\ref{3hf}),  we set the Fourier integral operators $\mathcal{T}_{m^{\alpha}_{s,j} }$  for $s=1,2$  as
\begin{align*}
\mathcal{T}_{m^{\alpha}_{s,j} }g(x,x_{d+1},t) 
&= \psi(x)\chi(t)\int
e^{2\pi i\left( (x,x_{d+1})\cdot(\xi,\xi_{d+1})+t|\xi+\xi_{d+1}A(x)|\right)} \chi\left(\frac{t|\xi+\xi_{d+1}A(x)|}{2^j}\right)\\
&\times \psi_s(\xi,\xi_{d+1}) (|\xi|+|\xi_{d+1}|+1)^{ -\alpha}   \widehat{g}(\xi,\xi_{d+1}) d\xi d\xi_{{d+1}}.
     \end{align*}
     Here the symbol $m^{\alpha}_{s,j}$ is
     \begin{align*}
    & m^{\alpha}_{s,j}(x,x_{d+1},t,\xi,\xi_{d+1})\\
     &=\psi(x)\chi(t) 
e^{2\pi i t|\xi+\xi_{d+1}A(x)| } \chi\left(\frac{t|\xi+\xi_{d+1}A(x)|}{2^j}\right)  \psi_s(\xi,\xi_{d+1}) (|\xi|+|\xi_{d+1}|+1)^{ \alpha}  
     \end{align*}
     In view of the previous reduction to the symbol  in Section \ref{Sec42}, we  control the $L^2$ norm of   $  |\mathcal{A}_{S^{d-1}(A)}f(x,x_{d+1},t) |$ in (\ref{47kg}) by that of the following summation 
    \begin{align}\label{n32}
& \sum_{j=1}^\infty 2^{-(d-1)j/2} ( |\mathcal{T}_{m^{\alpha}_{1,j} }g(x,x_{d+1},t)| + |\mathcal{T}_{m^{\alpha}_{2,j}}g(x,x_{d+1},t)|).
\end{align}
From this    
      combined with $ \left\|
f \right\|_{L^2_{ \alpha}(\mathbb{R}^{d+1}) } = \left\|
g\right\|_{L^2(\mathbb{R}^{d+1}) },$
      it suffices to prove that     for  $s=1,2$,
      \begin{align*}
  \left\| 2^{-(d-1)j/2}\mathcal{T}_{m^{\alpha}_{s,j} }g\right\|_{L^2(\mathbb{R}^{d+1}) \times [1,2])}\lesssim 2^{-\epsilon j}\left\|
g\right\|_{L^2(\mathbb{R}^{d+1}) } 
     \end{align*}
    for   $\alpha>-(d-1)/2+c(A)/2$.   By the similar reduction to a family of the oscillatory integral operators as in (\ref{011}),
      $$\left\| 2^{-(d-1)j/2}\mathcal{T}_{m^{\alpha}_{s,j} } \right\|_{L^2(\mathbb{R}^{d+1})\rightarrow L^2(\mathbb{R}^{d+1}) \times \mathbb{R})}\lesssim \sup_{\lambda} \left\|   2^{-(d-1)j/2} \mathcal{T}^\lambda_{s,j}\right\|_{L^2(\mathbb{R}^{d})\rightarrow L^2(\mathbb{R}^{d+1})}$$
 where for each $s=1,2,$
    \begin{align*}
 \mathcal{T}_{s,j}^{\lambda}g (x,t )&= \lambda^{d/2}  \chi\left(t\right)\psi(x)  \int_{\mathbb{R}^d}
e^{2\pi i\lambda \left(\langle x,  \xi \rangle+t| \xi+A(x)|\right)}
\chi\left(\frac{\lambda t|\xi+A(x)|}{2^j}\right)\\
&\times \psi_s(\lambda \xi,\lambda) (\lambda(1+|\xi|)+1)^{ -\alpha}   \widehat{g}(  \xi)d\xi.\end{align*} 
 Let $s=1$. From the support condition
  $ \psi_1(\lambda \xi,\lambda)=\psi(|\xi|/C)$ in (\ref{si9}) with $|x|\le 2$ and $-\alpha\ge 0$,
\begin{align}\label{0449}
2^j\lesssim \lambda |\xi+A(x)|\lesssim \lambda\ \text{and}\ 
  (\lambda(1+|\xi|)+1)^{ -\alpha} \lesssim   \lambda^{ -\alpha} .
 \end{align}
  If $2^j\ll \lambda$, then by using this and (\ref{545}),
\begin{align}\label{hj1}
 \left\|   2^{-(d-1)j/2} \mathcal{T}^\lambda_{s,j}\right\|_{L^2(\mathbb{R}^{d})\rightarrow L^2(\mathbb{R}^{d+1})}&\lesssim 2^{-(d-1)j/2} \lambda^{ -\alpha} \left(\frac{2^j}{\lambda}\right)^{d/2}\nonumber
 \\
 &\ll  \lambda^{-(d-1)/2-\alpha}  \ll 2^{-((d-1)/2+\alpha) j}
\end{align}
If $2^j\approx\lambda$, then by our hypothesis with (\ref{0449}), we obtain that
\begin{align}\label{hj2}
\nonumber \left\|   2^{-(d-1)j/2} \mathcal{T}^\lambda_{s,j}\right\|_{L^2(\mathbb{R}^{d})\rightarrow L^2(\mathbb{R}^{d+1})}&\lesssim 2^{-(d-1)j/2} \lambda^{-\alpha} \lambda^{c(A)/2+\epsilon}\\
 &\approx  2^{j(-(d-1)/2 -\alpha+c(A)/2+\epsilon)}\le 2^{-j((d-1)/2+\alpha-c(A)/2-\epsilon)}.
\end{align}
Let $s=2$. Then 
  $ \psi_2(\lambda \xi,\lambda)=1-\psi(|\xi|/C)$, which implies $|\xi|\ge C/2\gg 1$ for large $C$. Thus
  $2^j/\lambda \approx |\xi+A(x)|\approx |\xi|\gg 1$ because $|x|\le 2$. Hence,
  $ (\lambda(1+|\xi|)+1)^{ -\alpha} \lesssim | \lambda \xi|^{ -\alpha}\approx 2^{-j\alpha}  $. 
  By this   with the bound $O(1)$ in (\ref{546})
  \begin{align}
\nonumber \left\|   2^{-(d-1)j/2} \mathcal{T}^\lambda_{s,j}\right\|_{L^2(\mathbb{R}^{d})\rightarrow L^2(\mathbb{R}^{d+1})}&\lesssim 2^{-(d-1)j/2}2^{-j\alpha} =2^{-((d-1)/2+\alpha) j}.
\end{align}
 This with (\ref{hj1}) and (\ref{hj2}) for $\alpha>-\frac{d-1}{2}+\frac{c(A)}{2}$ enables us to sum  (\ref{n32}) to obtain (\ref{488}).    \end{proof}

\begin{proof}[Proof of (\ref{j994})]
From the observation that for  $\tilde{f}(x,x_3)=f(x,x_3-x_1x_2)$ and 
$$- \langle A(x),y\rangle -(x_1-y_1)(x_2-y_2)=-x_1x_2 -y_1y_2\ \text{where $A=\left(\begin{matrix}0&1\\ 1&0\end{matrix}\right)$,}  $$
our average can be expressed as
\begin{align*}
\mathcal{A}_{S^1(A)}f(x,x_3+x_1x_2,1)&=\int_{S^1} \tilde{f}(x-y,x_3-y_1y_2)d\sigma(y)\\
&=\int_0^{2\pi} \tilde{f}(x-(\cos\theta,\sin\theta),x_3-\cos \theta\sin\theta)d\theta.
\end{align*}
Thus, it suffices to put
$$\mathcal{A}_{S^1(A)}f(x,x_3,1) 
 =\int_0^{2\pi}f(x-(\cos\theta,\sin\theta),x_3-\cos \theta\sin\theta)d\theta  $$
which is the convolution operator on the Euclidean space having its multiplier
$$ \int_0^{2\pi} e^{2\pi i (\xi,\xi_3)\cdot (\cos\theta,\sin\theta,\cos\theta\sin  \theta)} d\theta.  $$ Note $c(A)/2=0$ since $\text{rank}(JA+(JA)^T)=2$.  We deal with the operator $T:f\rightarrow T(f)$ defined by \begin{align*}
 & \widehat{T f}(\xi,\xi_3)=m(\xi,\xi_3)\widehat{f}(\xi,\xi_3) \ \text{with}
\\
&\qquad \qquad \qquad m(\xi,\xi_3)= \int_0^{2\pi} e^{2\pi i (\xi,\xi_3)\cdot (\cos\theta,\sin\theta,\cos\theta\sin  \theta)} d\theta( |(\xi,\xi_3)|+1)^{-\alpha}.
\end{align*}
Then  from
  $\widehat{f}(\xi,\xi_3)=( |(\xi,\xi_3)|+1)^{-\alpha} ( |(\xi,\xi_3)|+1)^{ \alpha} \widehat{f}(\xi,\xi_3)$, we observe that (\ref{j994})  is equivalent to  the unboundedness of $T$
 from $L^2(\mathbb{R}^3)$ to $L^2(\mathbb{R}^3)$. As $m$ is a continuous multiplier,  it suffices to show that for all $ -1/4> \alpha>-1/2$,
\begin{align}\label{jf1}
&\|T\|_{L^2(\mathbb{R}^3)\rightarrow L^2(\mathbb{R}^3)}\nonumber\\
&\qquad=\sup_{(\xi,\xi_3)\in\mathbb{R}^3}\left|\int_0^{2\pi} e^{2\pi i (\xi,\xi_3)\cdot (\cos\theta,\sin\theta,(1/2)  \sin 2  \theta)} d\theta( |(\xi,\xi_3)|+1)^{-\alpha}\right|=\infty.
\end{align}
Assume the contrary, i.e.,  there  is $C>0$ such that  for all $(\xi,\xi_3)\in \mathbb{R}^3$, 
\begin{align}\label{jf2}
\left|\int_0^{2\pi} e^{2\pi i (\xi,\xi_3)\cdot (\cos\theta,\sin\theta,(1/2)  \sin 2  \theta)} d\theta\right|\le C ( |(\xi,\xi_3)|+1)^{\alpha}\le C |(\xi,\xi_3)|^{\alpha}.
\end{align}
We shall find a contradiction.
Let   $F_{\mathfrak{e}}(\theta)=(\cos \theta,\sin \theta, \frac{1}{2}\sin 2\theta)\cdot \mathfrak{e}$ with $\mathfrak{e}\in S^2$. Find  $\mathfrak{e}$ and the number $\theta=\theta_0$  satisfying  that 
$$F^{(k)}_{\mathfrak{e}}(\theta)\bigg|_{\theta=\theta_0}=0\ \text{for $k=1,2,3$,  i.e.},\ \left(\begin{matrix}-\sin\theta&\cos\theta&\cos 2\theta\\ 
-\cos\theta&-\sin\theta& -2\sin 2\theta\\
\sin\theta&-\cos\theta& -4\cos 2\theta\end{matrix}\right)\bigg|_{\theta=\theta_0}\mathfrak{e}={\bf 0}. $$ 
Obviously, one can check that the above function $ F_{\mathfrak{e}}(\theta)=(\cos \theta,\sin \theta, \frac{1}{2}\sin 2\theta)\cdot \mathfrak{e}$ with  $$\mathfrak{e}=\frac{(1,1,-1/\sqrt{2})}{\sqrt{5/2}}\ \text{and}\ \theta_0=\pi/4$$
satisfies the above equation.   
For $\theta_0=\pi/4$,
From $F^{(k)}_{\mathfrak{e}}(\theta_0)=0$ for $k=1,2,3$, we express  the analytic function
$$F_\mathfrak{e}(\theta)-F_{\mathfrak{e}}(\theta_0)=c_4(\theta-\theta_0)^4+c_5(\theta-\theta_0)^5+\cdots$$
Thus we obtain that
\begin{align}\label{jf3}
\int_0^{2\pi} \psi\left( \frac{F_{\mathfrak{e}}(\theta) -  F_{\mathfrak{e}}(\theta_0)}{\epsilon} \right)d\theta\gtrsim \left|\left\{\theta\in [0,2\pi]: |  F_\mathfrak{e}(\theta)-F_{\mathfrak{e}}(\theta_0) |<\epsilon\right\}\right|\gtrsim \epsilon^{1/4} \end{align} 
By applying the Fourier inversion formula and switching the order of the integral,
\begin{align}\label{jf4}
&\int_0^{2\pi} \psi\left(\frac{ F_{\mathfrak{e}}(\theta)-F_{\mathfrak{e}}(\theta_0)}{\epsilon}\right)d\theta\nonumber\\
&\qquad=\int_{\mathbb{R}}\epsilon \widehat{\psi}(\epsilon \lambda)e^{-2\pi i \lambda F_{\mathfrak{e}}(\theta_0)}\left[\int_0^{2\pi } e^{2\pi   \lambda F_{\mathfrak{e}}(\theta)}d\theta\right]d\lambda
\end{align}
By (\ref{jf2}) with $\lambda=|\xi|$ and $\xi/|\xi|=\mathfrak{e}$ for $F_{\mathfrak{e}}(\theta)=(\cos \theta,\sin \theta, \frac{1}{2}\sin 2\theta)\cdot \mathfrak{e}$, we have
\begin{align*}
RHS\ \text{of (\ref{jf4})} \le C\int_{\mathbb{R}}|\epsilon \widehat{\psi}(\epsilon \lambda)|  \, |\lambda|^\alpha d\lambda=O(\epsilon^{-\alpha})\ \text{where $1/4<-\alpha<1/2$}
\end{align*}
while from (\ref{jf3}),
\begin{align*}
c\epsilon^{1/4}\le LHS\ \text{of (\ref{jf4})} =RHS   \ \text{of (\ref{jf4})} =O(\epsilon^{-\alpha}).
\end{align*}
This is a contradiction for $0<\epsilon\ll 1$. Thus (\ref{jf2}) is not true. Hence we obtain (\ref{jf1}).
\end{proof}

\section{Lower Bounds in Main Theorem \ref{main1}}\label{sec9}
\begin{lemma}
 Suppose  $A\in M_{2\times 2}(\mathbb{R})$. Then there is a constant $c>0$ independent of $A$ so that
\begin{align}\label{nv32}
\|\mathcal{M}^\delta_{S^1(A)}\|_{L^2(\mathbb{R}^3)\rightarrow L^2(\mathbb{R}^3)}\ge c
(\log 1/\delta)^{1/2}.
\end{align}
\end{lemma}
\begin{proof}
Let $e(\theta)=(\cos\theta,\sin\theta)$. We take $f(x,x_3)=\psi(x/\delta)\psi(x_3/C)$ for a large $C>0$ bigger than $10+\max|a_{ij}|$ of $A=(a_{ij})$.  Let
$B:=\{(x,x_3):10\delta\le |x|\le 1, |x_3|\le 1\}$. For each $(x,x_3)\in B$, we choose $t=|x|$ and estimate
\begin{align*}
 \mathcal{A}^{\delta}_{S^1(A)}f(x,x_3,t) & =  \frac{1}{2\pi \delta}\int_{S^1_{\delta}}
f\left(x- ty,x_3 - \langle A(x), t y \rangle \right)dy\approx \frac{1}{2\pi \delta}\times[(\delta/|x|) \times \delta]\approx \frac{\delta}{|x|}
\end{align*}
because,  for each $(x,x_3)\in B$,  the   integral  can be evaluated as the sublevel set measure,
\begin{align*}
&|\{y\in S^1_\delta: |x-|x|y|\le \delta, |x_3-A(x,|x|y)|\le C\}|\\
& \qquad\qquad\approx   |\{e(\theta)\in S^1: |x-|x|e(\theta)|\le \delta\}|\times  \delta\approx  (\delta/|x|)\times \delta.
\end{align*}
Thus
\begin{align*}
\int_{B}| \mathcal{A}^{\delta}_{S^1(A)}f(x,x_3,|x|) |^2dxdx_3&\gtrsim   \int_{(x,x_3)\in B}  (\delta/|x|)^2 dx dx_3\\
&\approx C\delta^2 \int_{10\delta<|x|<1} |x|^{-2}dx\approx C\delta^2\log (1/\delta) 
\end{align*}
while 
$$\int |f(x,x_3)|^2dxdx_3\approx C\delta^2.$$
Therefore we have the lower bound $(\log 1/\delta)^{1/2}$ for the maximal average $\mathcal{M}_{S^1(A)}^\delta$.
\end{proof}
This lemma gives the lower bound  $(\log 1/\delta)^{1/2}$ for the case  $\text{rank}((JA)+(JA)^T)=2$,  and for the case $\text{rank}(A)\le 1$ in Main Theorem \ref{main1}. The following lemma gives the lower bound for  $\text{rank}((JA)+(JA)^T)=1$ with $\text{rank}(A)=2$ in Main Theorem \ref{main1}.
\begin{lemma}
Recall $\mathcal{M}^\delta_{S^1(A)} $ in (\ref{0g}). Suppose that  
$\text{rank}((JA)+(JA)^T)=1$ and $\text{rank}(A)=2$.  Then there exists $c>0$ such that
\begin{eqnarray}
\| \mathcal{M}^\delta_{S^1(A)}\|_{L^p(\mathbb{R}^3)\rightarrow L^p(\mathbb{R}^p)}&\ge
c \delta^{1/3-1/p}.\label{100932}
 \end{eqnarray}
If $p=2$, this     is the lower bound  $\delta^{-1/6}$  of Main Theorem \ref{main1} for  the case $\text{rank}((JA)+(JA)^T)=1$   and $\text{rank}(A)=2$.
 \end{lemma}
\begin{proof}[Proof of (\ref{100932})]
In view of Proposition \ref{lem10044} and Lemma \ref{lem2.2}, it suffices to regard $A$ as $I_c=\left(\begin{matrix}
1&c\\
0&1
\end{matrix}\right)$.  Let $(x,x_3)\in \mathbb{R}^2\times\mathbb{R}$. Then our average  is given by
\begin{align*}
 \mathcal{A}^{\delta}_{S^1(A)}f(x,x_3,t) & =  \frac{1}{2\pi \delta}\int_{S^1_{\delta}}
f\left(x- ty,x_3 - \langle I_c x, t y \rangle \right)dy\nonumber \\
&= \frac{1}{2\pi \delta}\int_{S^1_{\delta}}
f\left(x-ty,x_3 -  \langle x, t y \rangle -cx_2 ty_1\right)dy.
\end{align*}
By the change of variable  $f(x,x_3)=\tilde{f}(x,x_3-|x|^2/2)$, it suffices to work with 
\begin{align}
 \mathcal{A}^{\delta}_{S^1(A)}f(x,x_3,t) & =  \frac{1}{2\pi\delta}\int_{S^1_{\delta}}
\tilde{f}\left(x- ty,x_3-\frac{|x|^2}{2} -\frac{|ty|^2}{2}-cx_2ty_1 \right)dy.\label{4004}
\end{align}
Change of variable $x_3-|x|^2/2\rightarrow x_3$ on the last component, and 
set $B:=\{(x,x_3): |x_1|\le 5,\ |x_2|\le \delta^{1/3}\ \text{and}\ 1\le x_3 \le 2\} $. 
Choose 
   \begin{align}\label{func}
\tilde{f}(u_1,u_2,u_3)&:=\psi\left(\frac{u_1}{10}\right)\psi\left(\frac{u_2}{10\delta^{1/3}}\right)\psi\left(\frac{u_3}{10\delta}\right).
 \end{align} 
 We  express $y\in S^1_{\delta}$ in (\ref{4004}) as $ y=e(\theta)+O(\delta)$ for  $\theta\in [0,2\pi]$ where $e(\theta)=(\cos \theta,\sin \theta)$.  By  taking the measure of the thicknes $\delta$,  for $(x,x_3)\in B$ and $|t|\approx 1$, it holds that $\mathcal{A}^{\delta}_{S^1(A)}f(x,x_3,t)$ is bounded away from
\begin{align*} 
 & \frac{1}{2\pi} \int_{0}^{2\pi}
\tilde{f}\left(x_1-t\cos \theta +O(\delta), x_2-t\sin\theta+O(\delta),x_3 - \frac{t^2 }{2}-cx_2t\cos \theta+O(\delta) \right)d\theta
\\
&\quad\qquad\ge \frac{1}{2\pi}  \int_{0}^{\delta^{1/3}}
\tilde{f}\left(x_1-t  +O(\delta^{2/3}), x_2-t\theta+O(\delta),x_3 -\frac{ t^2}{2} -cx_2t +O(\delta) \right)d\theta 
\end{align*}
where we used $|x_2|\le \delta^{1/3}$, $\cos \theta=1-O(\theta^2)$ and $\sin\theta=\theta+O(\theta^3)$ for $|\theta|\le \delta^{1/3}$.
 Next choose $t=t(x,x_3)$ for each $(x,x_3)\in B$ satisfying
  $$  x_3-t^2/2-cx_2t=0.$$
For this $t=t(x,x_3)$ which is $\approx 1$, from the support condition of $B$ and (\ref{func}),  it holds  that the above integral (a lower bound of $\mathcal{A}^{\delta}_{S^1(A)}f(x,x_3,t)$) where $\tilde{f}=1$,    is bounded away from
$
  \frac{\delta^{1/3}}{2\pi}\ \text{if $(x,x_3)\in B$}.
$
From this lower bound  combined with the measure $|B| \ge  \delta^{1/3}$, we see that for $t=t(x,x_3)$ chosen as above,
 $$\|\mathcal{M}^\delta_{S^1(A)} f\|_{L^p(\mathbb{R}^3)}^p\ge \int_B |\mathcal{A}^{\delta}_{S^1(A)}f(x,x_3,t(x,x_3))|^pdx dx_3\ge   \delta^{1/3}   \left( \frac{\delta^{1/3}}{2\pi}\right)^p$$
 while 
 $$\|f\|_{L^p(\mathbb{R}^3)}^p=  \int |\tilde{f}(x,x_3)|^p dxdx_3\le 10^3 \delta^{1/3}\delta.$$
 Therefore, we obtain the desired lower bound as
 \begin{eqnarray*} 
\| \mathcal{M}^\delta_{S^1(A)}  \|^p_{L^p(\mathbb{R}^3)\rightarrow L^(\mathbb{R}^3)} 
\gtrsim \frac{\delta^{p/3}}{\delta}.
\end{eqnarray*}
This implies (\ref{100932}).
    \end{proof}
 
    The following lemma gives the lower bound for the case that $\text{rank}((JA)+(JA)^T)=0$ with $\text{rank}(A)=2$ in Main Theorem \ref{main1}.

\begin{lemma}
   If
$\text{rank}((JA)+(JA)^T)=0$ and $\text{rank}(A)=2$, then there is a constant $c>0$ such that
\begin{align}\label{nv32}
\|\mathcal{M}^\delta_{S^1(A)}\|_{L^p(\mathbb{R}^3)\rightarrow L^p(\mathbb{R}^3)}\ge c
 \delta^{-1/p}\ \text{for all $1\le p<\infty$ }
\end{align}
When $p=2$, this is  the lower bound $\delta^{-1/2}$ of Main Theorem \ref{main1} for the case $\text{rank}((JA)+(JA)^T)=0$ and $\text{rank}(A)=2$. 
\end{lemma}
\begin{proof}[Proof of (\ref{nv32})]
For this case    $A =I$ from (1-4) in Lemma \ref{lem10044}. Set $f(x,x_3)=\tilde{f}(x,x_3-\frac{1}{2}|x|^2)$ 
as 
\begin{eqnarray*}
 \mathcal{A}^{\delta}_{S^1(A)}f\left(x,x_3,t \right)&=&\frac{1}{2\pi\delta}\int_{S^1_{\delta}} f \left(x-ty,x_3 -\langle I(x), ty\rangle \right)dy 
 \\
 &=&  \frac{1}{2\pi \delta}\int_{S^1_{\delta}} \tilde{f}\left(x-ty,x_3-  \frac{1}{2}|x|^2-\frac{1}{2}|ty|^2   \right)dy.  
\end{eqnarray*}
Change of variable $x_3-|x|^2/2\rightarrow x_3$ on the last component, and set a region $B =\left\{(x,x_3):\    |x|\le
1, \
1\le x_3\le 2 \right\}  $. Next,  take
 \begin{align}\label{func4}
\tilde{f}(u_1,u_2,u_3)=\psi\left(u_1/10\right)\psi\left(u_2/10\right)
\psi\left(\frac{u_3}{100\delta}\right).\end{align}
 For each $(x,x_3)\in B $,
choose $t$ such that  
$
\frac{1}{2}t^2= x_3 
$
so that
\begin{align}
 x_3
-\frac{1}{2}|ty|^2 = x_3(1-|y|^2) \le 10\delta\ \text{where}\ y\in S^1_{\delta}. \label{ddta4}
\end{align}
This implies that for $t=\sqrt{2x_3}$ in (\ref{ddta4}) on the support condition  (\ref{func4}),   it holds that
 \begin{align*}
  \mathcal{A}^{\delta}_{S^1(A)}f(x,x_3,t)\ge 1
  \ \text{for $(x,x_3)\in B$ with $B$ in (\ref{func4}).
}
   \end{align*}
Hence  it holds that
 \begin{eqnarray}\label{pbe3}
 \|\mathcal{M}^\delta_{S^1(A)} f\|_{L^p(\mathbb{R}^3)}^p\ge \int | \mathcal{A}^{\delta}_{S^1(A)}f(x,x_3,\sqrt{2x^3})|^p
dxdx_3\ge 1.
\end{eqnarray}
From (\ref{func4}), we have $\| f \|^p_{L^p(\mathbb{R}^3)}\le  10^4\delta.
$
This combined with (\ref{pbe3}) implies that  $$
\| \mathcal{M}^\delta_{S^1(A)}\|^p_{L^p(\mathbb{R}^3)\rightarrow
L^p(\mathbb{R}^3)}\gtrsim \delta^{-1}
$$ which gives the desired lower bound for (\ref{nv32}).\end{proof}

\subsection{Final Remark}
To extend the result of the main theorem 1 to the higher dimension, we need to  generalize   $ \text{rank}\left(JA+(JA)^T\right)$. This number is related with  the multiplicities of  real eigenvalues of $A$  in Proposition \ref{lem10044}.  In our sub-sequential paper, we shall obtain the range of $p$ for $ \mathcal{M}_{S^{d-1}(A)}  $ to be bounded in $L^p(\mathbb{R}^{d+1})$  in terms of the multiplicities of  real eigenvalues of $A$ under the assumption that $A$ is diagonalizable.

\section{Appendix}\label{sec10}
\subsection{Proof of (\ref{kcc})}
To  prove   (\ref{kcc}).  We first observe that  $d\sigma_0$ is majorized by a rapidly decreasing function. So, we can replace $\sigma_0(y)$ with $\psi(y)$. From this combined with  the relation (\ref{0ch}),    it suffices to prove the $L^p$ boundedness of  the mapping $f\rightarrow \sup_{j\in\mathbb{Z}}|f|*_A \mu_j$, where the convolution is defined by   \begin{align*}
|f|*_A\mu_j(x,x_{d+1}): 
 = \int_{\mathbb{R}^d}  |f|\left(x-y,x_{d+1}-\langle A(x-y), y\rangle\right)\psi\left(\frac{y}{2^j}\right) \frac{1}{ 2^{dj}} dy. 
 \end{align*}
  Before dealing with the maximal average associated with $\mu_j$, we  treat the convolution with a less singular measure $\nu_j$ defined as
 \begin{align*}
&f*_A \nu_j(x,x_{d+1})\\
&\qquad: 
= \int_{\mathbb{R}^d}  f\left(x-y,x_{d+1}-(y_{d+1}+\langle A(x-y),y\rangle )\right)\psi\left(\frac{y}{2^j}\right) \frac{1}{ 2^{dj}} \psi\left(\frac{y_{d+1}}{2^j}\right)  \frac{1}{2^{2j}}dydy_{d+1}.
 \end{align*}
By change of variables, it can be expressed as 
 \begin{align*}
 \int_{\mathbb{R}^d}  f\left(y,y_{d+1}+\langle A(y),y\rangle\right)\psi\left(\frac{x-y}{2^j}\right) \frac{1}{ 2^{dj}} \psi\left(\frac{x_{d+1}-(y_{d+1}+\langle A(y),x\rangle)}{2^j}\right)  \frac{1}{2^{2j}}dydy_{d+1}.
 \end{align*} 
Then we have the well known result as
 \begin{lemma}  For any $A\in M_{d\times d}(\mathbb{R})$, there exists $C>0$ independent of  $f\in L^1(\mathbb{R}^{d+1})$ such that
$$ \left\|\sup_{j\in\mathbb{Z}} |f|*_A\nu_j\right\|_{L^{1,\infty}(\mathbb{R}^{d+1})}\le C \|f\|_{L^1(\mathbb{R}^{d+1})}.$$
An interpolation with  its $L^\infty(\mathbb{R}^{d+1})$ bound  yields the $L^p(\mathbb{R}^{d+1})$ bound  for $1<p<\infty$.
 \end{lemma}
 \begin{proof}
Let $r>0$ and   $B$ be a  bilinear form $(x,y)\rightarrow B(x,y)=\langle x,A(y)\rangle$. In view of the integral kernel of $f*_A\nu_j$ above,  with  a non-isotropic dilation $(y,y_{d+1})\rightarrow (ry_1,\cdots,ry_d,r^2y_{d+1})$ and a variable hyperplanes $ (x,x_{d+1})+\{(y,B(x,y)):y\in\mathbb{R}^d\}$, we define a ball centered at $(x,x_{d+1})$ with a radius $r$ as
 $$D_r(x,x_{d+1})=\{(y,y_{d+1}):|y-x|<r, |y_{d+1}-(x_{d+1}-B(x,y) |<r^2\}.$$
  We define
the maximal average associated with these balls  as \begin{align*}
\mathcal{N}(f)(x)=\sup_{r>0}\frac{1}{|D_r(x,x_{d+1})|}\int_{D_r(x,x_{d+1})} |f(y,y_{d+1})|dydy_{d+1}.
\end{align*}
Then, from the  expression of $f*_A\nu_j$   and $\mathcal{U}_Af(y,y_{d+1})=f(y,y_{d+1}+\langle A(y),y\rangle)$ in (\ref{0ch}), it holds that
\begin{align}\label{40p}
\sup_{j\in\mathbb{Z}} |f|*_A \nu_j(x,x_{d+1})\le C \mathcal{N}(\mathcal{U}_Af)(x,x_{d+1}).
\end{align}
  Moreover, the  following three properties hold   regarding the   balls $D_r(x,x_{d+1})$, 
 \begin{itemize}
 \item[(1)] $|D_r(x,x_{d+1})|=2c_dr^{d+2}$.
 \item[(2)] $r_1\le r_2$ implies $D_{r_1}(x,x_{d+2})\subset  D_{r_2}(x,x_{d+1})$.
 \item[(3)] Let   $\tilde{D}_r(x,x_{d+1})= \bigcup_{ D_r(y,y_{d+1})\cap D_r(x,x_{d+1})\ne \emptyset }D_r(y,y_{d+1})  $.   Then there exists $C>0$ independent of a center $(x,x_{d+1})$ and a dimension $d$ such that $$\tilde{D}_r(x,x_{d+1})\subset D_{Cr}(x,x_{d+1}). $$
 \end{itemize} 
One can obtain   (1) and (2) evidently with   $c_d$  the volume of a unit ball in $\mathbb{R}^d$. We give a proof of    (3),  which is a   Vitali-type covering property for the non-isotropic and  non-Euclidean  balls.
 \begin{proof}[Proof of (3)]
  Let $z\in \tilde{D}_r(x,x_{d+1})$. Then in view of the definition of $\tilde{D}_r(x,x_{d+1})$, there exists $(y,y_{d+1})$ such that $z\in D_r(y,y_{d+1})$ and
  $   D_r(y,y_{d+1})\cap D_r(x,x_{d+1})\ne \emptyset $. Take $w\in D_r(y,y_{d+1})\cap D_r(x,x_{d+1})$. Then $ |y-w|,|w-x|<r\ \text{and}\ |z-y|<r$. This implies $ |z-x|<3r<Cr$.
Rewrite
 \begin{align*}
 z_{d+1}&-(x_{d+1}-B(x,z)) = [ z_{d+1}-(y_{d+1}-B(y,z) ]+[(y_{d+1}-B(y,w))-w_{d+1}]\\
 &+[w_{d+1}-(x_{d+1}-B(x,w)]+[    +B(x,z)-B(y,z)+B(y,w)-B(x,w)]
 \end{align*}
 where the first three terms are controlled by $O(r^2)$ and the last term $[   B(x,z)-B(y,z)+B(y,w)-B(x,w)]=-B(x-y,w-z)$ satisfies $|B(x-y,w-z)|\le \|A\| |x-y| |w-z|\le \|A\| 4r^2$. Then
 we obtain that 
 $$  |z_{d+1}-(x_{d+1}-B(x,z))|< 3r^2+4\|A\| r^2<(Cr)^2.  $$
 Hence we obtain $z\in  D_{Cr}(x,x_{d+1}) $   with radius $Cr$ for $C=2\sqrt{3+\|A\|}$.
 \end{proof}
We can obtain the weak type (1,1) boundedness of the maximal operator $\mathcal{N}$  by utilizing   the standard argument, based on the three properties (1)-(3) analogous to those of the Euclidean balls defining the Hardy-Littlewood maximal function.  
 \end{proof}

  \begin{lemma}\label{lemi22}   For any $d\times d$ matrix $A$, there exists $C$  independent of $f$ such that
\begin{align}\label{i22}
 \left\|\sup_{j\in\mathbb{Z}} |f|*_A \mu_j\right\|_{L^{p}(\mathbb{R}^{d+1})}\le C  \|f\|_{L^p(\mathbb{R}^{d+1})}\ \text{for all $f\in L^p(\mathbb{R}^{d+1})$  for $1<p<\infty$.}
 \end{align}
 \end{lemma}
 In order to prove Lemma \ref{lemi22}, we can replace $\psi$ by   $\varphi$ supported in $|y|\approx 1$  for defining $\mu_j$ in the above. Let $f\in \mathcal{S}(\mathbb{R}^{d+1})$. Then for each  fixed $\lambda\in \mathbb{R}$, we 
define $$ \widehat{f}^{d+1}(x,\lambda)=\int e^{-2\pi i \lambda x_{d+1}} f(x,x_{d+1})dx_{d+1}$$   as the Fourier transform of $f$ with respect to  the single variable $x_{d+1}$.  To restrict the  frequency $\lambda$ as $\lambda 2^{2j}\approx 1$, we set
 $$f*_AQ_j(x,x_{d+1})=\int  e^{2\pi i\lambda x_{d+1}}\chi(\lambda 2^{2j})  \widehat{f}^{d+1}(x,\lambda)d\lambda.
 $$
 so that $\sum_{j\in\mathbb{Z}}  f*_AQ_j=f$. For this case, we observe that $*_A$ is same as the Euclidean convolution $*$.  
 We can assume that $f\ge 0$ in Lemma \ref{lemi22}.
 Split $$  f*_A\mu_j=\left[\sum_{\ell<0}  f*_AQ_{j+\ell}*_A\mu_j\right]+\left [\sum_{\ell\ge 0}  f*_AQ_{j+\ell}*_A(\mu_j-\nu_j)\right]+\left[\sum_{\ell\ge 0}  f*_AQ_{j+\ell}*_A\nu_j\right]. $$
Let  $ U_j:=\sum_{\ell\ge 0}Q_{j+\ell} $.  Since
 $f*_A\left|U_j\right| (x,x_{d+1})$ is majorized by the Hardy-Littlewood maximal operator along the $x_{d+1}$ axis,  the composition of the two bounded maximal operator in $L^p$ gives  
 $$\left\|\sup_j  f*_A\sum_{\ell\ge 0} Q_{j+\ell}*_A\nu_j  \right\|_{L^p}\le \left\| \sup_j\left(\sup_{j} |f|*_A |U_j|\right)*_A \nu_j\right\|_{L^p}\le C \left\|   f\right\|_{L^p}. $$
 Next, we can obtain Lemma \ref{lemi22}  by showing the following two square sum estimates:
\begin{align}
 \left\|\left(\sum_{j\in\mathbb{Z}} \left|f*_A Q_{j+\ell}* _A\mu_j  \right|^2\right)^{1/2}\right\|_{L^p(\mathbb{R}^{d+1})}
\lesssim 2^{-c|\ell|}\|f\|_{L^p(\mathbb{R}^{d+1})}\  \text{for $\ell<0$}, \label{i23}\\
 \left\|\left(\sum_{j\in\mathbb{Z}} \left| f  *_A Q_{j+\ell}* _A(\mu_j-\nu_j) \right|^2\right)^{1/2}\right\|_{L^p(\mathbb{R}^{d+1})}
\lesssim 2^{-c|\ell|}\|f\|_{L^p(\mathbb{R}^{d+1})}\  \text{for $\ell\ge 0$}.\label{i24}
 \end{align}
For $p=2$, we first take the Fourier transform of $f*_AQ_{j+\ell}*_A \mu_j   $ and $f *_AQ_{j+\ell}*_A\nu_j  $
along $x_{d+1}$ axis. Then we need to show the uniform $L^2$ estimate in $\lambda$ for the square sum of 
\begin{align*}
M_j^\lambda f(x)&=\chi(\lambda 2^{2(j+\ell)})\int_{\mathbb{R}^d}e^{2\pi i\lambda A(y)\cdot (x-y)} \varphi_{2^{j}}(x-y)   \widehat{f}^{d+1}(y,\lambda)dy\ \text{and}\ \\
N_j^\lambda f(x)&=\chi(\lambda 2^{2(j+\ell)})\widehat{\psi}(\lambda 2^{2j})\int_{\mathbb{R}^d}e^{2\pi i\lambda A(y)\cdot (x-y)} \varphi_{2^{j}}(x-y)   \widehat{f}^{d+1}(y,\lambda)dy.
\end{align*}

 \begin{proof}[Proof of (\ref{i23}) for $p=2$]
For this case $\lambda 2^{2j}\approx 2^{-2\ell}\ge 1$ with $\ell<0$. Then we can prove the decay estimate 
 $\|[M_j^{\lambda}]^*M_j^{\lambda}\|_{op} =O(1/(\lambda 2^{2j})^{\epsilon})=O( 2^{-2\epsilon |\ell|}).$
For a fixed $\ell$, we can sum these estimates over $j$ because $ \chi(\lambda 2^{2(j+\ell)})\ne 0 $ for at most five $j$'s. Hence we can
obtain (\ref{i23}) for $p=2$.
 \end{proof}
  \begin{proof}[Proof of (\ref{i24}) for $p=2$]
We obtain (\ref{i24}) for $p=2$  from the mean value theorem estimate
 $\|M_j^{\lambda}-N_j^{\lambda}\|_{op} =O(\lambda 2^{2j})=O( 2^{-2  |\ell|})$ and sum those over $j$ because  $ \chi(\lambda 2^{2(j+\ell)})\ne 0 $ for at most five $j$'s.
 \end{proof}
 \begin{proof}[Proof of (\ref{i23}) and (\ref{i24}) for $p<2$]
  The case $p=2$ of  (\ref{i23}) and (\ref{i24})  shows the $L^2$ boundedness of the maximal function $f\rightarrow \sup_j\mu_j*f$. 
 Let   $\tilde{p}_0=2/(1+1/p_0)$.  To treat $p<2$, we utilize the following property
 \begin{align}\label{i91}
&  \left\|\sup_{j\in\mathbb{Z}}|f| * \mu_j\right\|_{L^{p_0}}\lesssim  \|f\|_{L^{p_0}} \nonumber\\
  &\qquad\Rightarrow  \left\|\left(\sum |f_j|*\mu_j|^2\right)^{1/2}\right\|_{L^p}\lesssim\left\|\left(\sum |f_j|^2\right)^{1/2}\right\|_{L^p}\ \text{for $p>  \tilde{p}_0$}
 \end{align}
together with the Littlewood-Paley inequality for $f_j=f * Q_j$ above such that
 \begin{align*} 
   \left\|\left(\sum |f * Q_j|^2\right)^{1/2}\right\|_{L^p}\lesssim\left\|f\right\|_{L^p} \ \text{for $1<p<\infty$}.
 \end{align*}
Take $p_0=2$ in (\ref{i91}). This leads the vector valued inequality (\ref{i91}) for $p>\tilde{p}_0$, which proves (\ref{i23}) and (\ref{i24}) for the same range $p$. Thus we obtain the $L^p$ boundedness of $f\rightarrow \sup_j |f|*\mu_j$  for $p>\tilde{p}_0=4/3$ for $p_0=2$. This again implies  (\ref{i91}) for $p>\tilde{p}_1$ with $ p_1$ near $4/3$. By repeating this argument, we    cover the full range of $p>1$.
 \end{proof}

\subsection{Proof of Lemma \ref{prop23} }\label{app11.2}
  For   $(j,k)\in\mathbb{Z}_+\times\mathbb{Z}$,   we show (\ref{yb1}),
\begin{align*} 
\left\|2^{k/p}\mathcal{T}_{m_{j,k}}\right\|_{L^p(\mathbb{R}^{{d+1}})\rightarrow L^p(\mathbb{R}^{{d+1}}\times\mathbb{R})} = \left\|\mathcal{T}_{m_{j,0}}\right\|_{L^p(\mathbb{R}^{{d+1}})\rightarrow L^p(\mathbb{R}^{{d+1}}\times\mathbb{R})} 
\end{align*} 
  Recall $\mathcal{T}_{m_{j,k}}f(x,x_{d+1},t)$ in (\ref{3hf}) where 
\begin{align} \label{i7}
m_{j,k}(x,x_{d+1},t,\xi,\xi_{d+1})&=e^{2\pi i t |\xi+\xi_{{d+1}}A(x)|}   \chi\left(\frac{t|\xi+\xi_{{d+1}}A(x)|}{2^j} \right)\chi(2^kt) .  \end{align}
Given $p\ge 1$, we set the dilations as
\begin{align}
[\mathcal{D}^{p}_{2^{-k}}g](x,x_{d+1},t)&=2^{-(d+3)k/p}g(2^{-k}x,2^{-2k}x_{d+1},2^{-k}t),\label{4bv}\\
[D^{p}_{2^{-k}}f](x,x_{d+1})&=2^{-(d+2)k/p}f(2^{-k}x,2^{-2k}x_{d+1})\nonumber
\end{align}
satisfying the following two  $L^p$-invariance conditions:
\begin{align}\label{nk}
\|g\|_{L^p(\mathbb{R}^{{d+1}}\times \mathbb{R})}=
\|\mathcal{D}^{p}_{2^{-k}}g\|_{L^p(\mathbb{R}^{{d+1}}\times
 \mathbb{R})}\ \text{and}\ \|f\|_{L^p(\mathbb{R}^{{d+1}})}=
 \|D_{2^{-k}}f\|_{L^p(\mathbb{R}^{{d+1}})}.
 \end{align}
In (\ref{i7}),   we can check that
\begin{align}\label{i0}
m_{j,k}(2^{-k}x,2^{-2k}x_{d+1},2^{-k}t,2^{k}\xi,2^{2k}\xi_{d+1})= m_{j,0}( x, x_{d+1}, t,\xi,\xi_{d+1}).
\end{align}
Then     it holds that
\begin{align}
&\mathcal{D}^{p}_{2^{-k}}[2^{k/p} [\mathcal{T}_{m_{j,k}}] f](x, x_{d+1},t)=   \mathcal{T}_{m_{j,0}} D^p_{2^{-k}} f(x,x_{d+1},t).
\label{nka}
\end{align} 

\begin{proof}[Proof of (\ref{nka})]
We show (\ref{nka}) for the symbols $m_{j,k}$ satisfying   (\ref{i0}).  We recall   (\ref{3hf}),  
\begin{align*}
\mathcal{T}_{m_{j,k}}  f(x,x_{d+1},t)&=   \int
e^{2\pi i (x,x_{d+1})\cdot(\xi,\xi_{d+1}) } m_{j,k}( x, x_{d+1}, t,\xi,\xi_{d+1})
\widehat{f}(\xi,\xi_{d+1})d\xi d\xi_{d+1}. 
\end{align*}
By using $ \mathcal{D}^{p}_{2^{-k}}$ in (\ref{4bv}), we write
\begin{align*}
 \mathcal{D}^{p}_{2^{-k}}[2^{k/p}\mathcal{T}_{m_{j,k}}  f](x, x_{d+1},t)&= 2^{-(d+2)k/p}      \int
e^{2\pi i\left((2^{-k}x,2^{-2k}x_{d+1})\cdot(\xi,\xi_{d+1}) \right)}\nonumber\\
&\times m_{j,k}(2^{-k}x,2^{-2k}x_{d+1},2^{-k}t, \xi, \xi_{d+1})
\widehat{f}( \xi, \xi_{d+1})d\xi d\xi_{d+1}.
 \end{align*}
Apply the change of variable $ (\xi,\xi_{d+1}) \rightarrow (2^{k}\xi,2^{2k}\xi_{d+1}) $. Then
the above integral becomes
\begin{align*}
&   2^{-(d+2)k/p}    2^{(d+2)k/p}  \int
e^{2\pi i\left((2^{-k}x,2^{-2k}x_{d+1})\cdot(2^k\xi,2^{2k}\xi_{d+1}) \right)}\nonumber\\
&\times m_{j,k}(2^{-k}x,2^{-2k}x_{d+1},2^{-k}t, 2^k\xi, 2^{2k}\xi_{d+1})
 2^{(d+2)k/p'} \widehat{f}(2^k\xi, 2^{2k}\xi_{d+1})d\xi d\xi_{d+1} 
 \end{align*}
where $   2^{(d+2)k/p'} \widehat{f}(2^k\xi, 2^{2k}\xi_{d+1}) =\left[D^p_{2^{-k}} f\right]^{\vee}
 (\xi,\xi_{d+1}) $ due to  (\ref{4bv}).  This with (\ref{i0}) implies
\begin{align*}
 & \mathcal{D}^{p}_{2^{-k}}[2^{k/p}\mathcal{T}_{m_{j,k}}  f](x, x_{d+1},t)\\
 &\qquad\quad= \int
 e^{2\pi i\left((x,x_{d+1})\cdot(\xi,\xi_{d+1}) \right)}
m_{j,0}(x,x_{d+1},t, \xi, \xi_{d+1}) \left[D^p_{2^{-k}} f\right]^{\vee}
 (\xi,\xi_{d+1})d\xi d\xi_{d+1}
 \end{align*}
which is $ \mathcal{T}_{m_{j,0}} D^p_{2^{-k}} f(x,x_{d+1},t).$  This implies (\ref{nka}).  \end{proof}
  From (\ref{nka}) with (\ref{nk}),  it holds that
  \begin{align*} 
\left\|2^{k/p}\mathcal{T}_{m_{j,k}}\right\|_{L^p(\mathbb{R}^{{d+1}})\rightarrow L^p(\mathbb{R}^{{d+1}}\times\mathbb{R})} =\left\|\mathcal{T}_{m_{j,0}}\right\|_{L^p(\mathbb{R}^{d+1})\rightarrow L^p(\mathbb{R}^{d+1} \times \mathbb{R})}.
\end{align*} 
Therefore,
we have completed the proof of Lemma \ref{prop23}.

\subsection{The operator norm of $\mathcal{T}_j^\lambda $ for the case $2^j/\lambda\not \approx 1$ for the case $\text{rank}(A)<d$}
In Proposition \ref{pr02}, we treated the estimates of $\|\mathcal{T}_j^\lambda\|_{op}$ for the case $2^j\gg \lambda$ and $2^j\ll \lambda$ when $\text{rank}(A)= d$. We now deal with the case $\text{rank}(A)<d$ as we mentioned in Remark \ref{rek1}.
\begin{proposition}\label{prop111} Suppose   $\det(A)=k<d$. Then  that there exists $C\gg 1$ such that
\begin{align}\label{ihd5}
 \left\|\mathcal{T}_j^\lambda   \right\|_{L^2(\mathbb{R}^d)\rightarrow L^2(\mathbb{R}^{d+1}) }&\lesssim  \begin{cases} 1\   \ \text{ if $\left|\frac{2^j}{\lambda}\right|\ge C$, }\\
\left(\frac{2^j}{\lambda}\right)^{k}  \lambda^{k/2}   \ \text{ if $\left|\frac{2^j}{\lambda}\right|\le 1/C$.}
\end{cases}
\end{align}
\end{proposition}
\begin{proof}[Proof of (\ref{ihd5})] Since $\det(A)=k$, there exists a nonsingular matrix $R\in M_{d\times d}(\mathbb{R})$ (composition of column operations)  such that
 $AR(x)=(E(x_1,\cdots,x_k),{\bf 0})$ where $E\in M_{k\times k}(\mathbb{R})$ with $\text{rank}(E)=k$ and $\det(R)=1$. Set $$G(x_1,\cdots,x_k,\xi)=|(R^{-1})^*\xi+(E(x_1,\cdots,x_k),{\bf 0})|. $$
Apply the   change of variable  $\xi\rightarrow (R^{-1})^*\xi$ for the integral $\mathcal{T}_j^\lambda g(x,t)$ in (\ref{5rpp}). Then $\mathcal{T}_{j}^{\lambda}g(Rx,t )$ can be written as
  \begin{align*}
& \lambda^{d/2} \int_{\mathbb{R}^d}
e^{2\pi i\lambda \left(\langle Rx, (R^{-1})^*\xi\rangle+tG(x_1,\cdots,x_k,\xi)\right)}
  \chi\left(t\right)  \chi\left(\frac{ tG(x_1,\cdots,x_k,\xi)}{2^j/\lambda}\right)  \widehat{g
 }( (R^{-1})^*\xi)d\xi 
 \end{align*}
 where  $\langle Rx, (R^{-1})^*\xi\rangle=\langle x,\xi\rangle.$ 
 To each $(\xi_{k+1},\cdots,\xi_d)$, we assigne an operator  defined as
 \begin{align}
 \left[\mathcal{S}^{(\xi_{k+1},\cdots,\xi_d)}h\right](x_1,\cdots,x_k,t)&= \lambda^{k/2} \int_{\mathbb{R}^d}
e^{2\pi i\lambda \left(\langle  (x_1,\cdots,x_k),(\xi_1,\cdots,\xi_k)\rangle+tG(x_1,\cdots,x_k,\xi)\right)}
 \nonumber \\
  &\times  \chi\left(t\right)  \chi\left(\frac{ tG(x_1,\cdots,x_k,\xi)}{2^j/\lambda}\right)  h(\xi_1,\cdots,\xi_k)d\xi_1\cdots d\xi_k.\label{0bb}\end{align}
Here we take $h$ as a function  $g^{\xi_{k+1},\cdots,\xi_d}(\xi_1,\cdots,\xi_k)=\widehat{g
 }( (R^{-1})^*\xi)$ in the above. Then 
  we can express   $\mathcal{T}_{j}^{\lambda}g(Rx,t )$ in the above as 
  \begin{align}
\mathcal{T}_{j}^{\lambda}g(Rx,t ) &= \lambda^{(d-k)/2} \int_{\mathbb{R}^{d-k}}e^{2\pi i\lambda \langle (x_{k+1},\cdots,x_d),(\xi_{k+1},\cdots,\xi_d)\rangle}\nonumber\\
&\times \left[\mathcal{S}^{(\xi_{k+1},\cdots,\xi_d)}g^{(\xi_{k+1},\cdots,\xi_d)}\right](x_1,\cdots,x_k,t) d\xi_{k+1}\cdots d\xi_d.\label{ihd4}
 \end{align} 
We claim that  
 \begin{align}\label{ihd}
 \sup_{\xi_{k+1},\cdots,\xi_d} \| \mathcal{S}^{(\xi_{k+1},\cdots,\xi_d)}  \|_{L^2(\mathbb{R}^k)\rightarrow L^2(\mathbb{R}^{k+1}) } =\begin{cases} O(1)\ \text{if $2^j\gg \lambda$}\\
 \left(\frac{2^j}{\lambda}\right)^{k}  \lambda^{k/2} \ \text{if $2^j\ll \lambda$.}
 \end{cases}  
 \end{align}
 \begin{proof}[Proof of (\ref{ihd})]
As in (\ref{54ss}),  apply the dilation $(x_1,\cdots, x_k)\rightarrow (2^{j}/\lambda)(x_1,\cdots,x_k)$ and the change of variables  $ (\xi_1,\cdots,\xi_d)\rightarrow 2^j (\xi_1,\cdots,\xi_d)$ in (\ref{0bb}). Define
  \begin{align*} 
 \mathcal{S}_{\rm{dilate}}^{(\xi_{k+1},\cdots,\xi_d)}  h(x_1,\cdots,x_k,t)& :=\left(\frac{2^j}{\lambda}\right)^{k/2} \left[\mathcal{S}^{(2^{j}/\lambda)(\xi_{k+1},\cdots,\xi_d)}h\right]((2^{j}/\lambda)x_1,\cdots,(2^{j}/\lambda)x_k,t) 
 \end{align*}
 which is written as
\begin{align}
\left(\frac{2^j}{\lambda}\right)^{k}  \lambda^{k/2}  \int_{\mathbb{R}^k}
&e^{2\pi i\lambda \left(\frac{2^j}{\lambda}\right)^2  \left[    \langle (x_1,\cdots, x_k), (\xi_1,\cdots,\xi_k)\rangle+  \left(\frac{\lambda}{2^j}\right)tG(x_1,\cdots,x_k,\xi) \right]  }\nonumber\\
&\quad  \times  \chi\left(t\right)\chi\left(  tG(x_1,\cdots,x_k,\xi)\right)h(  \xi_1,\cdots,\xi_k)d\xi_1\cdots d\xi_k.\label{ikd1}
\end{align}
Then 
\begin{align}
\sup_{\xi_{k+1},\cdots,\xi_d} \left\|\mathcal{S}^{(\xi_{k+1},\cdots,\xi_d)}\right\|_{L^2(\mathbb{R}^k)\rightarrow L^2(\mathbb{R}^{k}\times\mathbb{R}) }&=\sup_{\xi_{k+1},\cdots,\xi_d} \left\|\mathcal{S}^{\left(\frac{2^j}{\lambda}\right)(\xi_{k+1},\cdots,\xi_d)}\right\|_{L^2(\mathbb{R}^k)\rightarrow L^2(\mathbb{R}^{k}\times\mathbb{R}) } \nonumber \\
 &=\sup_{\xi_{k+1},\cdots,\xi_d} \left\|\mathcal{S}_{\rm{dilate}}^{(\xi_{k+1},\cdots,\xi_d)}\right\|_{L^2(\mathbb{R}^k)\rightarrow L^2(\mathbb{R}^{k}\times\mathbb{R}) }. \label{ikd3}
\end{align} 
In (\ref{ikd1}), we decompose
\begin{align*} 
\chi\left(  tG(x_1,\cdots,x_k,\xi)\right) =\sum_{m\in \mathbb{Z}^d; \ |\epsilon m|\approx 1}\chi_m(t,x_1,\cdots,x_k,\xi) 
\end{align*}
where $ (R^{-1})^*\xi+(E(x_1,\cdots,x_k),{\bf 0}) $ is localized in a finer portion as
$$\chi_m(t,x_1,\cdots,x_k,\xi):= \psi\left( \frac{t[(R^{-1})^*\xi+(E(x_1,\cdots,x_k),{\bf 0}) ]-\epsilon   m}{\epsilon}\right)\chi\left(  tG(x_1,\cdots,x_k,\xi)\right) $$ supported in 
$$  |t[(R^{-1})^*\xi+E(x_1,\cdots,x_k),{\bf 0}) ]-\epsilon m| \le\epsilon. $$
 Since there are finitely many $m\in \mathbb{Z}^d$ satisfying $|\epsilon m|\approx 1$, it suffices to fix one $m\in\mathbb{Z}^d$.
The phase function of (\ref{ikd1}) is
$$\phi^{(\xi_{k+1},\cdots,\xi_{d})}(x_1,\cdots,x_k,t,\xi_1,\cdots,\xi_k)= \langle (x_1,\cdots, x_k), (\xi_1,\cdots,\xi_k)\rangle+  \left(\frac{\lambda}{2^j}\right)tG(x_1,\cdots,x_k,\xi). $$ 
Let $\lambda\ll 2^j$. Then   by using the fact $ t|(R^{-1})^*\xi+E(x_1,\cdots,x_k),{\bf 0}) |\gtrsim 1$ on the support of (\ref{ikd1})  and  the multilinearity of $\det$,  we can compute that  
$$ \det \left[ [\phi^{(\xi_{k+1},\cdots,\xi_{d})}]''_{(x_1\cdots x_{k})(\xi_1\cdots \xi_k)}(x,t,\xi)\right]= 1+  O\left(\frac{\lambda}{2^j}\right) \approx1. $$
  Thus we are able to apply the H\"{o}rmander Theorem to the operator $\mathcal{S}_{\rm{dilate}}^{(\xi_{k+1},\cdots,\xi_d)}$. For this we need to check the localization $|(\xi_1,\cdots,\xi_k)-(\eta_1,\cdots,\eta_k)|\le \epsilon$     from the factor
  $$\chi_m(t,x_1,\cdots,x_k,\xi_1,\cdots,\xi_k,\xi_{k+1},\cdots,\xi_d)\chi_m(t,x_1,\cdots,x_k,\eta_1,\cdots,\eta_k,\xi_{k+1},\cdots,\xi_d)$$   of the kernel
 $L(\xi_1,\cdots,\xi_k,\eta_1,\cdots,\eta_k)$ of $[\mathcal{S}_{\rm{dilate}}^{(\xi_{k+1},\cdots,\xi_d)}]^*[\mathcal{S}_{\rm{dilate}}^{(\xi_{k+1},\cdots,\xi_d)}]$.
 Therefore
\begin{align*}
\sup_{\xi_{k+1},\cdots,\xi_d} \left\|\mathcal{S}_{\rm{dilate}}^{(\xi_{k+1},\cdots,\xi_d)}\right\|_{L^2(\mathbb{R}^k)\rightarrow L^2(\mathbb{R}^{k}\times\mathbb{R}) }   
&\lesssim\left(\frac{2^j}{\lambda}\right)^{k}  \lambda^{k/2} \left(\left(\frac{2^j}{\lambda}\right)^2\lambda\right)^{-k/2} =1 
\end{align*} 
Let $\lambda\gg 2^j$. Then compute the size of $dx_1\cdots dx_k$ and $d\xi_1\cdots d\xi_k$ to obtain that
\begin{align*}
 \sup_{\xi_{k+1},\cdots,\xi_d} \left\|\mathcal{S}_{\rm{dilate}}^{(\xi_{k+1},\cdots,\xi_d)}\right\|_{L^2(\mathbb{R}^k)\rightarrow L^2(\mathbb{R}^{k}\times\mathbb{R}) }  
&\lesssim\left(\frac{2^j}{\lambda}\right)^{k}  \lambda^{k/2}   
\end{align*} 
The above two bound in (\ref{ikd3}) lead (\ref{ihd}).
  \end{proof}
 In (\ref{ihd4}), we
fix $x_1,\cdots,x_k$ and apply the Plancherel Theorem with respect to $x_{k+1},\cdots,x_d$. Then  we obtain 
$$\int |\mathcal{T}_{j}^{\lambda}g(Rx,t )|^2 dx_{k+1}\cdots dx_d= \int \left[\mathcal{S}^{(\xi_{k+1},\cdots,\xi_d)}g^{(\xi_{k+1},\cdots,\xi_d)}\right]^2(x_1,\cdots,x_k,t) d\xi_{k+1}\cdots d\xi_{d}.  $$
 Integrate both sides above with respect to $dx_1,\cdots dx_kdt$.  Next change the order of integration on the RHS.  Then we apply the operator norm in (\ref{ihd})  to obtain the desired bound $ \|\mathcal{T}^{\lambda}_j \|_{L^2(\mathbb{R}^d)\rightarrow L^2(\mathbb{R}^{d}\times\mathbb{R}) }   $ to finish the proof of (\ref{ihd5}). \end{proof}
Note that Remark \ref{rek1} follows from the  first part of (\ref{ihd5}).

\subsection{Proof of Lemma \ref{prop1}}
We recall the Littlewood-Paley  projection $\mathcal{P}_jf$ in (\ref{10061}), 
\begin{align*}
\mathcal{P}_jf(y,y_{{d+1}})&=\int_{(\eta,\eta_{d+1})\in \mathbb{R}^d\times \mathbb{R}}e^{2\pi i (\eta,\eta_{{d+1}})\cdot (y,y_{{d+1}}) }P_j(\eta+\eta_{d+1}A(y),\eta_{d+1})\widehat{f}(\eta,\eta_{{d+1}})d\eta d\eta_{{d+1}}
\end{align*}
where $P_j(\eta,\eta_{d+1})=\psi\left( \frac{\eta}{2^{j+1}},\frac{\eta_{d+1}}{2^{2(j+1)}}\right)- \psi\left( \frac{\eta}{2^{j}},\frac{\xi_{d+1}}{2^{2j}}\right)$.
To prove Lemma  \ref{prop1}, we shall show 
\begin{align}
\left\| 2^{k/2}\mathcal{T}_{m_{j,k}}\mathcal{P}_{j+k+\ell} \right\|_{ L^2(\mathbb{R}^{d+1})\rightarrow L^2(\mathbb{R}^{d+1}\times \mathbb{R})} &\lesssim 2^{-c|\ell|}  \ \text{for $|\ell|\ge 100d   j$,} \label{045b}
\end{align}
 and
 \begin{align}
\left\|  \mathcal{P}_{k_1} \mathcal{P}_{k_2}^*  \right\|_{ L^2(\mathbb{R}^{d+1})\rightarrow L^2(\mathbb{R}^{d+1})} &\lesssim   2^{-c|k_1-k_2|}.  \label{0444}
\end{align}
\noindent
   \begin{proof}[Proof of (\ref{045b})]
  We apply the dilation (\ref{4bv}) to the function $ \mathcal{P}_{j+\ell+k}f$ and obtain  the identity
\begin{align*} 
D^p_{2^{-k}} \mathcal{P}_{j+\ell+k}f(y,y_{d+1})=\mathcal{P}_{j+\ell}D^p_{2^{-k}} f(y,y_{d+1}).
\end{align*}
Apply $\mathcal{D}^{p}_{2^{-k}}[2^{k/p}  \mathcal{T}_{m_{j,k}}]  =   \mathcal{T}_{m_{j,0}} D^p_{2^{-k}}$ in  (\ref{nka}) first and the above indentity next to obtain  
\begin{align*}
&\mathcal{D}^{p}_{2^{-k}}[2^{k/p}  \mathcal{T}_{m_{j,k}} \mathcal{P}_{j+k+\ell}f] =  [\mathcal{T}_{m_{j,0}} D^p_{2^{-k}} \mathcal{P}_{j+k+\ell}f] 
= \mathcal{T}_{m_{j,0}} \mathcal{P}_{j+\ell}D^p_{2^{-k}} f.
\end{align*}
From this with  the $L^2$ norm invariance of (\ref{nk}),   in order to show  (\ref{045b}),   it suffices  to prove  that 
\begin{align}\label{100645}
\left\|  \mathcal{T}_{m_{j,0}}\mathcal{P}_{j+\ell} f \right\|_{L^2(\mathbb{R}^{d+1}\times \mathbb{R})}&\lesssim 2^{-c|\ell|}   \|f\|_{L^2(\mathbb{R}^{d+1})}\ \text{for  $|\ell|\ge 100 d   j$}.  \end{align}
Recall
\begin{align*} 
m_{j,0}(x,x_{d+1},t,\xi,\xi_{d+1})&=e^{2\pi i t |\xi+\xi_{{d+1}}A(x)|} \chi(t)  \chi\left(\frac{t|\xi+\xi_{{d+1}}A(x)|}{2^j} \right).  \end{align*}
Split $m_{j,0}=a_{j,\ell}+b_{j,\ell}$
\begin{align*} 
a_{j,\ell} (x,x_{d+1},t,\xi,\xi_{d+1})&=e^{2\pi i t |\xi+\xi_{{d+1}}A(x)|} \chi(t)  \chi\left(\frac{t|\xi+\xi_{{d+1}}A(x)|}{2^j} \right)  \psi\left(\frac{|\xi_{d+1}|}{ 2^{|\ell|/10} }\right),\\
b_{j,\ell} (x,x_{d+1},t,\xi,\xi_{d+1})&=e^{2\pi i t |\xi+\xi_{{d+1}}A(x)|} \chi(t)  \chi\left(\frac{t|\xi+\xi_{{d+1}}A(x)|}{2^j} \right) \left(1- \psi\left(\frac{|\xi_{d+1}|}{ 2^{|\ell|/10}  }\right)\right).\end{align*}
We   prove (\ref{100645})  from   the two separate estimates for   $m_{j,0}=b_{j,\ell}\ \text{and}\ a_{j,\ell}$ below.
\end{proof}
\noindent   {\bf Proof of (\ref{100645}) for $m_{j,0}=b_{j,\ell}$}.  We show that
\begin{align} \label{1085}
\left\|  \mathcal{T}_{b_{j,\ell}} f \right\|_{L^2(\mathbb{R}^{d+1}\times I)} \lesssim 2^{-c|\ell|}   \|f\|_{L^2(\mathbb{R}^{d+1})}\ \text{for all $|\ell|\ge 100 d j$.} 
\end{align} 
\begin{proof}[Proof of (\ref{1085})]
The last variable $\xi_{d+1}$ of the above symbol  $b_{j,\ell}$ corresponds to $\lambda$ in $\mathcal{T}^\lambda_j$ of (\ref{5rpp}).  Thus from the support condition $|\lambda|\ge 2^{|\ell|/10}$ and $|\ell|\ge 100dj$, we are able to apply the second case of (\ref{ihd5}) to obtain that
\begin{align*} 
\left\|\mathcal{T}^{\lambda}_j\right\|_{L^2(\mathbb{R}^d)\rightarrow L^2(\mathbb{R}^{d}\times\mathbb{R}) } =    O(2^{jk}/\lambda^{k/2}) = O(2^{-c|\ell|})   
\end{align*}
 which leads (\ref{1085}). 
 \end{proof}
  It suffices to prove  (\ref{100645}) for $m_{j,0}=a_{j,\ell}$.
 \\
 {\bf Proof of (\ref{100645}) for $m_{j,0}=a_{j,\ell}$}. 
It suffices to fix $t=1$ in (\ref{i7}) and prove 
\begin{align} \label{100645p}
\left\|  \mathcal{T}_{a_{j,\ell}}\mathcal{P}_{j+\ell} f(\cdot,\cdot,1) \right\|_{L^2(\mathbb{R}^{d+1})} \lesssim 2^{-c|\ell|}   \|f\|_{L^2(\mathbb{R}^{d+1})}\ \text{for  $|\ell|\ge 100 d j$.} 
\end{align}
Our proof  is based on the argument of  M. Christ in \cite{Ch}, where he facilitated the non-isotropic dilations combined with the cancellation property of the singular kernels  in  the nilpotent Lie groups.
For this purpose,   we obtain  the   kernel representation by using the Fourier inversion formula in the Euclidean space as it follows.
We write
\begin{align*}
\mathcal{T}_{a_{j,\ell}}f(x,x_{d+1},1)&=\int f\left(x-y,x_{d+1}-y_{d+1}-\langle A(x),y\rangle \right)K_{j,\ell}(y,y_{d+1})dydy_{d+1}\\
&=f\cdot_A K_{j,\ell}(x,x_{d+1})
\end{align*}
where
\begin{align}\label{92s}
K_{j,\ell}(x,x_{d+1})=\left(  e^{2\pi i|\cdot|} \chi\left(\frac{|\cdot |}{2^j} \right) \right)^{\vee}(x)2^{|\ell|/10} \psi^{\vee}(2^{|\ell|/10}x_{d+1}). 
 \end{align}
 In view of (\ref{10061}) in Definition \ref{de31},  we write
\begin{align*}
\mathcal{P}_{j+\ell} f (x,x_{d+1})&=\int f\left(x-y,x_{d+1}-y_{d+1}-\langle A(x),y\rangle \right)P_{j+\ell}(y,y_{d+1})dydy_{d+1}\\
&=f\cdot_A P_{j+\ell}(x,x_{d+1})
\end{align*}  where
 \begin{align}\label{10057}
 P_{j+\ell}(x,x_{d+1})&=2^{(j+\ell+1)d}2^{2(j+\ell+1)}  \psi^{\vee}(2^{j+\ell+1}x, 2^{2(j+\ell+1)}x_{d+1}) \nonumber\\
 &-
 2^{(j+\ell)d}2^{2(j+\ell)}  \psi^{\vee}(2^{j+\ell}x, 2^{2(j+\ell)}x_{d+1}).
 \end{align}
 We shall prove (\ref{100645p}) by Lemma \ref{lem599}-\ref{lem602} below.
 \begin{lemma}\label{lem599}
 Suppose that the two operators $\mathcal{T}_{a_{j,\ell}}$ and $\mathcal{P}_{j+\ell} $ have the integral kernels $K_{j,\ell}$ and $P_{j+\ell}$ in (\ref{92s}) and (\ref{10057}).   
Then, these two kernels satisfy the cancellation property:
\begin{align}\label{10055}
\int K_{j,\ell}(x,x_{d+1})dxdx_{d+1}=\int P_{j+\ell}(x,x_{d+1})dxdx_{d+1}=0.
\end{align}
Let  $\{e(\theta_k)\}$  be a set of equally distributed vectors in $S^{d-1}$ and let $e^{\perp}(\theta_k)$'s be the $(d-1)$ different  unit vectors perpendicular to each  $e(\theta_k) $.  Then it holds that 
    \begin{align}\label{10056}
 | K_{j,\ell}(x,x_{d+1})|\lesssim &\sum_{k=1  }^{2^{j(d-1)/2}} \frac{2^j }{(|(x-e(\theta_k)) \cdot e(\theta_k) 2^{j}|+1)^N}  \frac{2^{j/2}}{(|(x  \cdot e^{\perp}(\theta_k) 2^{j/2}|+1)^N} \nonumber \\
 &\times 2^{|\ell|/10}\psi^{\vee}(2^{|\ell|/10}x_{d+1}) \end{align}
  having   its $L^1$ norm $O(2^{j(d-1)/2})$.
  The support of $P_{j+\ell}$ in (\ref{10057}) is concentrated on
  \begin{align}\label{10058}
  \text{ $\{(y,y_{d+1}): |(2^{j+\ell}y,2^{2(j+\ell)}y_{d+1})|\lesssim1\}$  with its
  $L^1$ norm $O(1)$.   }
  \end{align}
  \end{lemma}
    \begin{proof}
The cancellation  in (\ref{10055}) follows from $\widehat{K_j}(0)=\widehat{P_{j+\ell}}(0)=0$ where  $\ \widehat{}\ $ indicates the Euclidean Fourier transform.  We derive the inequality  (\ref{10056}) by the decomposing of the frequency variables $\xi$ along angular sectors with the angle
  width $2^{-j/2}$ and gaining the decay along the sectors. Finally, we can verify (\ref{10058}) in view of (\ref{10057}).
  \end{proof}
  \begin{lemma}\label{lem600}
   Let the two bilinear operations $(f,g)\rightarrow f\cdot_A g$ and $(f,g)\rightarrow f*_Ag$ be defined as in (\ref{0ch}). Then
\begin{align}\label{0cc}
(f\cdot_A P)\cdot_A K=f\cdot_A( P*_{-A^T}K).
\end{align}
   Moreover,
our  composition operator is exprssed as
\begin{align}\label{kkgg}
\mathcal{T}_{a_{j,0}}\mathcal{P}_{j+\ell}f(\cdot,1) =
f\cdot_A(P_{j+\ell}*_{-A^T}K_{j,\ell})(\cdot).
\end{align}
This can be represented as  $$ \mathcal{T}_{a_{j,0}}\mathcal{P}_{j+\ell} f(x,x_{d+1},1)=\int f\left(x-y,x_{d+1}-y_{d+1}-\langle A(x),y\rangle \right)U_{j,\ell}(y,y_{d+1})dydy_{d+1}$$ 
where $U_{j,\ell}(y,y_{d+1})=P_{j+\ell}*_{-A^T}K_{j,\ell}(y,y_{d+1})$ is written as
\begin{align}\label{sspp}
& \int P_{j+\ell}\left(y-z,y_{d+1}-z_{d+1}+\langle  A^T (y-z),z\rangle \right)K_{j,\ell}(z,z_{d+1})dzdz_{d+1}\\
 &= \int K_{j,\ell}\left(y-z,y_{d+1}-z_{d+1}+\langle  A (y-z),z\rangle \right)P_{j+\ell}(z,z_{d+1})dzdz_{d+1} 
\nonumber
\end{align}
\end{lemma}
 
  \begin{proof}
From   $f\cdot_A g =\mathcal{U}_{-A^T}(\mathcal{U}_{A^T}f*_{-A^T}g)$ in   (\ref{0ch}),
\begin{align*}
(f\cdot_A P)\cdot_A K&=\mathcal{U}_{-A^T}[\mathcal{U}_{A^T}(f\cdot_A P)*_{-A^T}K]\\
&=\mathcal{U}_{-A^T}[\mathcal{U}_{A^T}[ \mathcal{U}_{-A^T}(\mathcal{U}_{A^T}f*_{-A^T} P)]  *_{-A^T}K]\\
&=\mathcal{U}_{-A^T}[ (\mathcal{U}_{A^T}f*_{-A^T} P)  *_{-A^T}K]= \mathcal{U}_{-A^T}[ \mathcal{U}_{A^T}f*_{-A^T} P  *_{-A^T}K]\\
& = \mathcal{U}_{-A^T}[ \mathcal{U}_{A^T}f*_{-A^T} (P  *_{-A^T}K)]=f\cdot_A (P  *_{-A^T}K)
\end{align*}
This implies (\ref{0cc}). Thus from (\ref{92s}) and (\ref{10057}) combined with (\ref{0cc}), it follows that
$$\mathcal{T}_{a_{j,0}}\mathcal{P}_{j+\ell}f(\cdot,1)=f\cdot_A P_{j+\ell}\cdot_A K_{j,\ell}(\cdot)=
f\cdot_A(P_{j+\ell}*_{-A^T}K_{j,\ell})(\cdot).$$
showing (\ref{kkgg}).
By (\ref{kkgg}) with $\cdot_A$ and $*_A$ in (\ref{0ch}), we obtain the first line of the expression (\ref{sspp}). The second lines follows from the change of variables. 
  \end{proof}
\begin{lemma}\label{lem602}
Let $|\ell| \ge 100d  j$.  Suppose that $U_{j,\ell}$ is defined in (\ref{sspp}). Then   \begin{align}\label{100301}
\int |U_{j,\ell}(y,y_{d+1})|dy dy_{d+1}\lesssim 2^{-c|\ell|}\ \text{for some $c>0$.}
\end{align}
\end{lemma}
\begin{proof}
{\bf Case 1}. Let $\ell>0$. Use $\int P_{j+\ell}(z,z_{d+1})dzdz_{d+1}=0$ and  the first line of (\ref{sspp}),
\begin{align*}
&U_{j,\ell}(y,y_{d+1})\\
&= \int \bigg[K_{j,\ell}\left(y-z,y_{d+1}-z_{d+1}+\langle  A (y-z),z\rangle \right)-K_{j,\ell}\left(y,y_{d+1}  \right)\bigg] P_{j+\ell}(z,z_{d+1})dzdz_{d+1}.   
\end{align*}
Note that   $K_{j,\ell}$ is supported essentially on  $S^{d-1}\times [-2^{-|\ell|/10},2^{-|\ell|/10}]$ in (\ref{10056}). From this we observe that
\begin{itemize}
\item  $|\nabla_y K_{j,\ell}(y,y_{d+1})|\lesssim 2^j\times |K_{j,\ell}(y,y_{d+1})|$  because of (\ref{92s})  
\item $|\nabla_{y_{d+1}} K_{j,\ell}(y,y_{d+1})|\lesssim  2^{|\ell|/10}\times |K_{j,\ell}(y,y_{d+1})|$  because of (\ref{92s})     
\item $|z|, |\langle  A(y-z),z\rangle|\lesssim 2^{-j-\ell} $ and $ |z_{d+1}|   \le 2^{-2j-2|\ell|}$ because of (\ref{10057}) and (\ref{10056}). 
\end{itemize}
By utilizing this condition with $y$ dominating over $z$,   we apply the mean value theorem   for   the function $K_{j,\ell}$  to obtain that
\begin{align}
 \int |U_{j,\ell}(y,y_{d+1})|dy dy_{d+1}&\lesssim \int  |2^{-j-\ell}2^j+2^{-j-\ell} 2^{|\ell|/10}|\times |\tilde{K}_{j,\ell}(y,y_{d+1})| dydy_{d+1}\nonumber\\
&\times\int |P_{j+\ell}(z,z_{d+1})|dzdz_{d+1}=O(2^{(d-1)j/2}2^{-8|\ell|/10})=O(2^{-|\ell|/2}).\label{100302}
\end{align}
where $\tilde{K}_{j,\ell}$ is a slight change of $K_{j,\ell}$ in  (\ref{92s}).\\
{\bf Case 2}. Let $\ell<0$. By subtracting from $U_{j,\ell}(y,y_{d+1})$,   the vanishing term
$$ \int K_{j,\ell}\left(z,z_{d+1} \right)  P_{j+\ell}\left(y,y_{d+1}  \right) dzdz_{d+1} =0,$$
we rewrite $U_{j,\ell}(y,y_{d+1})$ in the second line of (\ref{sspp}) as
 $$ \int K_{j,\ell}\left(z,z_{d+1} \right) \bigg[ P_{j+\ell}\left(y-z,y_{d+1}-z_{d+1}+  \langle A^T(y-z),z\rangle\right)- P_{j+\ell}\left(y,y_{d+1}  \right) \bigg]  dzdz_{d+1}.$$
Note that for $\ell<0$ in (\ref{10057}) and (\ref{10056}),
\begin{itemize}
\item  $|\nabla_y P_{j+\ell}(y,y_{d+1})|\lesssim 2^{j-|\ell|}\times |P_{j+\ell}(y,y_{d+1})|$
\item
  $|\nabla_{y_{d+1}} P_{j+\ell}(y,y_{d+1})|\lesssim 2^{2j-2|\ell|}\times |P_{j+\ell}(y,y_{d+1})|$, 
  \item $|z|\lesssim 1$,$|z_{d+1}|\lesssim 2^{- |\ell|/10}$,$|y-z|\lesssim 2^{-j+\ell}$ and $|\langle  A^T(y-z),z\rangle|\lesssim |y-z|\,|z|\lesssim 2^{-j+|\ell|}$.
\end{itemize}
We apply  the mean value theorem to obtain
\begin{align*}
&\bigg[P_{j+\ell}\left(y-z,y_{d+1}-z_{d+1}+\langle  A(z),y-z
\rangle \right)-P_{j+\ell}\left(y,y_{d+1} \right) \bigg]\\
&\qquad\qquad\lesssim \max\{ 2^{j-|\ell|},  2^{2j-2|\ell|} (2^{-j+|\ell|}+2^{-|\ell|/10})     \} \times P_{j+\ell}\left(y,y_{d+1} \right).    \end{align*}
Thus
\begin{align*}
  \int |U_{j,\ell}(y,y_{d+1})|dy dy_{d+1}& \lesssim  2^{j-|\ell|}\int  |P_{j+\ell}\left(y,y_{d+1} \right)|dydy_{d+1} \int |K_{j,\ell}(z,z_{d+1})| dzdz_{d+1}\\
  &\lesssim  2^{(d-1)\,j/2}2^j 2^{-|\ell|}=O(2^{-|\ell|/2}).
  \end{align*}
By this and (\ref{100302}) with $|\ell|>100d j$, we obtain (\ref{100301}).
\end{proof}
\noindent  
{\bf Almost orthogonality}.\ \ 
There remains  (\ref{0444}) in Lemma \ref{prop1}.
We claim that
\begin{align*} 
\|\mathcal{P}_{j}\mathcal{P}_{j+\ell}^*\|_{op}\lesssim 2^{-c|\ell|}.
\end{align*}
\begin{proof}
Notice $\mathcal{P}_{j+\ell}^*f=f*_AP_{j+\ell}$ and $(\mathcal{U}_AF)*_{-A^T}(\mathcal{U}_AG)=\mathcal{U}_A(F*_AG)$. Then
\begin{align*}
\mathcal{P}_{j}\mathcal{P}_{j+\ell}^*f&=(f*_AP_{j+\ell})\cdot_AP_j =\mathcal{U}_{-A}\left( [\mathcal{U}_A(f*_AP_{j+\ell})] *_{-A^T}P_j\right)\\
&=\mathcal{U}_{-A}   \mathcal{U}_A( f*_AP_{j+\ell}*_A \mathcal{U}_{-A}P_j )    = f*_AP_{j+\ell}*_A \mathcal{U}_{-A}P_j.
\end{align*}
Thus
$\mathcal{P}_{j}\mathcal{P}_{j+\ell}^*f(x,x_{d+1})=\int f(x-y,x_{d+1}-y_{d+1}-\langle A(x-y),y\rangle) U_{j,\ell}(y,y_{d+1})dydy_{d+1}$,
where  
\begin{align*}
&U_{j,\ell}(y,y_{d+1}) =P_{j+\ell}*_A \mathcal{U}_{-A}P_j(y,y_{d+1})= \mathcal{U}_{-A}P_j*_{A^T}P_{j+\ell}(y,y_{d+1})\\
&=
\int  P_j\left(y-z,y_{d+1}-z_{d+1}-\langle  A^T (y-z),z\rangle-\langle  (y-z),A(y-z)\rangle\right)P_{j+\ell}(z,z_{d+1})dzdz_{d+1}\\
&= \int  P_j\left(y-z,y_{d+1}-z_{d+1} -\langle  (y-z),A(y)\rangle\right)P_{j+\ell}(z,z_{d+1})dzdz_{d+1} . 
\end{align*}
It suffices to consider the case $\ell>0$. Using $\int P_{j+\ell}(z,z_{d+1})dzdz_{d+1}=0$, we
write $U_{j,\ell}(y,y_{d+1})$ in the above as $$ \int \bigg[P_j\left(y-z,y_{d+1}-z_{d+1}+\langle   (y-z),A(z)\rangle\right)-P_j(y,y_{d+1}) \bigg] P_{j+\ell}(z,z_{d+1})dzdz_{d+1}.   $$
Notice from (\ref{10058}), on the region
$
 \{(z,z_{d+1}): |(2^{j+\ell}z,2^{2(j+\ell)}z_{d+1})|\lesssim1\}   
$, it holds that
\begin{itemize}
\item  $|\nabla_y P_{j}(y,y_{d+1})|\lesssim 2^{j}\times |P_{j}(y,y_{d+1})|$ and 
  $|\nabla_{y_{d+1}} P_{j}(y,y_{d+1})|\lesssim 2^{2j }\times |P_{j+\ell}(y,y_{d+1})|$, 
  \item $|z|\lesssim 2^{-j-\ell}$,$|z_{d+1}|\lesssim  2^{-2j-2\ell}$, $|\langle  y-z, A(z)\rangle|\lesssim |z||y-z|\lesssim 2^{-j-\ell}2^{-j}$  
\end{itemize}
Using the mean value theorem as above, we obtain  $\int |U_{j,\ell}(y,y_{d+1})|dydy_{d+1}\lesssim 2^{-\ell}$.
 \end{proof}

\end{document}